\pdfoutput=1
\DeclareSymbolFont{AMSb}{U}{msb}{m}{n}
\documentclass[reqno, a4paper, noamsfonts, 11pt]{amsart}
\usepackage[english]{babel}
\usepackage[T1]{fontenc}
\usepackage[utf8]{inputenc}
\usepackage{enumerate,amssymb,dsfont,mathtools}
\usepackage{lmodern,microtype}
\usepackage{geometry}
\usepackage[dvipsnames]{xcolor}
\usepackage{xfrac}
\usepackage[colorlinks=true, pdfstartview=FitV, linkcolor=blue, citecolor=blue, urlcolor=blue,pagebackref=false]{hyperref}

\DeclareRobustCommand{\SkipTocEntry}[5]{} 

\newcommand{\N}{\mathbb{N}}
\newcommand{\Z}{\mathbb{Z}}
\newcommand{\R}{\mathbb{R}}
\newcommand{\del}{\partial}
\newcommand{\dsOne}{\mathds{1}}
\newcommand{\cC}{\ensuremath{\mathcal{C}}}
\newcommand{\cE}{\ensuremath{\mathcal{E}}}
\newcommand{\cG}{\ensuremath{\mathcal{G}}}
\newcommand{\cH}{\ensuremath{\mathcal{H}}}
\newcommand{\cI}{\ensuremath{\mathcal{I}}}
\newcommand{\cL}{\ensuremath{\mathcal{L}}}
\newcommand{\cM}{\ensuremath{\mathcal{M}}}

\newcommand{\cR}{\ensuremath{\mathcal{R}}}
\newcommand{\cS}{\ensuremath{\mathcal{S}}}
\newcommand{\cT}{\ensuremath{\mathcal{T}}}
\DeclareMathAlphabet{\mathup}{OT1}{\familydefault}{m}{n}

\newcommand{\dx}[1]{\mathop{}\!\mathup{d} #1}
\newcommand{\pderiv}[3][]{\frac{\mathop{}\!\mathup{d}^{#1} #2}{\mathop{}\!\mathup{d} #3^{#1}}}
\DeclarePairedDelimiter{\abs}{\lvert}{\rvert}
\DeclarePairedDelimiter{\norm}{\lVert}{\rVert}
\DeclarePairedDelimiter{\bra}{(}{)}

\DeclarePairedDelimiter{\set}{\{}{\}}
\DeclarePairedDelimiter{\skp}{\langle}{\rangle}
\newcommand{\PI}{\mathup{PI}}
\newcommand{\gel}{\mathup{gel}}
\newcommand{\loc}{\mathup{loc}}
\newcommand{\eps}{\varepsilon}
\DeclareMathOperator{\Ent}{Ent}
\DeclareMathOperator{\LSI}{LSI}
\newtheorem{thm}{Theorem}[section]
\newtheorem{prop}[thm]{Proposition}
\newtheorem{lem}[thm]{Lemma}
\newtheorem{cor}[thm]{Corollary}
\newtheorem{defn}[thm]{Definition}
\newtheorem{rem}[thm]{Remark}

\title[Exchange-driven growth with product kernel]{Self-similar behavior of the exchange-driven growth model with product kernel}
\author{Constantin Eichenberg \and André Schlichting}
\address[Constantin Eichenberg]{Institut für Angewandte Mathematik, Universität Bonn} 
\email{eichenberg@iam.uni-bonn.de}
\address[André Schlichting]{Institut für Analysis und Numerik, Universität Münster}
\email{a.schlichting@uni-muenster.de}

\numberwithin{equation}{section}
\begin{document}

\begin{abstract}
We study the self-similar behavior of the exchange-driven growth model, which describes a process in which pairs of clusters, consisting of an integer number of monomers, interact through the exchange of a single monomer. The rate of exchange is given by an interaction kernel $K(k,l)$ which depends on the sizes $k$ and $l$ of the two interacting clusters and is assumed to be of product form  $(k\,l)^\lambda$ for $\lambda \in [0,2)$. We rigorously establish the coarsening rates and convergence to the self-similar profile found by Ben-Naim and Krapivsky~\cite{BenNaim2003}. For the explicit kernel, the evolution is linked to a discrete weighted heat equation on the positive integers by a nonlinear time-change.  For this equation, we establish a new weighted Nash inequality that yields scaling-invariant decay and continuity estimates. Together with a replacement identity that links the discrete operator to its continuous analog, we derive a discrete-to-continuum scaling limit for the weighted heat equation. Reverting the time-change under the use of additional moment estimates, the analysis of the linear equation yields coarsening rates and self-similar convergence of the exchange-driven growth model.
\end{abstract}


\maketitle

\vspace{-\baselineskip}

\tableofcontents

\section{Introduction and main results}
\subsection{The exchange-driven growth model}\label{Ss.EDG}
The exchange-driven growth model describes a broad class of physical processes in which pairs of clusters consisting of an integer number of monomers can grow or shrink only by exchanging a single monomer~\cite{BenNaim2003}. 
The physical motivation behind the growth processes based on this exchange mechanism is quite different from classical aggregation models like the Smoluchowski coagulation equation~\cite{Smoluchowski1916,Friedlander2000}, which explains its recent interest. Moreover, the underlying exchange mechanism is not restricted to physical models but can be applied to social phenomena like migration~\cite{Ke2002},
population dynamics~\cite{Leyvraz2002}, and wealth exchange~\cite{Ispolatov1998}. It is also found in diverse phenomena at contrasting scales from microscopic level polymerization processes~\cite{DoiEdwards1988}, to cloud~\cite{Hidy1972} and galaxy formation mechanisms at massive scales, as well as in statistical physics~\cite{Krapivsky2010}.
Although this process is not necessarily realized by chemical kinematics, it is convenient to be interpreted as a reaction network of the form
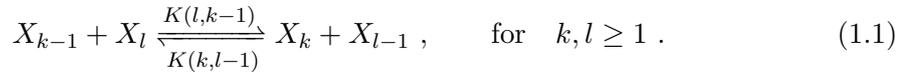
\begin{equation}\label{e:EG:ChemReact}
	X_{k-1} + X_{l} \xrightleftharpoons[K(k,l-1)]{K(l,k-1)} X_{k} + X_{l-1}\ , \qquad \text{for}\quad k, l \geq 1 \ . 
\end{equation}
The clusters of size $k\geq 1$ are denoted by $X_k$. Additionally, the variable~$X_0$ represents empty volume. The kernel  $K(k,l-1)$ encodes the rate of the exchange of a single monomer from a cluster of size $k$ to a cluster of size $l-1$. Here and in the following the notation $k\geq 1$ means $k\in \N=\set{1,2,\dots}$ and $l\geq 0$ denotes $l\in \N_0 = \N \cup\set{0}$.
The concentrations of $X_k$ in~\eqref{e:EG:ChemReact} are denoted by $(c_k)_{k\geq 0}$ and satisfy for $k\geq 0$ the reaction rate equation formally obtained from~\eqref{e:EG:ChemReact} by mass-action kinetics
\begin{equation}\label{EDG} \tag{EDG}
	\begin{aligned}
		\dot c_k \ \ = \ \ \ &\sum_{l\geq 1} K(l,k-1)c_l c_{k-1} - \sum_{l\geq 1} K(k,l-1) c_{k} c_{l-1}  \\
		- &\sum_{l\geq 1} K(l,k)c_l c_k + \sum_{l\geq 1} K(k+1,l-1)c_{k+1} c_{l-1} \ , \qquad\text{for } k\geq 0\ .
	\end{aligned}
\end{equation}
It is easy to see that, at least formally, the quantities 
\begin{align*}
	\sum_{k\geq 0} c_k \qquad\text{and}\qquad \sum_{k\geq 1} k \, c_k
\end{align*}
are conserved during the evolution. The first sum can be interpreted as the total volume of the system and the second one as mass (or density, if normalized). Without loss of generality by a suitable time-change, the volume is normalized to $1$, so throughout we assume that
\begin{align}\label{e:EDG:conservation}
	\sum_{k\geq 0} c_k = 1\qquad\text{and} \qquad \sum_{k\geq 1} k \, c_k = \rho \in [0,\infty).
\end{align}
With this normalization, equation \eqref{EDG} can also be viewed as mean-field limit of an interacting stochastic particle system, where $N$ particles on a complete graph of size~$L$ move between sites according to a jump process in which a jump between a site with $k$ particles and a site with $l$ particles occurs with rate $K(k,l)/(N-1)$. Then the statistical description of the population of clusters in the limit $N,L\to\infty$ such that $\sfrac{N}{L} \to \rho$ is given by equation~\eqref{EDG}, where $c_k$ is the fraction of sites with $k$ particles. 
This coarse-grained limit was rigorously derived in~\cite{JG2017}.  

The mathematical theory of the equation \eqref{EDG} itself started with well-posedness results for kernels with at most linear growth in~\cite{emre,Schlichting2019}, that is $K(k,l)\leq C \, k \,l$. In addition, for nearly symmetric kernels satisfying $K(k,l) = K(l,k)$, the global well-posedness can be extended to kernels satisfying $K(k,l) \leq C \bra*{ k^\mu l^\nu + k^\nu l^\mu}$ with $\mu,\nu \in [0,2]$ and $\mu+\nu\leq 3$ (cf.~\cite[Theorem 2]{emre}). 

The long-time behavior of solutions is investigated in~\cite{Schlichting2019,EsenturkVelazquez2019}, where the crucial assumption on the kernels is a detailed balance condition or some suitable monotonicity properties. For kernels satisfying the detailed balance condition, equation~\eqref{EDG} has many striking similarities to the Becker-D\"oring equation~\cite{BD35,BCP86}. In particular, there exist unique equilibrium states $\omega^\rho$ with density $\rho$ up to a critical value~$\rho_c$, and a solution $c$ with density $\rho$ converges $c(t) \to \omega^{\min(\rho,\rho_c)}$ as $t \to \infty$, where the convergence is strong if $\rho \leq \rho_c$ and weak if $\rho > \rho_c$. In the latter case, the bulk of the system relaxes to $\omega^{\rho_c}$ while the excess density $(\rho-\rho_c)$ condensates in larger and fewer clusters, which is analogous to the classical LSW~\cite{Lifshitz1961,Wagner1961} coarsening picture treated in~\cite{Niethammer2003,Schlichting2018}. In light of these results, it is natural to ask whether the excess density in~\eqref{EDG} coarsens in a self-similar way.

It is worth mentioning that condensation and self-similar behavior is already present on the level of stochastic particle systems.
In zero-range processes~\cite{Grosskinsky2003,Jatuviriyapornchai2016,Godreche2017} and explosive condensation models~\cite{Waclaw2012,Chau2015,Evans2015} the coarsening happens with rates satisfying the detailed balance condition and in particular $K(k,0)>0$ for all $k\in \N$. The attractive interaction between particles causes condensation in those models. Although the zero-range process's kernel is bounded, coarsening and convergence to a self-similar profile is formally described in~\cite{Godreche2017} in the mean-field case.
A first rigorous result beyond the mean-field situation is obtained in \cite{Beltran2017}. They derive an effective process of the multi-condensate phase in the zero-range process on a finite graph with diverging particle density.
Contrary to explosive models with unbounded kernels, it is possible to observe even instantaneous gelation within suitable limits. For inclusion processes, one often studies the case $K(k,0)\to 0$ for all $k\in \N$ in the limit of infinite volume or particle density~\cite{GrosskinskyRedigVafayi2013,Cao2014,JCG2019} so that the microscopic dynamics are irreducible and non-degenerate. However, the limiting coarsening mechanism is driven by the absorbing boundary with $K(k,0)=0$ for all $k\in \N$, which is also the case for the investigated product kernels of this work.

In this work, we provide rigorous results about the coarsening and self-similar behavior of solutions to~\eqref{EDG} with the specific family of product kernels
\begin{align}\label{e:def:K_alambda}
K(k,l) = K_\lambda(k,l) = a_{\lambda}(k)a_{\lambda}(l) \qquad\text{with}\qquad a_{\lambda}(k)=  \begin{cases}
k^\lambda, & \lambda > 0, \\
1-\delta_{k,0}, & \lambda = 0,
\end{cases} 
\end{align}
for all $\lambda \in \bigl[0,2)$. These and more general symmetric homogeneous kernels were introduced and investigated in~\cite{BenNaim2003}. A crucial property is that $K(k,0) = 0$, which on the level of clusters means that a cluster with no particles cannot regain particles and hence is virtually removed from the system. In particular this violates the aforementioned detailed balance condition. It is easy to see that the only equilibrium is the vacuum state $c_0 = 1$, $c_k = 0, \ k \geq 1$. During the evolution, particles distribute among fewer and larger clusters over time while smaller clusters die out. This means that the driving coarsening mechanism in this case is the loss of volume, in contrast to the detailed balance case, where coarsening is induced by attraction between particles and only affects the excess density. 

The symmetry and product form of $K$ simplify the system~\eqref{EDG} considerably. We introduce the moments for some $\kappa\in [0,\infty)$ by
\begin{equation}\label{e:def:moment}
 M_\kappa = M_\kappa[c] = \sum_{l \geq 1} l^\kappa c_l . 
\end{equation}
Note that we exclude $k=0$ in the summation, so $M_0[c] =  1 - c_0$ is not conserved and decreases over time. With this definition, the system \eqref{EDG} becomes
\begin{align} \label{e:EDG:lambda}\tag{EDG$_\lambda$}
\begin{cases} \dot{c}_0 = M_\lambda[c] \;c_1, \\
\dot{c}_1 = M_\lambda[c]\, \bra[\big]{-2 c_1 + 2^\lambda c_2}, \\
\dot{c}_k = M_\lambda[c] \; \bigl( (k-1)^\lambda c_{k-1} - 2k^\lambda c_k + (k+1)^\lambda c_{k+1}\bigr), \quad k \geq 2 .
\end{cases}
\end{align}
The first question regarding to coarsening is the large-time behavior of the average cluster size among living clusters, which plays the role of the characteristic length-scale. Intuitively, this quantity should grow in time. Indeed, by conservation of mass, the average cluster size, denoted $\ell(t)$, is given by
\begin{align}\label{e:average_cluster_size}
\ell(t) = \frac{1}{1-c_0(t)} \sum_{k=1}^\infty k\, c_k(t) = \frac{\rho}{M_0[c]},
\end{align}
hence the length-scale of the system is inversely proportional to the volume of living particles, which decreases by equation \eqref{e:EDG:lambda}. More specifically, the scaling analysis in~\cite{BenNaim2003} predicts that 
\begin{equation} \label{e:def:ell}
\ell(t) \propto \begin{cases}
t^\beta, &\text{if } 0 \leq \lambda < \sfrac{3}{2}, \\
\exp(Ct), &\text{if } \lambda = \sfrac{3}{2}, \\ 
(t_{\gel} - t)^{\beta}, &\text{if } \sfrac{3}{2} < \lambda < 2,
\end{cases} 
\qquad\text{with } \beta = (3-2\lambda)^{-1},
\end{equation}
and the involved constants $C$, $t_{\gel}$ and the one in $\propto$ depend on the initial data.
Hence, coarsening is expected on an algebraic timescale for $0 \leq \lambda < \sfrac{3}{2}$, transitioning into a gelation regime for $\sfrac{3}{2} < \lambda < 2$, where the solution only exists up to the gelation time $t_{\gel} < \infty$ at which all the mass vanishes to infinity. At the transition $\lambda = \sfrac{3}{2}$ solutions exist globally and we expect coarsening on an exponential timescale with a non-universal rate $C$. Our first result confirms these coarsening rates.
\begin{thm}[Coarsening rates] \label{thm_1}
Let $0 \leq \lambda < 2$ and set $\beta = (3-2\lambda)^{-1}$. Then the following statements hold, with all constants only depending on $\lambda, \rho$ and moments of the initial data up to order $\lambda$:
\begin{enumerate}
\item If $0 \leq \lambda < \sfrac{3}{2}$, then every solution $c$ to equation \eqref{e:EDG:lambda} exists globally and there are positive constants $C_1,C_2,t_0$ such that
\begin{align*}
C_1 t^{\beta} \leq  \ell(t) \leq C_2 t^{\beta} \qquad\text{for all } t \geq t_0 .
\end{align*}
\item Let $\lambda = \sfrac{3}{2}$, then every solution $c$ to equation \eqref{e:EDG:lambda} exists globally and there are positive constants $C_1,C_2,K_1,K_2,t_0$ such that
\begin{align*}
K_1 \exp(C_1t) \leq \ell(t) \leq K_2 \exp(C_2 t) \qquad\text{for all } t \geq t_0 .
\end{align*}
\item\label{thm_1:3} If $\sfrac{3}{2} < \lambda \leq 2$, then every solution $c$ to equation \eqref{e:EDG:lambda} exists only locally on a maximal interval $[0,t^*)$ for some $t^*>0$ and there are positive constants $C_1,C_2,t_0$ such that
\begin{align*}
C_1 (t^*-t)^{\beta} \leq \ell(t) \leq C_2 (t^*-t)^{\beta}  \qquad\text{for all } t_0 \leq t < t^*.
\end{align*}
\end{enumerate}
\end{thm}
\begin{rem}
In the case $\sfrac{3}{2} < \lambda \leq 2$ it is easy to see from the proof (see Proposition~\ref{prop:tau_estimates}) that the blow up time $t^*$ goes to zero as the $\lambda$-th moment of the initial data diverges.
\end{rem}
The next question is whether solutions become self-similar as $t \to \infty$, which is formally addressed in~\cite{BenNaim2003}. The crucial observation is, that \eqref{e:EDG:lambda} becomes a discrete linear weighted heat equation after a suitable non-autonomous time-change. Considering the corresponding weighted heat equation on the continuum scale (see Section~\ref{Ss.IntroHeat}) and formal scaling argument, the calculations in~\cite{BenNaim2003} suggest that any solution $c$ with mass $M_1[c]=\rho$ is asymptotically self-similar to a profile $g_\lambda: [0,\infty) \to (0,\infty)$ for a suitable scaling function $s(t) \propto \ell(t)$ of the form~\eqref{e:def:ell}. In mathematical terms, we expect that that following relation holds
\[ 
  c_k(t) \propto \rho \, s(t)^{-2}g_\lambda\bra[\big]{s(t)^{-1}k}  \qquad \text{for }t\gg 1 .
\]
Hereby, the profile $g_\lambda$ is explicitly given by
\begin{align} \label{profile_u}
g_\lambda(x) = \frac{1}{Z_\lambda} \frac{x^{1-\lambda}}{2-\lambda} \exp\bra*{-\frac{x^{2-\lambda}}{(2-\lambda)^2}},
\end{align}
where $Z_\lambda$ is a normalization constant such that $\int_{[0,\infty)} x \, g_\lambda(x) \dx{x} = 1$ and given by
\begin{equation}\label{e:def:Zlambda}
  Z_\lambda = (2-\lambda)^{\frac{2}{2-\lambda}}\, \Gamma\bra*{ 1 + \frac{1}{2-\lambda}} . 
\end{equation}
The appropriate object for the rigorous analysis of self-similarity is the empirical measure associated to a solution $c$ given by 
\begin{align} \label{empirical_measure_c}
\mu_c(t) = s(t) \sum_{k\geq 1} c_k(t)\delta_{s(t)^{-1}k}.
\end{align}
The normalization in~\eqref{empirical_measure_c} is chosen, such that
\begin{align}
M_0[c] = 1-c_0 = s^{-1}(t) \int_0^\infty \dx\mu_c\qquad\text{and}\qquad M_1[c] = \int_0^\infty x \dx\mu_c. \label{e:mu:M0}
\end{align}
Self-similar behavior of $c$ for $t \to \infty$ then corresponds to the existence of the limit $\mu_c(t) \to \rho \, g_\lambda$ in a suitable topology, which we define now. From here on, we use the notation $\R_+ = (0,\infty)$ and $\overline{\R}_+ = [0,\infty)$. In addition to the weak convergence of measures in $\cM(\overline{\R}_+)$, written as $\mu_n \rightharpoonup \mu$, and defined as 
\[
   \lim_{n\to\infty} \int_{\overline\R_+} f(x) \dx{\mu_n(x)} = \int_{\overline\R_+} f(x) \dx{\mu(x)} \qquad\text{for all } f\in C^0_b(\overline\R_+), 
\]
we need two further weak convergence concepts adjusted to the problem setting.
\begin{defn} \label{defn:weak_convergence}
Let $\mathcal{M}_1(\overline{\R}_+)$ be the space of all Borel measures on $\overline{\R}_+$ with finite first moment, that is $\int_0^\infty x \dx{\mu}<\infty$ for all $\mu\in \cM_1(\overline{\R}_+)$.

The space of continuous sublinear growing functions $\mathcal{C}$ and its subspace $\mathcal{C}_0$ with those vanishing at~$0$ are defined by
\begin{align*}
\mathcal{C} = \set*{ f \in C^0(\overline{\R}_+) \colon \lim_{x \to \infty} x^{-1} f(x) = 0 } \qquad \text{and}\qquad \mathcal{C}_0 = \set*{f \in \mathcal{C}\colon f(0) = 0 }.
\end{align*}
A family of measures $\mu_n \in \mathcal{M}_1(\overline{\R}_+)$ converges weakly to $\mu \in \mathcal{M}_1(\overline{\R}_+)$ with respect to $\mathcal{C}$, denoted by $\mu_n \rightharpoonup \mu$, if
\begin{align*}
\lim_{n\to \infty}\int_{\overline{\R}_+} f(x) \dx\mu_n(x) = \int_{\overline{\R}_+} f(x) \dx\mu(x) \qquad \text{for all }  f\in \cC .
\end{align*}
Likewise, $\mu_n \rightharpoonup \mu$ with respect to $\mathcal{C}_0$, if the above limit holds for all $f\in \cC_0$.
\end{defn}
With this definition, we prove weak convergence to the self-similar profile with an explicit scaling function except at the transition $\lambda = \sfrac{3}{2}$, where the scaling function can be described asymptotically. Furthermore, the weak convergence is with respect to $\mathcal{C}$ for $\lambda \in [0,1)$ and with respect to $\mathcal{C}_0$ if $ \lambda \geq 1$ due to technical reason, see Remark~\ref{rem:C_vs_C0}.
\begin{thm}[Self-similar behavior]\label{thm_2}
Let $\rho > 0$.
\begin{enumerate}
 \item For $0 \leq \lambda < \sfrac{3}{2}$ there exists $C = C(\lambda,\rho) > 0$ and a corresponding scaling function 
\begin{align*}
s(t) &= Ct^\beta \qquad\text{with }  \beta = (3-2\lambda)^{-1} ,
\end{align*}
such that every global solution $c$ to equation \eqref{e:EDG:lambda} with $M_1[c] = \rho$ converges
\begin{equation}
 \begin{cases}
  \mu_c(t) \rightharpoonup \rho\, g_\lambda \quad\text{ with respect to }\mathcal{C} &\text{ if }\lambda \in [0,1)\\
  \mu_c(t) \rightharpoonup \rho\, g_\lambda \quad\text{ with respect to } \mathcal{C}_0 &\text{ if }\lambda \in [1,\sfrac{3}{2})
 \end{cases} \qquad \text{as } t\to \infty .
\end{equation}
 \item For $\lambda = \sfrac{3}{2}$ there exists a scaling function $s:\overline{\R}_+ \to \overline{\R}_+$ and a constant $C=C(\rho)$ such that for every $\varepsilon > 0$ it holds
\begin{align*}
\lim_{t \to \infty} \exp\bra*{-(C+\varepsilon)t}\, s(t) = 0 \quad\text{ and }\quad \lim_{t \to \infty} \exp\bra*{-(C-\varepsilon)t}\, s(t) = \infty,
\end{align*}
and every global solution $c$ to equation \eqref{e:EDG:lambda} with $M_1[c] = \rho$ converges 
\begin{equation}
 \mu_c(t) \rightharpoonup \rho\, g_\lambda \quad\text{ with respect to }\mathcal{C}_0 \qquad\text{as } t\to \infty .
\end{equation}
\item For $\sfrac{3}{2} < \lambda < 2$ and $t^*$ as in Theorem~\ref{thm_1}~(\ref{thm_1:3}) there exists $C = C(\lambda,\rho) > 0$ such that for the scaling function
\begin{align*}
s(t) &= C(t^*-t)^\beta \qquad\text{with }  \beta = (3-2\lambda)^{-1} ,
\end{align*}
every solution $c$ to equation \eqref{e:EDG:lambda} existing on the finite time interval $[0,t^*)$, converges 
\begin{equation}
  \mu_c(t) \rightharpoonup \rho\, g_\lambda \qquad\text{ with respect to } \mathcal{C}_0 \qquad\text{as } t \to t^* . 
\end{equation}
\end{enumerate}
\end{thm}
\begin{rem}\label{rem:C_vs_C0}
Note that the difference between the weak convergence with respect to $\mathcal{C}$ in comparison to the one with respect to $\mathcal{C}_0$ in Definition~\ref{defn:weak_convergence} is that in the latter a Dirac measure at $0$ might occur. The reason why we can only prove convergence with respect to~$\cC_0$ in the case $\lambda \geq 1$ is that the analysis relies on an energy method involving a discrete version~\eqref{e:Dirichlet_form:discrete} of the weighted $H^1$-seminorm
\begin{align*}
\cE_\lambda(f) = \int_0^\infty x^\lambda \abs*{f'(x)}^2 \dx{x},
\end{align*}
for which the corresponding embedding into $C^{0,\frac{1-\lambda}{2}}([0,\infty))$ only holds for $\lambda < 1$, while for $\lambda \geq 1$ the modulus of continuity is only controlled away from $x=0$, see Lemma~\ref{lem:Nash_continuity} and Proposition~\ref{prop:equicontinuity}. We conjecture that the weak convergence in Theorem~\ref{thm_2} in fact holds with respect to $\cC$ for all $ \lambda \in [0,2)$.
Our results still imply that the total variation of~$\mu_c$, and hence the size of the Dirac, is a priori bounded from above in terms of moments of the initial data.
\end{rem}
\subsection{Time-change and tail distribution}\label{Ss.tail}
The common factor $M_\lambda$ in~\eqref{e:EDG:lambda} is eliminated through the time change 
\begin{equation}\label{e:def:tau}
\tau(t) = \int_0^t M_\lambda[c](s) \dx{s}.
\end{equation}
Since the right-hand-side of \eqref{e:EDG:lambda} never contains $c_0$ which can simply be obtained from the conservation law~\eqref{e:mu:M0}, $c_0$ is ignored in the following considerations. Consequently, we define $u(\tau(t),k) = c_k(t)$ for $k \in \N$. Since $a_\lambda(0)=0$ for all $\lambda\in [0,2)$, the value of $u(\tau(t),0)$ is not specified. Nevertheless, it is convenient to set it to zero $u(\tau,0) = 0$ for all $\tau\geq 0$. 
We see that $u$ solves the equation
\begin{align*} 
\pderiv{}{\tau} u(\tau,k) &= (k-1)^\lambda u(\tau,k-1) - 2k^\lambda u(\tau,k) + (k+1)^\lambda u(\tau,k+1), \quad k \geq 1, 
\end{align*}
which can be written such that it takes the form of a spatially discrete heat equation with Dirichlet boundary condition
\begin{align*} \tag{DP} \label{DP}
\begin{cases}
\del_\tau u = \Delta_\N (a_\lambda u), & k \geq 1, \\
(a_\lambda u)(\tau,0) = 0, & \tau \geq 0.
\end{cases}
\end{align*}
The case $\lambda=0$ can be treated explicitly (see~Appendix~\ref{S.lambda0}).
Here, the discrete Laplacian $\Delta_\N$ is conveniently expressed by the discrete differential operators
\begin{align}
\del^- u(k) = u(k) - u(k-1) \quad\text{and}\quad \del^+ u(k) = u(k+1) - u(k),  \quad\text{for } k\geq 1 , \label{e:def:del_pm}
\end{align}
such that it holds $\Delta_\N = \del^- \del^+$. An elementary calculation shows that the discrete differential operators satisfy a version of the integration by parts formula
\begin{align} \label{e:IntByParts}
\sum_{k=a}^{b}\del^+u(k)v(k) = u(b+1)v(b) - u(a)v(a-1) - \sum_{k=a}^b u(k)\del^-v(k).
\end{align}
For brevity, we abuse notation and subsequently write $u = u(t,k)$. If $u$ is a solution of equation \eqref{DP} we can directly calculate the evolution of the moments
\begin{align*}
\pderiv{}{t} M_0 &= \sum_{k=1}^\infty \Delta_\N (a_\lambda u) = -u(t,1) \leq 0, \\
\pderiv{}{t} M_1 &= \sum_{k=1}^\infty k \Delta_\N (a_\lambda u) = 0.
\end{align*}
The first identity is highlighting the fact, that $c_0$ is omitted and the second identity is the conservation of total mass. It is more convenient to study not the equation for $u$, but the one of the tail distribution $U$ associated to $u$ given by
\begin{align} \label{tail_distr}
U(t,k) = \sum_{l \geq k} u(t,l), \qquad\text{for } k\geq 1 . 
\end{align}
Again, the value of $U(t,0)$ is not specified, but by the convention $u(t,0)=0$, we obtain $U(t,0)=U(t,1)$, which we interpret as Neumann boundary condition.
The main motivation is that the evolution operator from~\eqref{DP} becomes a weighted Laplace operator $L_\lambda$ in divergence form defined on the Hilbert space $\ell^2(\N)$:
\begin{align}\label{e:def_Llambda}
L_\lambda U (k) &= \del^-(a_\lambda \del^+ U)(k).
\end{align}
It is obvious by the integration by parts formula \eqref{e:IntByParts} that $L_\lambda$ is symmetric and negative semi-definite with Dirichlet form given by
\begin{equation}\label{e:Dirichlet_form:discrete}
E_\lambda(U,V) = \skp{ V,-L_\lambda V}_2 = \sum_{k=1}^\infty k^\lambda \,\del^+ U \, \del^+ V.
\end{equation}
We also write $E_\lambda(U) = E_\lambda(U,U)$. Now, let $u$ be a solution to \eqref{DP} (cf.~Corollary~\ref{cor:DP_well_posedness} for well-posedness) and $U$ as in \eqref{tail_distr}, then $U$ solves the Neumann problem
\begin{align}\tag{NP} \label{NP}
\begin{cases}
\del_t U = L_\lambda U, \\
(a_\lambda\del^+U)(t,0) = 0.
\end{cases}
\end{align}
Indeed, for $k \geq 1$ we calculate 
\begin{align*}
\del_t U(t,k) &= \sum_{l \geq k} \Delta_\N (a_\lambda u)(l)  = - \del^-(a_\lambda u)(k), \\
\del^+ U(t,k) &= \sum_{l \geq k+1} u(t,l) - \sum_{l \geq k} u(t,l) = -u(t,k).
\end{align*}
Furthermore, we find that 
\begin{align*}
\sum_{k \geq 1} U(t,k) = \sum_{k \geq 1}\sum_{l \geq k}u(t,l) = \sum_{l \geq 1} u(t,l) \sum_{k \leq l} = \sum_{l \geq 1} u(t,l) l = M_1[u], 
\end{align*}
which shows that $M_0[U] = M_1[u] = \mathrm{const.}$ A formal dimensional analysis suggests that $k \propto t^\alpha$, where 
\begin{align} \label{alpha_def}
\alpha = \frac{1}{2-\lambda} \in \biggl[\frac{1}{2},\infty\biggr),
 \end{align}
which corresponds for $\lambda=0$ to the classical parabolic scaling. In the following identities and estimates, we see that occurrences of square-roots in the classical parabolic theory are replaced by the exponent $\alpha$. At the heart of the analysis of equation \eqref{NP} is the following discrete Nash-inequality, which connects the Dirichlet form~\eqref{e:Dirichlet_form:discrete} to the $L^1$ and $L^2$-norm of $U$. 
\begin{prop}[Discrete Nash-inequality] \label{prop:discrete_nash}
Let $\lambda \in [0,2)$. Then for all $U \in \ell^2(\N)$ with $E_\lambda(U) < \infty$
it holds
\begin{align} \tag{DNI} \label{DNI}
\norm{U}_{2}^2 \lesssim \norm{U}_1^{\frac{2(2-\lambda)}{3-\lambda}} E_\lambda(U)^\frac{1}{3-\lambda}.
\end{align}
\end{prop}
Here and in the following, we use the notation $A \lesssim B$ if there is a numerical constant $C=C(\lambda)>0$ independent of all other parameters such that $A\leq C B$. We write $A\approx B$ if $A\lesssim B$ and $B\lesssim A$.

The above discrete Nash-inequality enables us to obtain the optimal decay rates of the $L^2, L^\infty$ norms and of the Dirichlet energy $E_\lambda$ for the fundamental solution to equation~\eqref{NP}. These estimates imply scaling-invariant decay and continuity estimates in time and space for general solutions.
\begin{thm}[Decay and continuity]\label{thm:NP}
 Let $U$ be a solution to~\eqref{NP}, then 
\begin{align}
\norm{U(t,\cdot)}_\infty &\lesssim \norm{U_0}_{1} (1+t)^{-\alpha}.  \label{e:uniform_decay:Linfty}
\end{align}
Moreover, there exist explicit continuous functions (see Lemma~\ref{lem:Nash_continuity}) $\theta_\lambda:\R_+ \to \R_+$ and $\omega_\lambda: [1,\infty) \to \R_+$ with $\omega_\lambda(1)=0$ such that 
\begin{align}
|U(t,k_2)-U(t,k_1)| &\lesssim \norm{U_0}_1 t^{-\alpha} \abs[\big]{\theta_\lambda(t^{-\alpha}k_2) - \theta_\lambda(t^{-\alpha}k_1)}^\frac{1}{2}, \label{e:Nash_continuity:space} \\
|U(t,k)-U(s,k)| &\lesssim \norm{U_0}_1 s^{-\alpha} \omega_\lambda\bigl(\sfrac{t}{s}\bigr) \label{e:Nash_continuity:time},
\end{align}
for all $k,k_1,k_2\in \N$ and all $0 < s \leq t$.
\end{thm}
The above decay estimates translate to the lower coarsening bounds from Theorem~\ref{thm_1}, as the length-scale $\ell$ is inversely proportional to the zero moment. Together with the continuity estimates they also provide the necessary compactness to prove self-similarity in Theorem~\ref{thm_2}.

\subsection{Continuum equation and scaling solution}\label{Ss.IntroHeat}
The space-continuous analogue of equation \eqref{NP} is
\begin{align} \tag{NP'} \label{NP'} 
\begin{cases}\del_t \varphi = \del_x (a_\lambda \del_x \varphi) = \mathcal{L}_\lambda \varphi, & (t,x) \in \R_+^2, \\
a_\lambda\del_x \varphi|_{x=0} = 0, &t\in \R_+ , \\
\varphi(0,\cdot) = \varphi_0, &x\in \R_+ .  \\
\end{cases}
\end{align}
where the boundary condition is the natural one for the equation to conserve the zero moment. For $\lambda=0$, we get the classical Neumann boundary condition $\del_x \varphi|_{x=0} = 0$ for the heat equation, while for $\lambda > 0$ the coefficient $a_\lambda$ vanishes at $0$ and solutions are in general not smooth up to the boundary. The degeneracy has also the effect, that the boundary condition is only imposed for $\lambda< 1$, that is in the range $\lambda\geq 1$ the equation~\eqref{NP'} is well-posed without any boundary condition (see Remark~\ref{rem:NP:bc}). By the conservation of the zero moment, it is natural to look for a (normalized) scaling solution
\begin{align} \label{scaling_solution}
\gamma_\lambda(t,x) = t^{-\alpha}\mathcal{G}_\lambda(t^{-\alpha}x),
\end{align}
for some profile  $\mathcal{G}_\lambda$. Plugging this ansatz into the equation leads by simple calculations as in~\cite[Ch.~III]{BenNaim2003} to 
\begin{align} \label{U_profile}
\mathcal{G}_\lambda(x) &= Z_\lambda^{-1} \exp\bra[\big]{-\alpha^2 x^{2-\lambda} },
\end{align}
where $Z_\lambda$ is chosen such that $\mathcal{G}_\lambda$ has integral $1$ and explicitly given in~\eqref{e:def:Zlambda}.
In general a solution to equation \eqref{NP'} with given initial data can be constructed using the associated fundamental solution $\Psi_\lambda(t,x,y)$ (see Proposition \ref{prop:fund_sol}).

Our definition of solutions in Section~\ref{Ss.NP_properties} immediately entails that the scaling solution~$\gamma_\lambda$ from~\eqref{scaling_solution} with $\cG_\lambda$ given in~\eqref{U_profile} solves~\eqref{NP'} in a suitable measure-valued sense. To compare solutions on $\N$ of~\eqref{NP} with those on $\overline{\R}_+$ of~\eqref{NP'}, we introduce for $\varepsilon > 0$ the following operations between discrete measures on $\N$ and measures on $\overline{\R}_+$
\begin{align}
 \iota_\eps: \cM(\N)\to \cM\bigl(\overline{\R}_+\bigr) \qquad \text{with}\qquad (\iota_\varepsilon U) (x) &= \varepsilon^{-\alpha}U(\lfloor \varepsilon^{-\alpha}x \rfloor+1), \label{inclusion} \\
\pi_\eps : \cM\bigl(\overline{\R}_+\bigr)\to \cM(\N) \qquad\text{with}\qquad (\pi_\varepsilon \mu)(k) &= \mu\bra[\big]{[(k-1)\varepsilon^{-\alpha},k\varepsilon^{-\alpha})} . \label{projection} 
\end{align}
Note that we have $\pi_\varepsilon \circ \iota_\varepsilon= \mathrm{id}$, and both operations are adjoint to each other in the sense that for $U\in \cM(\N)$ and $\mu\in \cM(\overline{\R}_+)$ it holds
\begin{align*}
\int_0^\infty \iota_\varepsilon U(x)  \dx\mu(x) = \varepsilon^{-\alpha}\sum_{k=1}^\infty U(k) \pi_\varepsilon \mu(k).
\end{align*}
As a consequence, the maps are mass-conserving, i.e
\begin{align*}
\int_0^\infty \iota_\varepsilon U(x) \dx{x} &= \sum_{k=1}^\infty U(k) \qquad\text{and}\qquad \sum_{k=1}^\infty \pi_\varepsilon \mu(k) = \mu(\overline{\R}_+).
\end{align*}
Now let $U_\varepsilon$ be a sequence of solutions to equation \eqref{NP} with initial data $U_{0,\varepsilon}$. We define the sequence of functions $\mathcal{U}_\varepsilon$ by
\begin{align} \label{approximate_solution}
\mathcal{U}_\varepsilon(t,x) = (\iota_\varepsilon U_\varepsilon)(\varepsilon^{-1}t,x) = \varepsilon^{-\alpha} U_\varepsilon(\varepsilon^{-1}t,\lfloor \varepsilon^{-\alpha}x \rfloor+1).
\end{align}
Then we have the following convergence result.
\begin{thm}[Convergence of the tail distribution] \label{thm_discrete_to_continuous}
Let $\mathcal{U}_\varepsilon$ as above and assume that $\norm{U_{0,\varepsilon}}_1$ is bounded and $\mathcal{U}_\varepsilon(0,\cdot) \rightharpoonup \mu_0$ as $\varepsilon \to 0$ for some $\mu_0 \in \mathcal{M}(\overline{\R}_+)$. Then there exists a unique global-in-time weak solution~$\mathcal{U}$ to equation \eqref{NP'} with initial data $\mu_0$ and it holds
\begin{enumerate} 
\item If $ 0 \leq \lambda < 1$, $\mathcal{U}_\varepsilon \to \mathcal{U}$ locally uniformly on $\R_+ \times \overline{\R}_+$. 
\item If $1 \leq \lambda < 2$, $\mathcal{U}_\varepsilon \to \mathcal{U}$ locally uniformly on $\R_+^2$ and
\begin{align*}
\sup_{0 < \varepsilon \leq 1} \mathcal{U}_\varepsilon(t,0) < \infty \qquad\text{for all } t > 0 .
\end{align*}
\end{enumerate}
\end{thm}
The precise definition of weak solutions will be given in Section~\ref{S.scaling_limit}. If $U_{0,\varepsilon} = U_0$ for some $U_0 \in \ell_+^1(\N)$  with $M_0[U] = \rho$, then it is easy to check that $\mathcal{U}_\varepsilon(0,\cdot) \rightharpoonup \norm{U_0}_1 \delta_0=\rho\, \delta_0$. Hence, by applying Theorem~\ref{thm_discrete_to_continuous}, we get that $\mathcal{U}_\varepsilon$ converges to a multiple of the solution starting from $\delta_0$, which is the scaling solution $\gamma(t,x) = t^{-\alpha} \mathcal{G}_\lambda(t^{-\alpha}x)$ defined in~\eqref{scaling_solution}. In particular for $t=1$, we have 
\[
 \mathcal{U}_\varepsilon(1,x) = \varepsilon^{-\alpha}U(\varepsilon^{-1},\lfloor \varepsilon^{-\alpha} x \rfloor + 1) \to \rho\, \gamma(1,x) = \rho\, \mathcal{G}_\lambda(x) .
\]
Hence, the scaling limit in fact implies long-time behavior after setting $t=\varepsilon^{-1}$.
\begin{cor} \label{cor:U_longtime}
Let $U$ be a solution to equation \eqref{NP} with $M_0[U] = \rho$. Then the rescaled function $\hat{U}(t,x) = t^{\alpha}U(t,\lfloor t^{\alpha} x \rfloor + 1)$ is locally uniformly bounded on $\R_+ \times \overline{\R}_+$ and the following holds:
\begin{enumerate} 
\item If $ 0 \leq \lambda < 1$, $\hat{U}(t,\cdot) \to \rho \, \mathcal{G}_\lambda$ as $t \to \infty$ locally uniformly on $ \overline{\R}_+$.
\item If $1 \leq \lambda < 2$, $\hat{U}(t,\cdot) \to \rho \, \mathcal{G}_\lambda$ as $t \to \infty$ locally uniformly on $ \R_+$ and
\begin{align*}
\limsup_{t \to \infty} \hat{U}(t,0) < \infty.
\end{align*}
\end{enumerate}
\end{cor}

\addtocontents{toc}{\SkipTocEntry}
\subsection*{Outline} The paper is organized as follows. In Section~\ref{S.Discrete_Results} we analyze solutions to equation~\eqref{NP} and deduce results for equations~\eqref{DP} and \eqref{e:EDG:lambda}. The main results which are crucial throughout the paper are the $L^\infty$-decay and continuity estimates for solutions to~\eqref{NP} (Section~\ref{Ss.Continuity}), which utilize the discrete weighted Nash-inequality~\eqref{DNI} proved in Section~\ref{Ss.Discrete_Nash}. The usefulness of these estimates lies in the fact that they are optimal in terms of the scaling $k \sim t^\alpha$. On the level of equation~\eqref{DP}, the $L^\infty$-decay translates to a decay estimate for the zero moment, which is synonymous with the coarsening rate. Proving further scale-characteristic bounds for the moments (Section~\ref{Ss.moment_estimates}) we obtain estimates for the time change $\tau$ as in~\eqref{e:def:tau} which allows to relate the analysis to equation~\eqref{e:EDG:lambda} and obtain Theorem~\ref{thm_1}. 

In Section~\ref{S.scaling_limit} we prove the discrete-to-continuous scaling limit. First, we give the explicit construction of the fundamental solution of the continuum problem~\eqref{NP'}, which is possible via a suitable change of variables relating the evolution to the explicitly analyzable Bessel process. Decay and regularity properties of solutions can be read off from the explicit fundamental solution. Moreover, it allows us to define a sensible notion of weak solution for equation~\eqref{NP'} in terms of the adjoint equation, which includes the scaling solution $\gamma_\lambda$ and has a built-in uniqueness property. The proof of Theorem~\ref{thm_discrete_to_continuous} is given in Section~\ref{Ss.proof_scaling_limit} and relies on a replacement estimate for the defect between the discrete operator $L_\lambda$ and the continuous operator~$\mathcal{L}_\lambda$ (Section~\ref{Ss.replacement}) which yields that the rescaled discrete solutions $\mathcal{U}_\varepsilon$ are approximate weak solutions to equation~\eqref{NP'}. Here, the technical part is due to the degeneracy of the equation, which implies that the test functions are not smooth at $x=0$ (Section~\ref{Ss.approximate_weak_sol}). Using this approximation property and the compactness inherited from the scale-invariant decay and continuity estimates, one can pass to the limit and obtain Theorem~\ref{thm_discrete_to_continuous}. 

To prove Theorem~\ref{thm_2}, we relate the scaling-limit to the long-time behavior of the empirical measures associated with solutions to equation~\eqref{DP} in Section~\ref{S.convergence_empirical_meas}. In particular, the scaling limit implies precise asymptotics for the moments which translate to asymptotics for the time change $\tau$ and allow us to obtain the self-similar behavior for solutions to~\eqref{e:EDG:lambda} with explicit scaling function.

\section{Analysis of discrete equations}\label{S.Discrete_Results}

\subsection{Well-posedness}\label{Ss.DP-wellposed}

Before analyzing properties of solutions, we first collect some well-posedness results. For this we specify the notions of solutions for all three equations. We define suitable weighted $\ell^1(\N)$-spaces by setting for $\mu\geq 0$ 
\begin{align*}
X_\mu(\N) &= \{ u \in \ell^\infty(\N) \colon \norm{u}_{X_\mu} < \infty \}, \qquad\text{ with }\qquad \norm{u}_{X_\mu} = M_\mu\bigl[|u|\bigr], \\  
X_\mu^+(\N) &= \{u \in X_\mu \colon u \geq 0 \},
\end{align*}
and define $X_\mu(\N_0),X_\mu^+(\N_0)$ in the obvious way. We consider all of the above spaces as Banach spaces equipped with the norm $\norm{\cdot}_{X_\mu}$. When there is no danger of confusion, we just write $X_\mu$, $X_\mu^+$.

\begin{defn}[Solutions to \eqref{e:EDG:lambda}]\label{def:sol:EDG:lambda}
Let $\lambda \in [0,2)$, $c^{(0)} \in X_{\max(1,\lambda)}^+(\N_0)$ and $T \in (0,\infty]$. Then $c = c_k(t)\colon [0,T) \to X_{\max(1,\lambda)}^+(\N_0)$ is a solution to equation \eqref{e:EDG:lambda} with initial data $c^{(0)}$ and kernel given in \eqref{e:def:K_alambda} provided that:
\begin{enumerate}
\item It holds $t \mapsto \norm{c(t,\cdot)}_{X_{\max(1,\lambda)}} \in L^\infty_{\loc}([0,T))$.
\item For every $k \geq 0$ holds $t \mapsto c_k(t) \in C^0([0,T))$ and 
\begin{align*}
c_k(t) = c^{(0)}_k + \int_0^t \mathrm{EDG}_\lambda[c](s,k) \dx{s},
\end{align*}
where $\mathrm{EDG}_\lambda[c]$ denotes the right-hand side in \eqref{e:EDG:lambda}.
\end{enumerate}
\end{defn}

\begin{defn}[Solutions to \eqref{DP}]
Let $\lambda \in [0,2)$, $\mu \geq 1$, $u_0 \in X_\mu(\N)$ and $T \in (0,\infty]$. Then $u = u(t,k)\colon [0,T) \to X_\mu$ is a solution to equation \eqref{DP} in $X_\mu$ with initial data $u_0$ if the following holds:
\begin{enumerate}
\item It holds $t \mapsto \norm{u(t,\cdot)}_{X_{\mu}} \in L^\infty_{\loc}([0,T))$.
\item For every $k \in \N$ holds $t \mapsto u(t,k) \in C^0([0,T))$ and
\begin{align*}
u(t,k) = u_0(k) + \int_0^t \Delta_\N(a_\lambda u)(s,k) \dx{s}.
\end{align*}
\end{enumerate} 
\end{defn}
\begin{defn}[Solutions to \eqref{NP}]\label{def:sol:NP}
Let $\lambda \in [0,2)$, $\mu \geq 0$, $U_0 \in X_\mu(\N)$ and $T \in (0,\infty]$. Then $U = U(t,k)\colon [0,T) \to X_\mu$ is a solution to equation \eqref{NP} in $X_\mu$ with initial data $U_0$ if the following holds:
\begin{enumerate}
\item It holds $t \mapsto \norm{U(t,\cdot)}_{X_{\mu}} \in L^\infty_{\loc}([0,T))$.
\item For every $k \in \N$ holds $t \mapsto U(t,k) \in C^0([0,T))$ and 
\begin{align}\label{e:def:sol:NP}
U(t,k) = U_0(k) + \int_0^t L_\lambda U (s,k) \dx{s}.
\end{align}
\end{enumerate}
\end{defn}
In all three cases we call solutions \emph{global solutions} in the case $T = \infty$ and \emph{local solutions} in the case $T < \infty$. Note that the integral formulations for $u$ and $U$ imply that solutions are indeed smooth and \eqref{DP}, respectively \eqref{NP} hold pointwise, while a solution to~\eqref{e:EDG:lambda} is a priori only Lipschitz continuous. We start with the well-posedness of~\eqref{NP} and then deduce corresponding results for the other two equations, see Corollaries \ref{cor:DP_well_posedness} and \ref{cor:EDG_well_posedness}.

\begin{prop}\label{prop:NP:well-posed}
For $\mu \geq 0$ and $U_0 \in X_\mu(\N)$ exists a unique global solution $U$ to equation \eqref{NP} in $X_\mu$ with initial data $U_0$ given by the representation formula 
\begin{align*}
U(t,k) = \sum_{l=1}^\infty \Phi(t,k,l)\,U_0(l),
\end{align*}
where $\Phi = \Phi(t,k,l)$ is the fundamental solution, i.e $\Phi(\cdot,\cdot,l)$ is the solution to equation~\eqref{NP} in $X_\mu$ for all $\mu \geq 0$ with initial data $\Phi(0,k,l) = \delta_{kl}$.
\end{prop}
The proof of Proposition \ref{prop:NP:well-posed} is split into several auxiliary results. First, we provide a technical Lemma, which in its full scope is not needed in the existence proof, but plays a crucial role in the derivation of moment estimates later.
\begin{lem}
\begin{enumerate}
\item For all $\mu \geq 0$ and $\lambda \in [0,2)$, it holds
\begin{align*}
\frac{L_\lambda(k^\mu)}{k^{\mu+\lambda-2}} \to \mu(\mu+\lambda-1), \qquad \text{as }  k \to \infty.
\end{align*}
\item For all $\mu > 0$ and $\lambda \in [1,2)$, there exists a positive constant $C=C(\mu,\lambda)\geq 1$ such that
\begin{align*}
C^{-1} k^{\mu+\lambda-2} \leq L_\lambda(k^\mu) \leq C k^{\mu+\lambda-2}, \qquad \text{for all } k \in \N.
\end{align*}
\end{enumerate}\label{lem:L_asymptotics}
\end{lem}
\begin{proof}
We calculate
\begin{align*}
L_\lambda (k^\mu) &= k^\lambda \bra[\big]{(k+1)^\mu-k^\mu} - (k-1)^\lambda \bra[\big]{k^\mu - (k-1)^\mu} \\
&= k^{\mu + \lambda} \bra*{(1+k^{-1})^\mu - 1 + (1-k^{-1})^\lambda \bra[\big]{(1-k^{-1})^\mu - 1} } = k^{\mu + \lambda}f_{\mu,\lambda}(k^{-1}),
\end{align*}
where $f_{\mu,\lambda}(x) = (1+x)^\mu - 1 + (1-x)^\lambda \bra*{(1-x)^\mu - 1}$. 
To show the first statement, a simple calculation gives that
\begin{align*}
f_{\mu,\lambda}(0) = 0 =f'_{\mu,\lambda} \qquad\text{and}\qquad f''_{\mu,\lambda}(0) = 2\mu(\mu+\lambda-1),
\end{align*}
hence $f_{\mu,\lambda}(x)x^{-2} \to \mu(\mu+\lambda-1)$ as $x \to 0$, which gives $L_\lambda (k^\mu)k^{2-\lambda-\mu} \to \mu(\mu+\lambda-1)$ as $k \to \infty$. The second statement also follows from the asymptotic behavior of $f_{\mu,\lambda}$ if we can show in addition that $f_{\mu,\lambda} > 0$ on $(0,1]$. For that end, we view $f_{\mu,\lambda}$ as a function of two variables $f_{\lambda}(\mu,x) = f_{\mu,\lambda}(x)$ with parameter $\lambda$. Note that we have $0 = f_\lambda(0,x) = f_\lambda(\mu,0)$. We calculate the partial derivative for $\mu > 0$, $x \in (0,1)$
\begin{align*}
\del_\mu f_\lambda(\mu,x) &= \log(1+x)(1+x)^\mu + (1-x)^\lambda \log(1-x)(1-x)^\mu \\
&\geq \log(1+x)(1+x)^\mu + (1-x)\log(1-x)(1+x)^\mu \\
 &= (1+x)^\mu (\log(1+x) + (1-x) \log(1-x)) = (1+x)^\mu g(x),
\end{align*}
where the lower bound follows from the fact that $\log(1-x) \leq 0$, $(1-x)^\lambda \leq (1-x)$ for $\lambda \geq 1$ and $(1-x)^\mu \leq (1+x)^\mu$ for $\mu > 0$. If we can show that $g(x) > 0$ for all $x \in (0,1)$, then it follows from $f_\lambda(0,x) = 0$ that $f_\lambda(\mu,x) > 0$ for all $\mu > 0$, $x \in (0,1]$. Calculating derivatives of $g$, we have
\begin{align*}
g'(x) &= (1+x)^{-1} - \log(1-x) - 1, \\
g''(x) &= -(1+x)^{-2} + (1-x)^{-1}, \\
g'''(x) &= 2(1+x)^{-3} + (1-x)^{-2}.
\end{align*}
Thus we have $g(0) = g'(0) = g''(0) = 0$ and $g''' > 0$ on $[0,1)$, hence $g > 0$ on $[0,1)$.
\end{proof}
Next, we prove a first existence result. To that end we introduce the following notation: For a function $f\colon \overline{\R}_+ \times \N \to \R$ we write $f \in C_c^\infty(\overline{\R}_+ \times \N)$ if $f$ satisfies the following properties
\begin{enumerate}
\item There exists $N \in \N$ and $ T \in \R$ such that $f(t,k) = 0$ if $t \geq T$ or $k \geq N$.
\item For every $k \in \N$ the map $t \mapsto f(t,k)$ is smooth.
\end{enumerate}
\begin{lem} \label{lem:NP_inhomogeneous}
Let $\mu \geq 0$, $U_0 \in X_\mu^+(\N)$ and $f \in C_c^\infty(\overline{\R}_+ \times \N)$ with $f \geq 0$. Then there exists a global solution $U$ in $X_\mu^+$ with initial data $U_0$ to the inhomogeneous equation
\begin{align} \label{NP_inhomogeneous}
\del_t U - L_\lambda U = f ,
\end{align}
in the sense of Definition~\ref{def:sol:NP} with~\eqref{e:def:sol:NP} replaced by $U(t,k) = U_0(k) + \int_0^t\bra*{L_\lambda U +f} (s,k) \dx{s}$.
\end{lem}
\begin{proof}
For existence of solutions we apply standard regularization and truncation techniques. For $m > 0$ define $a_{\lambda}^{(m)}(k)$ to be
\begin{align*}
a_{\lambda}^{(m)}(k) = \begin{cases} k^\lambda, &k \leq m, \\
m^{\lambda}, &k > m.
\end{cases}
\end{align*}
Then $a_{\lambda}^{(m)}$ is bounded from above and below on $[1,\infty)$ and the corresponding elliptic operator is $L_\lambda^{(m)} =\del^- (a_{\lambda}^{(m)} \del^+ )$. By standard arguments, there exists a unique non-negative solution $U^{(m)}$ to equation
\begin{align*}
\del_t U^{(m)} - L_\lambda^{(m)} U^{(m)} = f
\end{align*}
with initial data $U_0$ that satisfies $\sup_{0 \leq t \leq T} M_\mu\bigl[U^{(m)}(t,\cdot)\bigr] < \infty$ for all $T \geq 0$ and $\mu \geq 0$ such that $M_\mu[U_0] < \infty$.  Let $N \in \N$, then by using the discrete integration by parts~\eqref{e:IntByParts}, we arrive at
\begin{align*}
\pderiv{}{t} \sum_{l=1}^N k^\mu U^{(m)} &= \sum_{l=1}^N k^\mu \del_t U^{(m)} = \sum_{l=1}^N k^\mu L_\lambda^{(m)} U^{(m)} + \sum_{l=1}^N k^\mu f \\
&= \sum_{l=1}^N L_\lambda^{(m)} (k^\mu)  U^{(m)} + \sum_{l=1}^N k^\mu f + R(t,N),
\end{align*}
where $R(t,N) = (N+1)^\mu a_{\lambda}^{(m)}(N)\del^+U^{(m)}(t,N) - \del^+(k^\mu)(N)a_{\lambda}^{(m)}(N)U^{(m)}(t,N+1)$ are the boundary terms from the discrete integration by parts. Now for $k \leq m$ we have $L_\lambda^{(m)} (k^\mu) = L_\lambda (k^\mu) \lesssim k^{\mu + \lambda - 2} \lesssim k^\mu$ by Lemma~\ref{lem:L_asymptotics}, while for $k > m$ we have
\begin{align*}
L_\lambda^{(m)} (k^\mu) = m^\lambda \del^- \del^+ (k^\mu) \lesssim m^\lambda k^{\mu-2} \leq k^{\mu + \lambda - 2}\lesssim k^\mu,
\end{align*}
again by Lemma~\ref{lem:L_asymptotics} and the fact that $L_0 = \del^- \del^+$.  Thus we obtain
\begin{align*}
\pderiv{}{t} \sum_{l=1}^N k^\mu U^{(m)} &\lesssim M_\mu[U^{(m)}] + M_\mu[f] + R(t,N).
\end{align*}
Next, we note that for fixed $m$, $R(t,N) \leq C(m) M_\mu[U^{(m)}]$, hence $R$ is bounded on each compact time interval. Moreover, because $M_\mu[U^{(m)}] < \infty$, we have $R(t,N) \to 0$ as $N \to \infty$ for every $t \geq 0$. Thus by integrating in time and letting $N \to \infty$, we obtain
\begin{align*}
M_\mu[U^{(m)}](t) \lesssim M_\mu[U_0] + \int_0^t M_\mu[f](s) \dx{s} + \int_0^t M_\mu[U^{(m)}](s) \dx{s}.
\end{align*}
We conclude with Gronwall's lemma that $\sup_{0 \leq t \leq T} M_\mu[U^{(m)}](t) \leq C(T)$ independent of $m$, because of $f \in C_c^\infty(\overline{\R}_+\times \N)$. Next, we establish compactness of the sequence~$U^{(m)}$. Because the above calculation holds in particular for $M_0$, we have that $U^{(m)}(t,k)$ is uniformly bounded on each compact time interval. We also have
\begin{align*}
\abs[\big]{L_\lambda^{(m)} U^{(m)}}(k) \lesssim k^\lambda \norm{U^{(m)}}_\infty \leq k^\lambda M_0[U^{(m)}].
\end{align*} 
Hence, for each $k \in \N$ the time derivative of $U^{(m)}(t,k)$ is uniformly bounded on every bounded time interval. By using Arzela-Ascoli's theorem and standard diagonal arguments it is easy to see that for a subsequence we have $U^{(m)}(t,k) \to U(t,k)$ as $m \to \infty$ for every $k \in \N$ locally uniformly in time. Passing to the limit in the time integrated equation
\begin{align*}
U^{(m)}(t,k) = U_0(k) + \int_0^t \bra[\big]{L_\lambda^{(m)} U^{(m)}+f}(s,k) \dx{s},
\end{align*}
it follows that $U$ is a solution to equation \eqref{NP_inhomogeneous} with initial data $U_0$. Furthermore, by Fatou's lemma all moment bounds of $U^{(m)}$ carry over to $U$, which finishes the proof.
\end{proof}
\begin{cor} \label{cor:NP_uniqueness}
For $U_0 \in X_0(\N)$ and any $T > 0$ exists at most one solution $U$ to equation~\eqref{NP} on $[0,T)$ in $X_0$ with initial data $U_0$.
\end{cor}
\begin{proof}
Because of linearity it suffices to show that every solution $U$ in $X_0$ with initial data $0$ is equal to $0$. For that end, let $T > 0$, $l \in \N$, $\eta \in C_c^\infty((0,T))$ with $\eta \geq 0$ and let $f(t,k) = \eta(t)\delta_{lk}$. Then by Lemma \ref{lem:NP_inhomogeneous} there exists a non-negative solution to the backwards equation
\begin{align*}
\del_t V + L_\lambda V = -f, 
\end{align*}
on the interval $[0,T]$ with terminal data $V(T,\cdot) = 0$ and $\sup_{0 \leq t \leq T} M_\mu[V(t,\cdot)] < \infty$ for all $\mu \geq 0$. Multiplying equation \eqref{NP} for $U$ with $V$, taking sums and integrating in time we obtain
\begin{align*}
0 &= \int_0^T \sum_{k=1}^\infty \ V \del_t U - V L_\lambda U \dx{t} = - \int_0^T \sum_{k=1}^\infty (\del_t V + L_\lambda V) U \dx{t} \\
&=  \int_0^T \sum_{k=1}^\infty f U \dx{t} =  \int_0^T \eta(t) U(t,l) \dx{t}.
\end{align*}
Here we used integration by parts in time and space, which is easily justified using that~$U$ is a solution in $X_0$ and the moment bounds on $V$. Since $\eta(t)$ and $l$ were arbitrary, we conclude that $U = 0$.
\end{proof}
Based on the well-posedness result for~\eqref{NP} from Lemma~\ref{lem:NP_inhomogeneous} and Corollary~\ref{cor:NP_uniqueness}, we have a Green function representation for the solutions.
\begin{cor} \label{cor:NP_fundamental_solution}
There exists a function $\Phi: \overline\R_+ \times \N \times \N \to \overline \R_+$ with the following properties
\begin{enumerate}
\item For all $\mu \geq 0$, $l \in \N$, $\Phi(\cdot,\cdot,l)$ is the unique global solution to equation \eqref{NP} in $X_\mu^+$ with initial data $\Phi(0,k,l) = \delta_{kl}$.
\item For all $t \geq 0$, $k,l \in \N$ it holds $\Phi(t,k,l)=\Phi(t,l,k)$.
\item For all $t \geq 0$, $l \in \N$ it holds $M_0[\Phi(t,\cdot,l)] = 1$.
\item For all $\mu \geq 0$, $U_0 \in X_\mu(\N)$ the function 
\begin{align*}
U(t,k) = \sum_{l=1}^\infty \Phi(t,k,l)U_0(l)
\end{align*}
is the unique global solution to equation \eqref{NP} in $X_\mu$ with initial data $U_0$ and $M_0[U(t,\cdot)] = M_0[U_0]$ for all $t \geq 0$.
\end{enumerate}
\end{cor}
\begin{proof}
The existence and uniqueness of $\Phi$ follows directly from Lemma \ref{lem:NP_inhomogeneous} and Corollary~\ref{cor:NP_uniqueness}. The symmetry of $\Phi$ follows by using the same parabolic approximation as in the proof of Lemma \ref{lem:NP_inhomogeneous}, where symmetry is clear for the approximating functions and hence carries over to the limit. The fourth property is a direct consequence of Lemma~\ref{lem:NP_inhomogeneous}. Indeed, since all moments of $\Phi$ are finite, the following calculations involving discrete integration by parts and differentiating under the sum can be rigorously justified to conclude that
\begin{align*}
\pderiv{}{t} \sum_{l=1}^\infty \Phi(t,k,l)U_0(l) &= \sum_{k=1}^\infty \del_t \Phi(t,k,l)U_0(l) \\
&= \sum_{k=1}^\infty L_\lambda \Phi(t,k,l)U_0(l) = L_\lambda \bra*{\sum_{l=1}^\infty \Phi(t,k,l)U_0(l)},
\end{align*}
which proves the representation formula.
Finally, from Lemma~\ref{lem:L_asymptotics} we get the bound $L_\lambda(k^\mu) \leq C k^\mu$, which gives the bound
\begin{align*}
\pderiv{}{t} M_\mu[\Phi(t,\cdot,l)] =  \sum_{k=1}^\infty L_\lambda(k^\mu)\Phi(t,\cdot,l)  \leq C M_\mu[\Phi(t,\cdot,l)] .
\end{align*}
Hence, we get $\sup_{0\leq t \leq T}M_\mu[\Phi(t,\cdot,l)] \lesssim C(T) l^\mu $. Then from the representation formula follows
\begin{align*}
M_\mu[|U(t,\cdot)|] \leq \sum_{l=1}^\infty M_\mu[\Phi(t,\cdot,l)] |U_0(l)| \leq C(T)M_\mu[|U_0|],
\end{align*}
which finishes the proof.
\end{proof}
As a consequence of the well-posedness result for equation \eqref{NP} we also obtain a well-posedness result for equation \eqref{DP}, since both equations are linked by taking the discrete derivative, respectively anti-derivative. 
\begin{cor} \label{cor:DP_well_posedness}
Let $\mu \geq \max(1,\lambda)$, $u_0 \in X_\mu^+(\N)$. Then there exists a unique global solution $u$ to equation \eqref{DP} in $X_\mu^+$ with initial data $u_0$. Furthermore, this solution satisfies 
\begin{align*}
\pderiv{}{t}M_0[u(t,\cdot)] &= -u(t,1) \qquad\text{and}\qquad \pderiv{}{t}M_1[u(t,\cdot)] = 0.
\end{align*} 
\end{cor}
\begin{proof}
Let $U_0$ denote the tail distribution associated to $u_0$. Then there exists a unique global solution to equation \eqref{NP} with initial data $U_0$ and $M_0[U(t,\cdot)] = M_0[U_0]$ for all $t \geq 0$. Then it is easily checked that $u(t,k) = -\del^+ U(t,k)$ is a solution to equation \eqref{DP} with initial data $u_0$ and $M_1[u(t,\cdot)] = M_0[U(t,\cdot)] = M_0[U_0] = M_1[u_0]$ by Corollary \ref{cor:NP_fundamental_solution}, and $\pderiv{}{t}M_0[u] = \pderiv{}{t}U(t,1) = \del^+ U(t,1) = -u(t,1).$ The bound for the higher moments also follows from Corollary \ref{cor:NP_fundamental_solution}, observing that $M_\mu[u]$ is comparable to $M_{\mu-1}[U]$, hence $u$ is a solution in $X_\mu$. The non-negativity of the solution $u$ easily follows from the comparison principle for the discrete Laplacian, showing that $u(t,k) = 0$ implies $\del_t u(t,k) \geq 0$. For uniqueness, let $u,v$ be two solutions to equation \eqref{DP} in $X_\mu$ with the same initial data. Then their tail distributions $U,V$ are solutions to equation \eqref{NP} in $X_{\mu-1}$. Indeed, for $k,N \in \N$, $N > k$ we calculate
\begin{align*}
\pderiv{}{t} \sum_{l = k}^N u(t,k) = \sum_{l = k}^N \Delta_\N(a_\lambda u)(t,k) = \del^+(a_\lambda u)(t,N) - \del^-(a_\lambda u)(t,k) . 
\end{align*}
Hence, we have the representation
\begin{align*}
\sum_{l=k}^N u(t,k) = \sum_{l=k}^N u(0,k) + \int_0^t \del^+(a_\lambda u)(s,N) - \del^-(a_\lambda u)(s,k) \dx{s}.
\end{align*}
Now $\del^+(a_\lambda u)(s,N) \leq \sup_{0 \leq s \leq t} M_\mu[u] < \infty$ and $\del^+(a_\lambda u)(s,N) \to 0$ as $N \to \infty$ for all $s \in [0,t]$ because $u$ is a solution in $X_\mu$. Hence letting $N \to \infty$ in the above equality yields
\begin{align*}
U(t,k) = U(0,k) - \int_0^t \del^-(a_\lambda u)(s,k) \dx{s} = U_0(k) + \int_0^t \del^-(a_\lambda \del^+U)(s,k) \dx{s},
\end{align*}
which, together with the fact that $M_\mu[u]$ and $M_{\mu-1}[U]$ are comparable, shows that $U$ is a solution to equation \eqref{NP} in $X_{\mu-1}$. We conclude that if $u$ and $v$ have the same initial data, then $U = V$ due to uniqueness for equation \eqref{NP} and thus $u=v$.
\end{proof}
By classical means it is straightforward to prove well-posedness of the equation~\eqref{DP} also in $X_\mu^+(\N)$ for $\mu\in [0,1]$. However, in our applications to~\eqref{e:EDG:lambda}, the $\mu$-moment with $\mu\geq \max\set{1,\lambda}$ is naturally appearing and we obtain from the above result the local well-posedness of~\eqref{e:EDG:lambda} for the full relevant range $\lambda \in [0,2)$.
\begin{cor} \label{cor:EDG_well_posedness}
Let $\lambda \in [0,2)$, $c^{(0)} \in X_{\max(1,\lambda)}^+(\N_0)$. Then there exists $T > 0$ and a solution $c$ to equation \eqref{e:EDG:lambda} on $[0,T)$ with initial data $c^{(0)}$. Furthermore, the following statements hold:
\begin{enumerate}
\item Any solution $c$ to equation \eqref{e:EDG:lambda} on a finite interval $[0,T)$ can be extended if $\sup_{0 \leq t < T}M_\lambda[c(t,\cdot)] < \infty$. 
\item Any solution $c$ to equation \eqref{e:EDG:lambda} on $[0,T)$ conserves for all $t \in [0,T)$
\begin{align*}
\sum_{k=0}^\infty c_k(t) = \sum_{k=0}^\infty c_k(0)\qquad\text{and}\qquad \sum_{k=0}^\infty kc_k(t) = \sum_{k=0}^\infty kc_k(0) .
\end{align*}
\end{enumerate}
\end{cor}
\begin{proof}
Let $u$ be the global solution to equation \eqref{DP} with initial data $u_0 = c^{(0)}$. Then by Corollary \ref{cor:DP_well_posedness} we have that $M_{\max(1,\lambda)}[u]$ is bounded on each finite time interval, which allows to define 
\begin{align*}
s: [0,\infty) \to [0,s^*) \qquad\text{with}\qquad s(t) = \int_0^t \frac{1}{M_\lambda[u](r)} \dx{r},
\end{align*}
for some $s^* \in (0,\infty]$. By setting $c_k(s(t)) = u(t,k)$ for $k \geq 1$ and $c_0 = 1-M_0[c]$, we obtain by straightforward calculus that $c_k(s)$ is a solution to equation \eqref{e:EDG:lambda} on the time interval $[0,s^*)$. To show the statement regarding extension of solutions, if $c$ is a solution to equation \eqref{e:EDG:lambda} on $[0,T)$ with $\sup_{0 \leq t < T}M_{\max(1,\lambda)}[c(t,\cdot)] < \infty$, then by equation \eqref{e:EDG:lambda} we get that $c_k$ is uniformly Lipschitz continuous on $[0,T)$ and hence there exists $c^*_k$ such that $\lim_{t \to T} c_k(t) = c^*_k$. By Fatou's lemma we also have that $M_{\max(1,\lambda)}[c^*] < \infty$, hence $c^* \in X_{\max(1,\lambda)}^+$ and we can extend the solution by solving \eqref{e:EDG:lambda} locally with initial data $c^*$. For the last statement, we recall that the time change
\begin{align*}
\tau(t) = \int_0^t M_\lambda[c(s,\cdot)] \dx{s}
\end{align*}
yields a (local) solution $u(\tau(t),k) = c_k(t)$ to equation \eqref{DP} in $X_{\max(1,\lambda)}$, and then by uniqueness for $u$ the desired result follows from the identities for $M_0[u], M_1[u]$ from Corollary \ref{cor:DP_well_posedness}.
\end{proof}

\subsection{Discrete Nash inequality: Proof of Proposition \ref{prop:discrete_nash}} \label{Ss.Discrete_Nash}

The goal of this section is to prove the discrete Nash-type interpolation inequality in Proposition~\ref{prop:discrete_nash}. We first give a proof of the continuous version of the Nash-inequality and then use it to prove the discrete version. As in the discrete case~\eqref{e:Dirichlet_form:discrete}, we define the Dirichlet form for $f,g\in L^2(\overline\R_+)$ by
\begin{align}\label{e:Dirichlet_form}
\cE_\lambda(f,g) = \int_0^\infty |x|^\lambda f'(x)g'(x) \dx{x}.
\end{align}
\begin{prop}[Continuous Nash-inequality] \label{continuous_nash}
Let $\lambda \in [0,2)$. Then for all $f \in L^2(\overline{\R}_+)$ with $\cE_\lambda(f) < \infty$ it holds
 \begin{align} \tag{CNI} \label{CNI}
\norm{f}_{2}^2 \lesssim \norm{f}_1^{\frac{2(2-\lambda)}{3-\lambda}} \cE_\lambda(f)^\frac{1}{3-\lambda}.
\end{align}
\end{prop} 
Taking Proposition~\ref{continuous_nash} for granted, we can now reduce the discrete Nash-inequality to the continuous case by considering the piecewise linear interpolation of the discrete function $U$.
\begin{proof}[Proof of Proposition~\ref{prop:discrete_nash}]
We define the function $f$ by 
\begin{align*}
f(x) = \begin{cases}
U(1), &0 \leq x \leq 1, \\
U(k) + \del^+U(k)(x-k),  &k \leq x \leq k+1.
\end{cases}
\end{align*}
Then we calculate
\begin{align*}
\norm{f}_1 &= \abs[\big]{U(1)} + \sum_{k = 1}^\infty \int_{0}^{1} \bra[\big]{ \abs[\big]{U(k)}(1-x) + \abs[\big]{U(k+1)}x}  \dx{x} 
\leq 2 \norm{U}_1 .
\end{align*}
Similarly, we get
\begin{align*}
\cE_\lambda(f) &= \sum_{k=1}^\infty |\del^+ U(k)|^2 \int_k^{k+1} x^\lambda \dx{x} \lesssim E_\lambda(U), \\
\text{and}\qquad \norm{f}_2^2 &= U(1)^2 + \sum_{k=1}^\infty \int_0^1 \bra[\big]{U(k) + \del^+U(k)x}^2 \dx{x} \\
&= U(1)^2 + \frac{1}{3}\sum_{k=1}^\infty \bra*{ U(k)^2 + U(k+1)^2 + U(k) U(k+1)} \geq \frac{1}{3}\norm{U}_2^2.
\end{align*}
Thus applying \eqref{CNI} to the function $f$ yields the desired inequality for $U$.
\end{proof}
Hence, it remains to proof Proposition~\ref{continuous_nash}. We generalize the argument given in~\cite[Section 4.4]{BDS2018} due to~\cite{CL1993}. For doing so, we have to adapt two ingredients of the proof to cover the weighted Dirichlet form~\eqref{e:Dirichlet_form}. First, we need the following weighted version of the Pólya--Szegő rearrangement inequality. 
\begin{lem}[Weighted Pólya--Szegő]\label{lem:weighted_PS}
 Let $\lambda \in [0,2)$. Then for all non-negative $f\in H^1_{\loc}(\overline{\R}_+)$ with $\cE_\lambda(f) <\infty$ holds
 \begin{equation}\label{e:weighted_PS}
   \cE_\lambda(f^*) \leq \cE_\lambda(f) ,
 \end{equation}
  where $f^*$ is the non-increasing rearrangement of $f$.
\end{lem}
\begin{proof}
  For the proof, let $k=\frac{\lambda}{2}\in [0,1)$.  Let $\Omega\subset \overline{\R}_+$ arbitrary and $\Omega^*=[0,\abs{\Omega})$ be the interval from $0$ to $\abs{\Omega}$, then it holds
  \begin{equation}\label{e:iso_perimetric}
    \int_{\partial \Omega^*} \abs{x}^{k} \, \cH^0(\dx{x}) \leq \int_{\partial \Omega} \abs{x}^k \, \cH^0(\dx{x}) .
  \end{equation}
  Indeed, for the proof one can argue that it is enough to consider intervals $\Omega=(x,x+r)$ for $x\in \overline{\R}_+$ and $r>0$ (see~\cite[Theorem 6.1]{ACMP2017} on how to reduce to this statement). Then, the desired inequality becomes the obvious statement
  \[
    r^k \leq \abs{x}^k + \abs{x+r}^k .
  \]
  For the rest of the proof, we can follow exactly along the same lines as in~\cite[Theorem 8.1]{ACMP2017} with the only difference, that now the isoperimetric inequality~\eqref{e:iso_perimetric} is used. 
\end{proof}
The next ingredient is a weighted Poincaré inequality.
\begin{lem}[Weighted Poincaré inequality]\label{lem:weighted_PI}
  For any $\lambda\in [0,2)$ exists $C_{\PI}(\lambda)\in (0,\infty)$ such that for any $R>0$ and any $f\in H^1_{\loc}(\overline{\R}_+)$ with $\int_{0}^R f \dx{x} = 0$ holds 
  \begin{equation}\label{e:weighted_PI}
    \int_{0}^R \abs{f}^2 \dx{x} \leq R^{2-\lambda} C_{\PI}(\lambda) \int_{0}^R \abs{f'}^2 |x|^\lambda \dx{x}. 
  \end{equation}
  Moreover, the constant $C_{\PI}(\lambda)$ is bounded by
  \[
    C_{\PI}(\lambda) \leq \frac{1}{2(2-\lambda)(4-\lambda)} . 
  \]
\end{lem}
\begin{proof}
  Rescaling reduces~\eqref{e:weighted_PI} to the inequality
  \[
   \int_{0}^1 \abs{f}^2 \dx{x} \leq C_{\PI}(\lambda) \int_{0}^1 \abs{f'}^2 |x|^\lambda \dx{x}. 
  \]
  The above inequality follows from an argument by~\cite{Chen1999}. Let $g:[0,1]\to \R$ be a monotone increasing absolutely continuous function. Then, it holds
  \begin{align*}
  2\int_{0}^1 \abs{f}^2 \dx{x} &= \int_{0}^1 \int_{0}^1 \bra*{ f(x)-f(y)}^2 \dx{x} \dx{y}\\
  &=  \int_{0}^1 \int_{0}^1 \bra*{ \int_x^y \frac{f'(\xi)}{\sqrt{g'(\xi)}} \sqrt{g'(\xi)} \dx{\xi}}^2 \dx{x} \dx{y} \\
  &\leq \int_{0}^1 \int_{0}^1 \int_x^y \frac{\abs*{f'(\xi)}^2}{g'(\xi)} \dx{\xi} \int_x^y g'(\xi)\dx{\xi} \dx{x} \dx{y} \\
  &= \int_{0}^1 \frac{\abs*{f'(\xi)}^2 \abs{\xi}^\lambda}{g'(\xi) \abs{\xi}^{\lambda}} \int_{0}^\xi \int_\xi^1 \bra*{g(y)-g(x)} \dx{y}\dx{x} \dx{\xi} \\
  &\leq \sup_{\xi\in [0,1]} \bra*{ \frac{1}{g'(\xi) \abs{\xi}^{\lambda}} \int_{0}^\xi \int_\xi^1 \bra*{g(y)-g(x)} \dx{y}\dx{x} } \int_{0}^1 \abs{f'(\xi)}^2 |\xi|^\lambda \dx{\xi}
  \end{align*}
  Hence for any choice of $g$, where the $\sup$ in $\xi$ is finite, the first term provides an upper bound on $C_{\PI}(\lambda)$. 
  We choose $g(\xi) = \frac{\xi^r-1}{r}$ for some $r>0$ yet to be determined and obtain the upper bound
  \begin{align}
   C_{\PI}(\lambda) \leq \frac{1}{2} \sup_{\xi\in[0,1]} \bra*{ \xi^{1-r-\lambda} \frac{\xi\bra*{1-\xi^{r}}}{r(1+r)}} .
  \end{align}
  We choose $r=1-\frac{\lambda}{2}$ and note that with this choice
  \[
    \xi^{2-r-\lambda} \bra[\big]{1-\xi^{r}} = \xi^{1-\frac{\lambda}{2}}\bra[\big]{1-\xi^{1-\frac{\lambda}{2}}} \leq \frac{1}{4} . 
  \]
  Hence, we obtain the bound
  \[
    C_{\PI}(\lambda) \leq \frac{1}{8} \frac{1}{\bra*{1-\frac{\lambda}{2}}\bra*{2-\frac{\lambda}{2}}}=\frac{1}{2(2-\lambda)(4-\lambda)} . 
    \qedhere
  \]
\end{proof}
Now, the proof of Proposition~\ref{continuous_nash} follows along the same lines as in~\cite[Section 4.4]{BDS2018}.
\begin{proof}[Proof of Proposition~\ref{continuous_nash}]
  We can assume without loss of generality that $f$ is non-negative and denote with $f^*$ its non-increasing rearrangement. Then, we have $\norm{f^*}_2 = \norm{f}_2$ by Cavalieri's principle and thanks to Lemma~\ref{lem:weighted_PS} also $\cE_\lambda(f^*)\leq \cE_\lambda(f)$. So, we can consider non-increasing non-negative functions. For any $R>0$ let $f_R = f \dsOne_{[0,R)}$. Since $f$ is non-increasing, it holds 
  \[
    f- f_R \leq f(R) \leq \bar f_R = \frac{\norm{f_R}_1}{R} .
  \]
  Taking the $L^2$-norm of the above inequality gives
  \begin{equation}\label{e:Nash:p1}
    \norm{f-f_R}_2^2 \leq \bar f_R \norm{f-f_R}_1 = \frac{\norm{f_R}_1}{R} \norm{f-f_R}_1 .
  \end{equation}
  Likewise, we can write
  \[
    \norm{f_R}_2^2 = \norm[\big]{f_R - \bar f_R}_2^2 + \norm[\big]{\bar f_R \, \dsOne_{[0,R)}}_2^2  .
  \]
Applying the weighted Poincaré inequality from Lemma~\ref{lem:weighted_PI} to the first term results in the estimate
  \begin{equation}\label{e:Nash:p2}
    \norm{f_R}_2^2 \leq R^{2-\lambda} C_{\PI}(\lambda) \cE_\lambda(f) + \frac{\norm{f_R}_1^2}{R} ,
  \end{equation}
  where we used that $\cE_\lambda(f_R)\leq \cE_\lambda(f)$. Now, we write $\norm{f}_2^2 \leq \norm{f_R}_2^2 + \norm{f-f_R}_2^2$ and apply the two estimates~\eqref{e:Nash:p1} and~\eqref{e:Nash:p2} to arrive at
  \[
    \norm{f}_2^2 \leq R^{2-\lambda} C_{\PI}(\lambda) \cE_\lambda(f) + \frac{\norm{f_R}_1}{R} \bra*{ \norm{f_R}_1+ \norm{f-f_R}_1} \leq R^{2-\lambda} C_{\PI}(\lambda) \cE_\lambda(f) + \frac{\norm{f}_1^2}{R} .
  \]
  The choice 
  \[
    R^* = \bra*{ \frac{\norm{f}_1^2}{C_{\PI}(\lambda) \cE_\lambda(f)}}^{\frac{1}{3-\lambda}}, 
  \] 
  yields the claimed estimate~\eqref{CNI}.
\end{proof}

\subsection{Decay and continuity: Proof of Theorem \ref{thm:NP}} \label{Ss.Continuity}

Recall that every solution $U$ to equation \eqref{NP} can be represented by 
\begin{equation}\label{e:NP:representation}
U(t,k) = \sum_{l=1}^\infty \Phi(t,k,l)U_0(l),
\end{equation}
where $\Phi$ is the fundamental solution, see Proposition~\ref{prop:NP:well-posed}.
By the classical arguments from~\cite{Nash1958} and the discrete Nash inequality~\eqref{DNI}, we obtain the decay of the Green function.
\begin{lem} \label{decay_lem}
 Let $\Phi:\overline\R_+ \times \N\times \N$ be the fundamental solution of~\eqref{NP} from Proposition~\ref{prop:NP:well-posed}, then it holds
\begin{align}
\norm{\Phi(t,\cdot,l)}_2 &\lesssim (1+t)^{-\frac{\alpha}{2}}\qquad\text{and}\qquad \norm{\Phi(t,\cdot,l)}_\infty \lesssim (1+t)^{-\alpha} .
\end{align}
\end{lem}
\begin{proof}
For convenience we use the notation $\Phi(t) = \Phi(t,\cdot,l)$ for some fixed $l$ during the proof. We define $f(t) = \norm{\Phi(t)}_2^2$ and calculate
\begin{align*}
\pderiv{}{t} f(t) = -2E_\lambda(\Phi(t)) .
\end{align*}
Then by the Nash-type inequality \eqref{DNI} we estimate
\begin{align*}
\pderiv{}{t} f(t) \lesssim -f(t)^{3-\lambda},
\end{align*}
where we used that all fundamental solutions have unit mass. Integrating this differential inequality and using $f(0) = 1$ we get the desired inequality for $\norm{\Phi}_2$. To strengthen this estimate to a uniform bound (with faster decay), we apply the representation formula~\eqref{e:NP:representation} and obtain
\begin{align*}
\Phi(2t,k,l) = \sum_{m=1}^\infty \Phi(t,k,m) \Phi(t,m,l) \leq \norm{\Phi(t)}_2^2,
\end{align*}
where we used that $\Phi(t,k,l) = \Phi(t,l,k)$ by symmetry of the operator $L_\lambda$.
\end{proof}
The representation formula~\eqref{e:NP:representation} for general solutions then directly implies the $L^\infty$-decay estimate in Theorem~\ref{thm:NP}. Next we analyze the temporal decay of the Dirichlet form $E_\lambda$ of solutions. The identity $\pderiv{}{t} \norm{U}_2^2 = -2E_\lambda(U)$ and the above $L^2$ estimate suggest an estimate of the form $E_\lambda(U) \lesssim t^{-(\alpha + 1)}$. The next lemma shows that this is indeed the case and in particular we have a Nash-continuity estimate.
\begin{lem} \label{lem:Nash_continuity}
 Let $\Phi:\overline\R_+ \times \N\times \N$ be the fundamental solution of~\eqref{NP} from Proposition~\ref{prop:NP:well-posed}, then for all $0 < s < t$ and $k_1,k_2,l \in \N$ it holds
\begin{align}
E_\lambda(\Phi(t,\cdot,l)) &\lesssim t^{-(\alpha+1)}, \\
|\Phi(t,k_2,l)-\Phi(t,k_1,l)| &\lesssim t^{-\alpha} \abs[\big]{\theta_\lambda(t^{-\alpha}k_2) - \theta_\lambda(t^{-\alpha}k_1)}^\frac{1}{2}, \\
|\Phi(t,k,l)-\Phi(s,k,l)| &\lesssim s^{-\alpha} \omega_\lambda\bra*{\sfrac{t}{s}},
\end{align}
where
\begin{align}
\theta_\lambda(x) &= \begin{cases}
\frac{1}{1-\lambda}x^{1-\lambda}, &\lambda \neq 1, \\
\log(x), &\lambda = 1,
\end{cases} \label{e:def:theta_lambda}\\
\omega_\lambda(r) &= \begin{cases}
\frac{2}{|1-\alpha|} \abs*{\bra*{r - \sfrac{1}{2}}^{\frac{1-\alpha}{2}} - \bra*{\sfrac{1}{2}}^{\frac{1-\alpha}{2}}}, &\lambda \neq 1, \\
\log\bra*{2 r - 1}, &\lambda = 1.
\end{cases}\label{e:def:omega_lambda}
\end{align}
\end{lem}
\begin{proof}
  The Cauchy-Schwarz inequality gives $|\skp{ U,-L_\lambda U}_2| \leq \norm{U}_2 \norm{L_\lambda U}_2$, hence with $f(t) = E_\lambda(\Phi(t,\cdot,l)) = \skp{ \Phi(t,\cdot,l),-L_\lambda \Phi(t,\cdot,l)}_2$ we have
\begin{align*}
\pderiv{}{t} f(t) = -2\skp{L_\lambda \Phi(t,\cdot,l),L_\lambda \Phi(t,\cdot,l)} \leq -2 \norm{ \Phi(t,\cdot,l)}_2^{-2}f^2 \lesssim -t^\alpha f(t)^2,
\end{align*}
where we used the $L^2$ decay estimate on $\Phi$. Integrating this differential inequality yields
\begin{align*}
f(t) \lesssim \frac{1}{f(0)^{-1} + t^{\alpha + 1}} \leq t^{-(\alpha+1)}.
\end{align*}
For the second statement the (discrete) fundamental theorem of calculus and Cauchy-Schwarz inequality imply
\begin{align*}
|\Phi(t,k_2,l)-\Phi(t,k_1,l)| &\leq \sum_{m = k_1}^{k_2} |\del^+ \Phi(t,m,l)| \leq E_\lambda(\Phi(t,\cdot,l))^\frac{1}{2}  \bra[\Bigg]{\sum_{m = k_1}^{k_2} m^{-\lambda}}^\frac{1}{2} \\
&\lesssim t^{-\frac{\alpha+1}{2}} |\theta_\lambda(k_2) - \theta_\lambda(k_1)|^{\frac{1}{2}} = t^{-\alpha} \abs[\big]{\theta_\lambda(t^{-\alpha}k_2) - \theta_\lambda(t^{-\alpha}k_1)}^\frac{1}{2}.
\end{align*}
For the last statement we use for $0 \leq t_0 < t$ the representation
\begin{align*}
\Phi(t,k,l) = \sum_{m=1}^\infty \Phi(t-t_0,k,m)\Phi(t_0,m,l).
\end{align*}
In particular for $0 < t_0 < s < t$ we have
\begin{align*}
|\del_t \Phi(t,k,l)| &= \abs[\Bigg]{\sum_{m=1}^\infty \del_t \Phi(t-t_0,k,m)\Phi(t_0,m,l) } \\
&= \abs[\Bigg]{\sum_{m=1}^\infty L_{\lambda,m} \Phi(t-t_0,m,k)\Phi(t_0,m,l) } \\
&\leq E_\lambda(\Phi(t-t_0,\cdot,k))^\frac{1}{2} E_\lambda(\Phi(t_0,\cdot,l))^\frac{1}{2} \lesssim (t-t_0)^{-\frac{\alpha+1}{2}} t_0^{-\frac{\alpha+1}{2}}, 
\end{align*}
hence 
\begin{align*}
|\Phi(t,k,l)-\Phi(s,k,l)| \leq \int_s^t |\del_r \Phi(r,k,l)| \dx{r} \leq t_0^{-\frac{\alpha+1}{2}} \int_s^t (r-t_0)^{-\frac{\alpha+1}{2}} \dx{r}.
\end{align*}
Choosing $t_0 = s/2$ and evaluating the integral on the right-hand-side we arrive at 
\begin{align*}
|\Phi(t,k,l)-\Phi(s,k,l)| &\lesssim \frac{2}{|1-\alpha|} s^{-\frac{\alpha+1}{2}} \abs*{\bra*{t - \frac{s}{2}}^{\frac{1-\alpha}{2}} - \bra*{\frac{s}{2}}^{\frac{1-\alpha}{2}} } = s^{-\alpha}\omega_\lambda\bra*{\frac{t}{s}}. \qedhere
\end{align*}
\end{proof}
Again the continuity estimates for general solutions to equation~\eqref{NP} in the second part in Theorem~\ref{thm:NP} follow from the representation~\eqref{e:NP:representation}.

\subsection{Moment estimates} \label{Ss.moment_estimates}

We want to estimate the moments $M_\mu[\Phi(t,\cdot,l)]$ of the fundamental solution of equation $\eqref{NP}$ from above and below optimally in terms of scaling. 
\begin{lem} \label{lem:moment_estimates}
  Let $\Phi:\overline\R_+ \times \N\times \N$ be the fundamental solution of~\eqref{NP} from Proposition~\ref{prop:NP:well-posed}, then  for some $C=C(\lambda,\mu)>0$ the following moment bounds hold:
\begin{align*}
&\text{for } \mu > 0:& M_\mu[\Phi(t,\cdot,l)] &\leq \bra*{l^\frac{1}{\alpha} + Ct}^{\alpha \mu} ; \\
&\text{for } \lambda \geq 1 \text{ and } \mu > 0 : & M_\mu[\Phi(t,\cdot,l)] &\geq \bra*{l^\frac{1}{\alpha} + Ct}^{\alpha \mu} ;  \\
&\text{for } \mu < 0: & M_\mu[\Phi(t,\cdot,l)] &\geq \bra*{l^\frac{1}{\alpha} + Ct}^{\alpha \mu}.
\end{align*}
\end{lem}
\begin{proof}
 For the proof, $C$ always denotes a constant that may depend on $\lambda$ and exponents~$\mu$ and $\nu$. The first estimate is obtained for $\mu \geq 2-\lambda$. Taking the time derivative, applying Jensen's inequality with the power $0 \leq \frac{\mu+\lambda-2}{\mu} \leq 1$, and using Lemma~\ref{lem:L_asymptotics} to estimate the term $L_\lambda(k^\mu)$, we have
\begin{align*}
\pderiv{}{t}M_\mu[\Phi(t,\cdot,l)] &= \sum_{k=1}^\infty L_\lambda(k^\mu)\Phi(t,k,l) \leq C \sum_{k=1}^\infty k^{\mu+\lambda-2}\Phi(t,k,l) \\
 &\leq C \bra[\Bigg]{\sum_{k=1}^\infty k^{\mu}\Phi(t,k,l)}^{\frac{\mu+\lambda-2}{\mu}} = CM_\mu[\Phi(t,\cdot,l)]^{\frac{\mu+\lambda-2}{\mu}}.
\end{align*}
By using $M_\lambda[\Phi(0,\cdot,l)] = l^\mu$, the above differential inequality is integrated to 
\begin{align*} 
M_\mu[\Phi(t,\cdot,l)] \leq \bra*{l^\frac{1}{\alpha} + Ct}^{\alpha \mu}.
\end{align*}
Then, for any $0 < \nu < \mu$ we apply again Jensen's inequality to arrive at
\begin{align*}
M_\nu[\Phi] \leq M_\mu[\Phi]^{\frac{\nu}{\mu}} \leq \bra*{l^\frac{1}{\alpha} + Ct}^{\alpha \nu},
\end{align*}
which shows that the above upper estimate holds in fact for any $\mu > 0$. Next we derive a lower bound in the case $\lambda \geq 1$, $0 < \mu < 2-\lambda$. Indeed, a similar calculation as above yields
\begin{align*}
\pderiv{}{t}M_\mu[\Phi(t,\cdot,l)] &= \sum_{k=1}^\infty L_\lambda(k^\mu)\Phi(t,k,l) \geq C \sum_{k=1}^\infty k^{\mu+\lambda-2}\Phi(t,k,l) \geq C M_\mu[\Phi(t,\cdot,l)]^{\frac{\mu+\lambda-2}{\mu}},
\end{align*}
by the second statement of Lemma~\ref{lem:L_asymptotics} and the fact that $ \frac{\mu+\lambda-2}{\mu} < 0$. This is then integrated as above and yields the inequality
\begin{align*} 
M_\mu[\Phi(t,\cdot,l)] \geq \bra*{l^\frac{1}{\alpha} + Ct}^{\alpha \mu},
\end{align*}
and by applying Jensen's inequality this inequality holds for all $\mu > 0$. Also note that the above estimate holds for all $\mu < 0$. Indeed, for $\mu < 0$ we apply Jensen's inequality again to obtain
\begin{equation*}
M_{\mu}[\Phi] \geq M_1[\Phi]^\mu \geq \bra*{l^\frac{1}{\alpha} + Ct}^{\alpha \mu}. \qedhere
\end{equation*}
\end{proof}
Next we prove a general interpolation inequality for moments.
\begin{lem} \label{lem:moment_interpolation}
For $u \in \ell_+^\infty(\N)$ with $M_1[u] < \infty$ and every $\mu \in (0,1)$ holds
\begin{align*}
M_\mu[u] \leq 2M_0[u]^{1-\mu}M_1[u]^\mu.
\end{align*}
\end{lem}
\begin{proof}
For any $N \in \N$, we have the estimate 
\begin{align*}
M_\mu[u] = \sum_{k=1}^\infty k^\mu u(t,k) &= \sum_{k=1}^N k^\mu u(t,k) + \sum_{k=N+1}^\infty k^\mu u(t,k) \leq N^\mu M_0[u] + (N+1)^{\mu - 1}M_1[u].
\end{align*}
Now, we choose $N$ to be the largest natural number such that $N \leq M_1[u]M_0[u]^{-1}$, which implies also $N \geq M_1[u]M_0[u]^{-1} - 1$ and thus the estimate
\begin{align*}
N^\mu M_0[u] + (N+1)^{\mu - 1}M_1[u] &\leq (M_1[u]M_0[u]^{-1})^\mu M_0[u] + (M_1[u]M_0[u]^{-1})^{\mu - 1}M_1[u] \\
&= 2M_0[u]^{1-\mu}M_1[u]^\mu. \qedhere
\end{align*}
\end{proof}
The estimates from Lemma~\ref{lem:moment_estimates} and Lemma~\ref{lem:moment_interpolation} yield various moment bounds for solutions to equation~\eqref{DP}.
\begin{prop} \label{prop:moment_estimates}
Any solution $u$ to the equation \eqref{DP} in $X_{\max(1,\lambda)}^+$ with $M_1[u] = \rho$ satisfies the moment bounds:
\begin{enumerate}
\item\label{prop:moment_estimates:1} For $0 < \mu < 1$, there exist constants $C_1 = C_1(\lambda,\mu)>0$ and $C_2 = C_2(\lambda)>0$ such that 
\begin{align*}
C_1 \rho^{-\frac{\mu}{1-\mu}}M_\mu[u_0]^{\frac{1}{1-\mu}}(1 \vee t)^{-\alpha} \leq M_0[u(t,\cdot)] \leq C_2\, \rho\, t^{-\alpha}.
\end{align*}
\item\label{prop:moment_estimates:2} For $0 < \mu < 1$, there exist  constants $C_1 = C_1(\lambda,\mu)>0$ and $C_2 = C_2(\lambda)>0$ such that 
\begin{align*}
C_1 M_\mu[u_0](1 \vee t)^{\alpha(\mu-1)} \leq M_\mu[u(t,\cdot)] \leq C_2\, \rho\, t^{\alpha(\mu - 1)}.
\end{align*}
\item\label{prop:moment_estimates:3} For $\mu > 1$, there exist constants $C_1 = C_1(\lambda,\mu)\geq 0$ and $C_2 = C_2(\lambda,\mu)>0$ such that 
\begin{align*}
C_1\,\rho\,t^{\alpha(\mu - 1)} \leq M_\mu[u(t,\cdot)] \leq C_2 M_\mu[u_0](1 \vee t)^{\alpha(\mu-1)}.
\end{align*}
Furthermore, $C_1$ is strictly positive for $\lambda \in [1,2)$.
\end{enumerate}
\end{prop}
\begin{proof}
We start with the second statement. Note that, up to constants that depend only on $\mu$, $M_\mu[u]$ is comparable to $M_{\mu-1}[U]$, where $U$ is the tail distribution corresponding to~$u$ and thus a solution to equation \eqref{NP}. Then the third inequality from Lemma \ref{lem:moment_estimates} and the representation formula~\eqref{e:NP:representation} yield 
\begin{align}\label{e:moment_estimates:p0}
M_{\mu-1}[U(t,\cdot)] = \sum_{l=1}^\infty M_{\mu-1}[\Phi(t,\cdot,l)]U_0(l) \geq \sum_{l=1}^\infty \bra[\big]{l^{\sfrac{1}{\alpha}} + Ct}^{\alpha (\mu-1)}U_0(l).
\end{align}
Next, for $t \leq 1$ we have 
\begin{align*}
\bra[\big]{l^{\sfrac{1}{\alpha}} + Ct}^{\alpha (\mu-1)} &\geq \bra[\big]{l^{\sfrac{1}{\alpha}} + C}^{\alpha (\mu-1)} \geq l^{\mu-1}\bra*{1 + C }^{\alpha(\mu-1)},
\end{align*}
while for $t \geq 1$ we estimate
\begin{align*}
\bra[\big]{l^{\sfrac{1}{\alpha}} + Ct}^{\alpha (\mu-1)} 
\geq t^{\alpha(\mu-1)} \bra[\big]{l^{\sfrac{1}{\alpha}} + C }^{\alpha (\mu-1)} \geq t^{\alpha(\mu-1)}l^{\mu-1}\bra*{1 + C }^{\alpha(\mu-1)} .
\end{align*}
Thus, for $t \geq 0$ we have the estimate $\bra[\big]{l^{\sfrac{1}{\alpha}} + Ct}^{\alpha (\mu-1)} \geq l^{\mu-1}\bra*{1 + C }^{\alpha(\mu-1)}(1 \vee t)^{\alpha(\mu-1)}$. Plugging this estimate into the representation formula~\eqref{e:moment_estimates:p0} gives the lower bound in statement~\emph{(\ref{prop:moment_estimates:2})}. 
Next, we note that the upper bound from the statement~\emph{(\ref{prop:moment_estimates:1})} immediately follows from the fact that $M_0[u(t,\cdot)] = U(t,0)$ and \eqref{e:uniform_decay:Linfty}. This enables us to prove the upper bound in the statement~\emph{(\ref{prop:moment_estimates:2})} by interpolation. Indeed, by Lemma \ref{lem:moment_interpolation} we have 
\begin{align*}
M_\mu[u] \leq 2 M_0[u]^{1-\mu}M_1[u]^\mu = 2 M_0[u]^{1-\mu} \rho^\mu \leq  C\,\rho\,t^{\alpha(\mu-1)}.
\end{align*}
Using the interpolation inequality in the other direction and the lower bound obtained for $M_\mu[u]$ in~\eqref{e:moment_estimates:p0}, we have for every $ \mu \in (0,1)$ that
\begin{align*}
M_0[u]^{1-\mu} \geq \frac{1}{2}\rho^{-\mu}M_\mu[u] \geq \frac{1}{2}\rho^{-\mu}C M_\mu[u_0](1 \vee t)^{\alpha(\mu-1)},
\end{align*}
which implies the lower bound in statement~\emph{(\ref{prop:moment_estimates:1})}. We turn to the proof of statement~\emph{(\ref{prop:moment_estimates:3})}. In the case $\lambda \geq 1$, the lower bound follows immediately from the second inequality in Lemma \ref{lem:moment_estimates} with $\bra[\big]{l^{\sfrac{1}{\alpha}} + Ct}^{\alpha (\lambda-1)} \geq Ct^{\alpha (\lambda-1)}$ and the representation formula, whereas the upper bound is proved along the same lines as the lower bound in the case $0 < \mu < 1$, making use of the first inequality from Lemma \ref{lem:moment_estimates}.
\end{proof}

\subsection{Coarsening rates: Proof of Theorem \ref{thm_1}}

Recall that equation \eqref{DP} and equation \eqref{e:EDG:lambda} are linked by the time change $\tau$ defined in~\eqref{e:def:tau},
where the function $u(\tau,k)$ defined by $u(\tau(t),k) = c_k(t)$ for $k \geq 1$ is a solution to equation~\eqref{DP} if $c_k(t)$ is a solution to the system \eqref{e:EDG:lambda}. Then the moment estimates from above imply the following estimates on $\tau$, from which Theorem \ref{thm_1} easily follows.
\begin{prop} \label{prop:tau_estimates}
The time change $\tau$ in~\eqref{e:def:tau} satisfies for any $0 \leq \lambda < 2$ and $\beta = (3-2\lambda)^{-1}$ the following bounds, with all constants only depending on $\lambda, \rho$ and $M_\lambda[c^{(0)}]$:
\begin{enumerate}
\item Let $0 \leq \lambda < \sfrac{3}{2}$, then every solution $c$ to equation \eqref{e:EDG:lambda} exists globally and there are positive constants $C_1,C_2,t_0$ such that
\begin{align*}
C_1 t^{\frac{\beta}{\alpha}} \leq  \tau(t) \leq C_2 t^{\frac{\beta}{\alpha}} \qquad \text{for all } t \geq t_0.
\end{align*}
\item Let $\lambda = \sfrac{3}{2}$, then every solution $c$ to equation \eqref{e:EDG:lambda} exists globally and there are positive constants $C_1,C_2,K_1,K_2,t_0$ such that
\begin{align*}
K_1 \exp(C_1t) \leq \tau(t) \leq K_2 \exp(C_2 t) \qquad \text{for all } t \geq t_0.
\end{align*}
\item Let $\sfrac{3}{2} < \lambda \leq 2$, then every solution $c$ to equation \eqref{e:EDG:lambda} exists only locally on a maximal interval $[0,t^*)$ for some $t^*>0$ and there are positive constants $C_1,C_2,t_0$ such that 
\begin{align*}
C_1 (t^*-t)^{\frac{\beta}{\alpha}} \leq \tau(t) \leq C_2 (t^*-t)^{\frac{\beta}{\alpha}} \qquad \text{for all } t_0 \leq t \leq t^* .
\end{align*}
\end{enumerate}
\end{prop}
\begin{proof}
By construction we have $\dot{\tau} = M_\lambda[u(\tau,\cdot)]$, where $u$ is a solution to equation \eqref{DP} on $\mathrm{Im}(\tau)$. Because of Corollary \ref{cor:DP_well_posedness} we can assume without loss of generality that $u$ is a global solution, even if $\tau$ is a bounded function. Then plugging the bounds from Proposition \ref{prop:moment_estimates} with $\mu = \lambda$ into the differential equation for $\tau$ one easily sees that $\tau(t)$ remains locally bounded ($\lambda \leq \sfrac{3}{2}$) or blows up in finite time ($\lambda > \sfrac{3}{2}$), see calculations below. Corollary \ref{cor:EDG_well_posedness} then implies that in the first case solutions can be extended globally, while in the second case $M_\lambda[c]$ blows up in finite time. Next, the lower moment bounds imply in any case there exists some $t_0 > 0$, depending only on $\lambda, \rho$ and $M_\lambda[c_0]$, such that $\tau(t) \geq 1$ for $t \geq t_0$, so we have differential inequalities
\begin{align} \label{e:tau:bound}
C_1\tau^{\alpha(\lambda-1)} \leq \dot{\tau} \leq C_2\tau^{\alpha(\lambda-1)} .
\end{align}
We first consider the case $\lambda < \sfrac{3}{2}$ in which $\alpha(\lambda-1) < 1$. Dividing~\eqref{e:tau:bound} by $\tau^{\alpha(\lambda-1)}$ and integrating from $t_0$ to $t$ yields
\begin{align}\label{e:tau:bound2}
\bra[\Big]{\tau(t_0)^{\frac{\alpha}{\beta}}+\tfrac{\alpha}{\beta}C_1(t-t_0)}^{\frac{\beta}{\alpha}} \leq \tau(t) \leq \bra[\Big]{\tau(t_0)^{\frac{\alpha}{\beta}}+\tfrac{\alpha}{\beta}C_2(t-t_0) }^{\frac{\beta}{\alpha}}.
\end{align}
It is easy to see that $\tau(t_0)$ can also be estimated from above and below in terms of $\lambda, \rho$ and $M_\lambda[c^{(0)}]$, hence after adjusting $t_0$ the desired inequality for $\tau$ holds. In the case $\lambda = \sfrac{3}{2}$ we have $\alpha(\lambda-1) = 1$, and hence integrating the differential inequality yields
\begin{align*}
\tau(t_0)\exp\bra[\big]{C_1 (t-t_0)} \leq \tau(t) \leq \tau(t_0)\exp\bra[\big]{C_2 (t-t_0)},
\end{align*}
which leads to the second statement. For the third statement, we have to consider~\eqref{e:tau:bound2}, but 
with $\beta$ negative in this case, which shows that $\tau$ has to blow up. The behavior at the blowup time follows after dividing the differential inequality for $\tau$ by $\tau^{\alpha(\lambda-1)}$ and integrating from $t$ to $t^*$ for $t_0 < t < t^*$ to arrive at
\begin{equation*}
\bra[\Big]{-\tfrac{\alpha}{\beta}C_2(t^* - t)}^{\frac{\beta}{\alpha}} \leq \tau(t) \leq \bra[\Big]{-\tfrac{\alpha}{\beta}C_1(t^* - t) }^{\frac{\beta}{\alpha}}. \qedhere
\end{equation*}
\end{proof}
\begin{proof}[Proof of Theorem~\ref{thm_1}]
The first statement of Proposition \ref{prop:moment_estimates} shows that $M_0[u(t,\cdot)]$ is of order $t^{-\alpha}$. Hence, the average cluster size $\ell(t)$ defined in~\eqref{e:average_cluster_size} translates to the time rescaled moment $\rho / M_0[u(\tau(t),\cdot)]$ and becomes $\rho\, \tau(t)^\alpha$. With this, Theorem \ref{thm_1} is a direct consequence of Proposition~\ref{prop:tau_estimates}. 
\end{proof}

\section{Scaling limit from discrete to continuum} \label{S.scaling_limit}

\subsection{Solutions to the continuum equation}\label{Ss.NP_properties}

First, we give the explicit construction of the fundamental solution of the problem~\eqref{NP'}. We emphasize that in this subsection the value of $\lambda$ can be taken in the range $\lambda \in (-\infty,2)$. We make a change of variables that  transforms the operator $\cL_\lambda$ in~\eqref{NP'} into the generator of the Bessel process, see \cite{Lawler2018}. For this we define the new variable
\begin{equation*}
z(x) = \int_0^x \frac{1}{\sqrt{a_\lambda(y)}}\dx{y}  = \frac{2}{2-\lambda} x^{1-\frac{\lambda}{2}}, \qquad\text{whence}\qquad x(z) = \bra*{\frac{2-\lambda}{2} z}^{\frac{1}{1-\frac{\lambda}{2}}} .
\end{equation*}
Then if $\varphi(t,x)$ is a solution to equation \eqref{NP'}, the function $\tilde{\varphi}$ defined by $\tilde{\varphi}(2t,z(x)) = \varphi(t,x)$ solves the equation
\begin{align} \label{TNP}
\del_t \tilde{\varphi}= \tfrac{1}{2} \del_z^2 \tilde{\varphi} + \tilde{a}_\lambda \del_z \tilde{\varphi} \qquad\text{and}\qquad \del_z \tilde{\varphi}|_{z=0} = 0 , \tag{TNP}
\end{align}
where
\begin{align*}
\tilde{a}_\lambda(z(x)) = \frac{a_\lambda'(x)}{4\sqrt{a_\lambda(x)}} = \frac{\lambda}{4} x(z)^{\frac{\lambda}{2}-1} = \frac{c_\lambda}{z} \qquad\text{with}\qquad c_\lambda = \frac{\lambda}{2(2-\lambda)} \in \bra*{ -\frac{1}{2}, \infty} .
\end{align*}
Hence, the equation~\eqref{TNP} becomes the generator of the reflected Bessel process~\cite[p. 10]{Lawler2018} of dimension $2c_\lambda+1$.
By comparison with~\cite[Chapter 3]{Lawler2018}, the fundamental solution is explicitly given by
\begin{align*}
\tilde{\Psi}_{c_\lambda}(t,z,y) = \frac{y^{2c_\lambda}}{t^{c_\lambda + \frac{1}{2}}}\exp\bra*{-\frac{z^2+y^2}{2t}} h_{c_\lambda}\bra*{\frac{zy}{t} } . 
\end{align*}
Here, $h_\nu$ is an entire function that can be expressed in terms of the modified Bessel function of the first kind $I_\nu(z) = z^\nu h_{\nu + \frac{1}{2}}(z)$. We have $c_\lambda + \frac{1}{2} = \alpha$ and $2c_\lambda = \lambda \alpha$, which allows to rewrite $\tilde{\Psi}$ as 
\begin{align}\label{e:fundamental_Bessel}
\tilde{\Psi}(t,z,y) = \frac{y^{\lambda \alpha}}{t^{\alpha}}\exp \bra*{-\frac{z^2+y^2}{2t} }h_{c_\lambda}\bra*{\frac{zy}{t} }.
\end{align}
\begin{rem}\label{rem:NP:bc}
 As noted in~\cite[Section 3]{Lawler2018}, the fundamental solution~\eqref{e:fundamental_Bessel} to~\eqref{TNP} with Neumann (reflecting) boundary condition agrees in the range $c_{\lambda}\geq \sfrac{1}{2}$ to the one with Dirichlet (absorbing) boundary conditions, which has the stochastic interpretation that both boundary conditions are in this case non-effective since the process cannot reach $0$ in finite time~\cite[Proposition 1]{Lawler2018}. The range $c_{\lambda}\in \bra[\big]{-\sfrac{1}{2},\sfrac{1}{2}}$ translates to $\lambda < 1$, whereas $c_{\lambda}\geq \sfrac{1}{2}$ is the range $\lambda \in [1,2)$. 
\end{rem}
Next we want to transform back to the equation~\eqref{NP'}. Because $\tilde{\Psi}_{c_\lambda}(t,\cdot,y) \to \delta_y$ for $t \to 0$, we arrive for all smooth $f$ at the identity
\begin{align*}
\int_0^\infty \tilde{\Psi}_{c_\lambda}(t,z(x),y) f(x) \dx{x} &= \int_0^\infty \frac{1}{z'(x(z))} \tilde{\Psi}_{c_\lambda}(t,z,y)  f(x(z))  \dx{z} \\
&\to \frac{1}{z'(x(y))} f(x(y)) = \bra*{\frac{2-\lambda}{2} }^{\lambda \alpha}y^{\lambda \alpha} f(x(y)), \quad\text{as } t \to 0.
\end{align*}
Since we want the fundamental solution $\Psi_\lambda$ for equation \eqref{NP'} to converge to a Dirac mass as $t \to 0$, we transform the equation back, normalize accordingly and end up with the definition
\begin{align}\label{fund_sol}
\Psi_\lambda(t,x,y) &= \bra*{2\alpha}^{\lambda\alpha}z(y)^{-\lambda \alpha} \tilde{\Psi}_{c_\lambda}(2t,z(x),z(y)) \notag \\ 
&= \bra*{\frac{2}{2-\lambda} }^{\lambda \alpha} \frac{1}{(2t)^{\alpha}} \exp \bra*{-\frac{z(x)^2+z(y)^2}{4t} }h_{c_\lambda}\bra*{\frac{z(x)z(y)}{2t}}.
\end{align}
By consulting~\cite[(14)]{Lawler2018}, we see that 
\[
  h_{c_\lambda}(0)=\frac{1}{2^{c_\lambda-\sfrac{1}{2}} \Gamma\bra*{c_\lambda +\sfrac{1}{2}}} = \frac{1}{2^{(\lambda-1)\alpha}\Gamma\bra*{\alpha}}
\]
and can rewrite the normalization constant~\eqref{e:def:Zlambda} of the scaling profile~\eqref{U_profile} as
\[
  Z_\lambda = \alpha^{-2\alpha} \Gamma\bra*{\alpha+1} = \alpha^{-2\alpha+1} \Gamma(\alpha) = \alpha^{\lambda \alpha} \Gamma(\alpha) . 
\]
Hence, for $y = 0$ we arrive at the scaling solution~\eqref{scaling_solution}. Also, by definition, $\Psi_\lambda(0,\cdot,y) = \delta_y$ and $\Psi(\cdot,\cdot,y)$ is a solution to equation~\eqref{NP'}. In this explicit form it is easy to verify basic properties of the fundamental solution.
\begin{prop}[Fundamental solution] \label{prop:fund_sol}
For every $\lambda \in [0,2)$ the function $\Psi_\lambda$ defined by~\eqref{fund_sol} has the following properties:
\begin{enumerate}
\item $\Psi_\lambda \in C^\infty(\R_+^3) \cap C^0(\R_+ \times \overline{\R}_+^2)$.
\item $\Psi_\lambda(t,x,y) = \Psi_\lambda(t,y,x)$. 
\item For every $y \in \overline{\R}_+$ holds $\del_t \Psi_\lambda(t,\cdot,y) - \mathcal{L}_{\lambda}\Psi_\lambda(t,\cdot,y) = 0$ and $a_\lambda \del_x \Psi_\lambda(t,\cdot,y)|_{x=0} = 0$.
\item It holds the normalization property
\begin{align*}
\int_{\R_+} \Psi_\lambda(t,x,y) \dx{x} = 1.
\end{align*}
\item It holds $\Psi_\lambda(t,\cdot,y) \rightharpoonup \delta_y$ in $\mathcal{M}(\overline{\R}_+)$ as $t \to 0$.
\item For all $k \geq 0$ it holds $\mathcal{L}_{\lambda,1}^{(k)}\Psi_\lambda(t,x,y) = \mathcal{L}_{\lambda,2}^{(k)}\Psi_\lambda(t,x,y)$, where $\mathcal{L}_{\lambda,i}^{(k)}$ denotes the $k$-fold composition of $\mathcal{L}_\lambda$ applied to the $i$-th spatial variable for $i=1,2$.
\end{enumerate}
\end{prop}
\begin{proof}
Properties \emph{(1)--(5)} are easily verified based on the above calculations. The last property is implied by the properties \emph{(1)--(3)}. Indeed, property 3 states that 
\begin{align*}
\del_t \Psi_\lambda(t,x,y) &= \mathcal{L}_{\lambda}\Psi_\lambda(t,\cdot,y)|_x  = a_\lambda(x)\del_1^2 \Psi_\lambda(t,x,y) + a_\lambda'(x)\del_1 \Psi_\lambda(t,x,y) \\
&= \mathcal{L}_{\lambda,1}\Psi_\lambda(t,x,y).
\end{align*}
Using the symmetry of $\Psi_\lambda$ we also have
\begin{align*}
\del_t \Psi_\lambda(t,x,y) &= \del_t \Psi_\lambda(t,y,x) = \mathcal{L}_{\lambda}\Psi_\lambda(t,\cdot,x)|_y = \mathcal{L}_{\lambda}\Psi_\lambda(t,x,\cdot)|_y \\
 &= a_\lambda(y)\del_2^2 \Psi_\lambda(t,x,y) + a_\lambda'(y)\del_2 \Psi_\lambda(t,x,y) = \mathcal{L}_{\lambda,2}\Psi_\lambda(t,x,y),
\end{align*}
which implies the statement for $k = 1$. The rest of the statement follows easily by induction, since the function $\mathcal{L}_{\lambda,1}\Psi_\lambda(t,x,y)$ is also symmetric and solves the same equation as $\Psi_\lambda$.
\end{proof}
Proposition~\ref{prop:fund_sol} motivates to define for $g \in C_c^\infty(\R_+)$ the time-evolution $\cS_\lambda(t)g$ by using the fundamental solution as integral kernel, i.e
\begin{equation}\label{e:def:Slambda}
 \cS_\lambda(t)g = \int_{\R_+}\Psi_\lambda(t,\cdot,y)g(y) \dx{y} .
\end{equation}
For this, we deduce the following properties.
\begin{cor} \label{cor:semigroup_estimates}
For any $g \in C_c^\infty(\R_+)$, $\cS_\lambda(t) g$ from~\eqref{e:def:Slambda} is a solution of equation \eqref{NP'} with initial data $g$. Furthermore, the following estimates hold:
\begin{enumerate}
\item For all $p \in [1,\infty]$, $k \geq 0$ and $t \geq 0$ it holds
\begin{align*}
\norm{\mathcal{L}_{\lambda}^{(k)}\cS_\lambda(t)g}_p \leq \norm{\mathcal{L}_{\lambda}^{(k)}g}_p.
\end{align*}
\item For all $\nu \geq 0$, $k \geq 0$ and $t \geq 0$ it holds
\begin{align*}
\norm{x^\nu \mathcal{L}_{\lambda}^{(k)}\cS_\lambda(t)g}_\infty \lesssim \begin{cases}
(1 \vee t)^{\alpha(\nu - 1)}\norm{\mathcal{L}_{\lambda}^{(k)}g}_1  + \norm{\mathcal{L}_{\lambda}^{(k)}g}_\infty + \norm{x^\nu \mathcal{L}_{\lambda}^{(k)}g}_\infty, &\text{if} \ \nu < 1, \\
t^{\alpha(\nu - 1)}\norm{\mathcal{L}_{\lambda}^{(k)}g}_1 + \norm{x^\nu \mathcal{L}_{\lambda}^{(k)}g}_\infty, &\text{if} \ \nu \geq 1.
\end{cases}
\end{align*}
\end{enumerate}
\end{cor}
\begin{proof}
Using the properties of Proposition \ref{prop:fund_sol} and the fact that $g \in C_c^\infty(\R_+)$ it is easy to prove that $\cS_\lambda(t)g$ is a solution to equation \eqref{NP'} with initial data $g$. Also, since 
\begin{align*}
\mathcal{L}_{\lambda}^{(k)}\cS_\lambda(t)g &= \int_{\R_+} \mathcal{L}_{\lambda,1}^{(k)}\Psi_\lambda(t,x,y)g(y) \dx{y} = \int_{\R_+} \mathcal{L}_{\lambda,2}^{(k)}\Psi_\lambda(t,x,y)g(y) \dx{y} \\
&= \int_{\R_+} \Psi_\lambda(t,x,y)\mathcal{L}_{\lambda}^{(k)}g(y) \dx{y} = \cS_\lambda(t) \mathcal{L}_{\lambda}^{(k)}g,
\end{align*}
and $\mathcal{L}_{\lambda}^{(k)}g \in C_c^\infty(\R)$ if $g \in C_c^\infty(\R_+)$, it suffices to prove all desired inequalities for $k = 0$. For the first inequality we have
\begin{align*}
\norm{\cS_\lambda(t)g}_p^p &= \int_{\R_+} \abs[\bigg]{\int_{\R_+} \Psi_\lambda(t,x,y)g(y) \dx{y} }^p \dx{x} \leq \int_{\R_+} \int_{\R_+}\Psi_\lambda(t,x,y) |g(y)|^p \dx{y}  \dx{x} \\
&= \int_{\R_+} |g(y)|^p \int_{\R_+}\Psi_\lambda(t,x,y)  \dx{x}  \dx{y} = \int_{\R_+} |g(y)|^p \dx{y}.
\end{align*}
Here we used Jensen's inequality with respect to the probability measure $\Psi_\lambda(t,x,\cdot)$. For the second inequality we split the integral
\begin{align*}
|x^\nu \cS_\lambda(t)g(x)| \leq \int_{\R_+}x^\nu \Psi_\lambda(t,x,y)|g(y)| \dx{y} \leq \int_{0}^{rx} x^\nu \Psi_\lambda(t,x,y)|g(y)| \dx{y} + r^{-\nu}\norm{x^\nu g}_\infty.
\end{align*}
Using the explicit form \eqref{fund_sol} and the asymptotics of Bessel functions, one can show that for some $r > 0$ small enough we have for $y \leq r x$ the bound
\begin{align*}
\Psi_\lambda(t,x,y) \leq Z_\lambda^{-1} t^{-\alpha}\exp(-c_r t^{-1}x^{2-\lambda}) =: t^{-\alpha}F_\lambda(t^{-\alpha}x).
\end{align*}
Now a simple calculation shows that the maximum of the function $x \to x^\nu F_\lambda(t^{-\alpha}x)$ is attained at $x$ of order $t^\alpha$, hence for $y \leq r x$ the estimate $x^\nu \Psi_\lambda(t,x,y) \lesssim t^{\alpha(\nu-1)}$ holds. This directly implies the desired estimate in the case $\nu \geq 1$. If $\nu < 1$, we want an estimate that does not blow up at $t=0$. Here we use that $\norm{x^\nu f}_\infty \leq \norm{f}_\infty + \sup_{x \geq 1}|x^\nu f|$. On $[1,\infty)$, the function $x \to x^\nu F_\lambda(t^{-\alpha}x)$ attains its maximum at $x$ of order $1 \vee t^\alpha$, and finally
\begin{equation*}
\sup_{t \in \R_+} t^{-\alpha}F_\lambda(t^{-\alpha}) < \infty . \qedhere
\end{equation*}
\end{proof}
To analyze the relation between the discrete and the continuous model, we have to work with a weak formulation of equation \eqref{NP'}, which is based on the adjoint equation~\eqref{NP_rhs}. Thus, we define for $f \in C_c^\infty ((0,T) \times \R_+)$ the solution operator $\cT_\lambda (t) f$ by
\begin{align*}
\cT_\lambda(t) f &= \int_0^t \cS_\lambda(t-s)f(s,\cdot) \dx{s}.
\end{align*}
Note that $\varphi(t,\cdot) = \mathcal{T}_\lambda(t) f$ is a solution to the inhomogeneous equation 
\begin{align} \label{NP_rhs}
\begin{cases}
\del_t \varphi - \mathcal{L}_\lambda \varphi = f, \quad\text{ on } \R_+\times\R_+\\
a_\lambda \del_x \varphi|_{x=0} = 0, \quad\text{ on } \R_+ ,\\
\varphi(0,\cdot) = 0, \quad\text{ on } \R_+ .
\end{cases}
\end{align}
Therewith, the definition of weak solutions reads as follows.
\begin{defn}\label{defn:weak_solution}
For $T > 0$, a family of measures $\{\mu_t\}_{t \in [0,T)} \subset \mathcal{M}(\overline{\R}_+)$ is a weak solution to equation \eqref{NP'} on $[0,T)$ with initial data $\mu_0\in \cM(\overline{\R}_+)$ if for all $f \in C_c^\infty ((0,T) \times \R_+)$ it holds
\begin{align} \label{weak_formulation_1}
\int_0^T \int_{\overline{\R}_+} f(T-t,x) \dx\mu_t(x) \dx{t} =  \int_{\overline{\R}_+} \bra*{\mathcal{T}_\lambda(T)f}(x) \dx\mu_0(x).
\end{align}
\end{defn}
Equivalently, since $\mathcal{T}_\lambda(t)f$ solves the inhomogeneous equation \eqref{NP_rhs}, $\mu_t$ is a weak solution if and only if 
\begin{align} \label{weak_formulation_2}
\int_0^T \int_{\overline{\R}_+} \bra*{\del_t \varphi + \mathcal{L}_\lambda \varphi} \dx\mu_t(x) \dx{t} = - \int_{\overline{\R}_+} \varphi(0,x) \dx\mu_0(x),
\end{align}
for all $\varphi(t,\cdot) = \mathcal{T}_\lambda(T-t)f$ with $f \in C_c^\infty ((0,T) \times \R_+)$. The definition~\eqref{weak_formulation_2} looks more like standard weak formulations of PDE. However, we specify the test function class~$\varphi$ only in terms of the image of the adjoint operator on smooth functions. The reason for this is that \eqref{weak_formulation_1} automatically implies that weak solutions are unique as distributions on $(0,T) \times \R_+$, which is needed to identify the limit of a sequence of approximate solutions (see next subsection). By Corollary \ref{cor:semigroup_estimates}, the class of test functions has good regularity and decay properties. With these, it is easy to verify that the scaling solution $\gamma_\lambda$~\eqref{scaling_solution} with $\cG_\lambda$ given in~\eqref{U_profile} solves~\eqref{NP'} in the weak sense with initial data $\delta_0$.

We close this subsection with the observation that the scaling solution $\gamma_\lambda$ is indeed also attractive for all solutions in relative entropy, as in the classical result for the heat equation with $\lambda=0$. 
\begin{rem}\label{rem:NP'longtime}
  Let $\mu(t)$ be a solution to equation~\eqref{NP'} starting from some $\mu^{(0)}$ with mass $\rho>0$. 
  Then the relative entropy of $\mu(t)$ with respect to $\rho\,\gamma_\lambda(t)$ is dissipated, which follows from the simple calculation 
  \begin{equation}
    \pderiv{}{t} \cH\bigl(\mu(t) \mid \rho\,\gamma_\lambda(t)\bigr) = - \cI_\lambda\bigl(\mu(t) \mid \rho\,\gamma_\lambda(t)\bigr) = - \int a_\lambda \partial_x \log \frac{\mu(t)}{\rho\,\gamma_\lambda(t)} \partial_x \frac{\mu(t)}{\rho\,\gamma_\lambda(t)} \dx{\gamma_\lambda(t)}, \label{e:free_energy_est}
  \end{equation}
once sufficient regularity is established for any $t>0$.  In Appendix~\ref{S.lsi} we prove that the following weighted logarithmic Sobolev inequality holds: For all $\lambda\in [0,2)$, there exists $C_{\LSI} = C_{\LSI}(\lambda)$ such that for any $t>0$ and any measure $\mu\in \cM(\R_+)$ with mass $\rho>0$ and $\cH(\mu \mid \rho\,\gamma_\lambda(t)) <\infty$  the inequality
\begin{equation}\label{e:LSI}
  \cH\bigl(\mu\mid \rho\,\gamma_\lambda(t)\bigr) \leq 4C_{\LSI} \, t \, \cI_{\lambda}\bigl(\mu \mid \rho\,\gamma_\lambda(t)\bigr)
\end{equation}
holds. Hence, once sufficient regularity for solutions to~\eqref{NP'} is established, it immediately follows that those converge to the self-similar profile $\rho\,\gamma_{\lambda}(t)$ in relative entropy and hence also in $L^1(\overline\R_+)$ by the Pinsker inequality.
  
The argument suggests that solutions to the discrete equation \eqref{NP} also 
get close to the continuum equation \eqref{NP'} in the limit $t\to \infty$. The weighted logarithmic Sobolev inequality suggests that the entropy method after~\cite{Yau1991} might be applicable as well. However, in this work we opted for a more classical approach based on the Nash inequality (Proposition~\ref{prop:discrete_nash}) and the resulting Nash continuity estimates (Theorem~\ref{thm:NP}).
\end{rem}

\subsection{Strategy and proof of Theorem \ref{thm_discrete_to_continuous}} \label{Ss.proof_scaling_limit}

Let $U_\varepsilon$ be a sequence of solutions to equation \eqref{NP} with $\sup_{0 < \varepsilon \leq 1} \norm{U_{0,\varepsilon}}_1 < \infty$ and $\mathcal{U}_\varepsilon$ be the associated sequence of approximate solutions as in \eqref{approximate_solution}. To see that $\mathcal{U}_\varepsilon$ converges to a solution to equation \eqref{NP'}, let $T > 0$ and $\varphi(t,\cdot) = \mathcal{T}_\lambda(T-t)f$, $f \in C_c^\infty ((0,T) \times \R_+)$. Multiplying $\mathcal{U}_\varepsilon$ with $\del_t \varphi$ and integrating over space-time, we get
\begin{align*}
\int_0^T \int_0^\infty \mathcal{U}_\varepsilon \del_t \varphi \dx{x} \dx{t} &= \varepsilon^{-\alpha}\int_0^T \sum_{k=1}^\infty U_\varepsilon(\varepsilon^{-1}t,k) \pi_\varepsilon \del_t \varphi(t,k) \dx{t} \\
&= \varepsilon^{-\alpha}\int_0^T \sum_{k=1}^\infty U_\varepsilon(\varepsilon^{-1}t,k) \del_t \pi_\varepsilon \varphi(t,k) \dx{t}.
\end{align*}
Integrating by parts in time, using the equation for $U$ and the symmetry of $L_\lambda$, we arrive at
\begin{align} \label{weak_formulation_calculation}
\int_0^T \int_0^\infty \mathcal{U}_\varepsilon \del_t \varphi \dx{x} \dx{t} = &-\int_0^\infty \mathcal{U}_\varepsilon(0,x) \varphi(0,x) \dx{x} \notag \\
 &- \varepsilon^{-\alpha}\int_0^T \sum_{k=1}^\infty U_\varepsilon(\varepsilon^{-1}t,k) \varepsilon^{-1} L_\lambda \pi_\varepsilon \varphi (t,k) \dx{t},
\end{align}
with $\pi_\varepsilon$ as in \eqref{projection}. To relate the above identity to the weak formulation of equation~\eqref{weak_formulation_2}, we need to express the last line in terms of $\mathcal{L}_\lambda \varphi$. Adding and subtracting the term $\pi_\varepsilon \cL_\lambda \varphi$ in the last summation and using the identity
\begin{align*}
\varepsilon^{-\alpha}\sum_{k=1}^\infty U_\varepsilon(\varepsilon^{-1}t,k) \pi_\varepsilon \mathcal{L}_\lambda \varphi (t,k) = \int_0^\infty \mathcal{U}_\varepsilon \, \mathcal{L}_\lambda \varphi \dx{x} ,
\end{align*}
we get
\begin{align*}
\varepsilon^{-\alpha} \sum_{k=1}^\infty U_\varepsilon(\varepsilon^{-1}t,k) \varepsilon^{-1} L_\lambda \pi_\varepsilon \varphi (t,k)  =  \int_0^\infty \mathcal{U}_\varepsilon \, \mathcal{L}_\lambda \varphi \dx{x} + \varepsilon^{-\alpha}\sum_{k=1}^\infty U_\varepsilon(\varepsilon^{-1}t,k) \cR_\varepsilon(\varphi,k),
\end{align*}
where 
\begin{align} \label{e:defect}
\cR_\varepsilon(\varphi,k) = \varepsilon^{-1} L_\lambda \pi_\varepsilon \varphi(k) - \pi_\varepsilon \mathcal{L}_\lambda \varphi(k)
\end{align}
denotes the defect between the discrete and continuous operator. The crucial ingredient for the proof is the following estimate on the defect which shows that the rescaled discrete operator can be replaced with the continuous operator on functions that are regular enough.
\begin{lem}[Replacement lemma] \label{lem:replacement}
Let $\varphi \in C^0(\overline{\R}_+) \cap C^3(\R_+)$ with the following properties:
\begin{enumerate}
\item The map $x \mapsto a_\lambda(x)\del_x \varphi(x)$ is Lipschitz-continuous on $\overline{\R}_+$.
\item The boundary condition $a_\lambda \del_x \varphi|_{x=0} = 0$ is satisfied.
\end{enumerate}
Then the following estimates hold:
\begin{align}
\norm{\varepsilon^{-1} L_\lambda \pi_\varepsilon \varphi}_\infty &\lesssim \norm{\mathcal{L}_\lambda \varphi}_\infty \varepsilon^\alpha, \\
|\mathcal{R}_\varepsilon(\varphi,k)| &\lesssim  \varepsilon^\alpha \int_{(k-2)\varepsilon^{\alpha}}^{(k+1)\varepsilon^{\alpha}} \bra*{ x^{\lambda - 1} |\del_x^2 \varphi| + x^\lambda |\del_x^3 \varphi| } \dx{x}, \qquad \text{for} \ k \geq 3. \label{e:bound:replacement}
\end{align}
\end{lem}
The above result together with the previous calculations yields that rescaled solutions of the discrete problem~\eqref{NP} are approximate solutions of the continuous equation~\eqref{NP'}.
\begin{prop}[Approximate weak solutions] \label{prop:approximate_weak_formulation}
Let $U_\varepsilon$ and $\mathcal{U}_\varepsilon$ be as above. Then for $\varphi(t,\cdot) = \mathcal{T}_\lambda(T-t)f$ with $f \in C_c^\infty ((0,T) \times \R_+)$ it holds
\begin{align*}
\int_0^T \int_0^\infty \mathcal{U}_\varepsilon \, \bra[\big]{\del_t \varphi + \mathcal{L}_\lambda \varphi} \dx{x} \dx{t} = &-\int_0^\infty \mathcal{U}_\varepsilon(0,x) \varphi(0,x) \dx{x} + \mathrm{o}(1) \qquad\text{as }\varepsilon \to 0,
\end{align*}
where the terms in $\mathrm{o}(1)$ depend on $T, f$ and the bound on $U_{0,\varepsilon}$.
\end{prop}
The full rigorous proof of Proposition \ref{prop:approximate_weak_formulation} is given at the end of this section. To make use of Lemma~\ref{lem:replacement} we need regularity estimates for $\varphi$, as the error term $\mathcal{R}_\varepsilon$ contains derivatives up to third order. However, because of the degeneracy of $a_\lambda$ the higher derivatives blow up at $0$. This can be resolved by introducing a small-scale boundary region at $0$ where the error term vanishes in the limit thanks to the uniform bound~\eqref{e:uniform_decay:Linfty} on  $\mathcal{U}_\varepsilon$ from Theorem~\ref{thm:NP}. 

To pass to the limit in the approximate weak formulation we have to establish compactness in a suitable topology. The  scale-invariant estimates from Section~\ref{Ss.Continuity} in fact imply boundedness and equicontinuity on compact sets, which yields compactness with respect to (local) uniform convergence.
\begin{prop}[Compactness] \label{prop:equicontinuity}
Let $U_\varepsilon$ and $\mathcal{U}_\varepsilon$ be as above. Then for all $x,y \in \overline{\R}_+$ and $0 < s < t$ it holds
\begin{align*}
\norm{\mathcal{U}_\varepsilon(t,\cdot)}_\infty &\lesssim \norm{U_{0,\varepsilon}}_1 t^{-\alpha}, \\
|\mathcal{U}_\varepsilon(t,x) - \mathcal{U}_\varepsilon(t,y)| &\lesssim t^{-\alpha}\norm{U_{0,\varepsilon}}_1 \left(|\theta_\lambda(t^{-\alpha}x)-\theta_\lambda(t^{-\alpha}y)|^\frac{1}{2} + \Xi_{\lambda,\varepsilon}(t,x,y) \right), \\
|\mathcal{U}_\varepsilon(t,x)-\mathcal{U}_\varepsilon(s,x)| &\lesssim s^{-\alpha}\norm{U_{0,\varepsilon}}_1 \omega_\lambda(t,s),
\end{align*}
where $\theta_\lambda, \omega_\lambda$ are as in Lemma \ref{lem:Nash_continuity} and $\Xi_{\lambda,\varepsilon} \to 0$ as $\varepsilon \to 0$ locally uniformly on $\R_+ \times \overline{\R}_+^2$ in the case $\lambda < 1$ and locally uniformly on $\R_+^3$ in the case $\lambda \geq 1$.
\end{prop}
\begin{proof}[Proof of Proposition \ref{prop:equicontinuity}]
This result is an easy consequence of Theorem~\ref{thm:NP}. The $L^\infty$ bound and continuity estimate for the time variable are immediate from~\eqref{e:uniform_decay:Linfty}, respectively~\eqref{e:Nash_continuity:time} in Theorem~\ref{thm:NP}. For continuity in space we apply~\eqref{e:Nash_continuity:space} with $x_\varepsilon = \varepsilon^\alpha (\lfloor \varepsilon^{-\alpha}x \rfloor + 1)$, $y_\varepsilon = \varepsilon^\alpha (\lfloor \varepsilon^{-\alpha}y \rfloor + 1)$ and obtain
\begin{align*}
|\mathcal{U}_\varepsilon(t,x)-\mathcal{U}_\varepsilon(t,y)| &\lesssim \norm{U_0}_1 t^{-\alpha} \left|\theta_\lambda(t^{-\alpha}x_\varepsilon) - \theta_\lambda(t^{-\alpha}y_\varepsilon) \right|^\frac{1}{2}\\
&\leq \norm{U_0}_1 t^{-\alpha} \left|\theta_\lambda(t^{-\alpha}x) - \theta_\lambda(t^{-\alpha}y) \right|^\frac{1}{2} \\
&\quad+ \norm{U_0}_1 t^{-\alpha} \left(|\theta_\lambda(t^{-\alpha}x)-\theta_\lambda(t^{-\alpha}x_\varepsilon)|^\frac{1}{2} + |\theta_\lambda(t^{-\alpha}y)-\theta_\lambda(t^{-\alpha}y_\varepsilon)|^\frac{1}{2} \right).
\end{align*}
Note that we have $|x-x_\varepsilon| \lesssim \varepsilon^\alpha$, $|y-y_\varepsilon| \lesssim \varepsilon^\alpha$. Thus in the case $0 < \lambda < 1$ the H\"older continuity of $\theta_\lambda$ from~\eqref{e:def:theta_lambda} implies
\begin{align*}
|\theta_\lambda(t^{-\alpha}x)-\theta_\lambda(t^{-\alpha}x_\varepsilon)| \lesssim \theta_\lambda(t^{-\alpha}\varepsilon^\alpha),
\end{align*}
where the right-hand side does not depend on $x$, whereas for $1 \leq \lambda < 2$ we have
\begin{align*}
|\theta_\lambda(t^{-\alpha}x)-\theta_\lambda(t^{-\alpha}x_\varepsilon)| \lesssim \left|\int_{t^{-\alpha}x}^{t^{-\alpha}x_\varepsilon} \theta_\lambda'(\xi) \dx\xi \right|,
\end{align*}
which goes to zero locally uniformly for $t,x > 0$. The same line of reasoning applies to $y$ and $y_\varepsilon$, which finishes the proof.
\end{proof}
Taking the above statements for granted, the convergence result for $\mathcal{U}_\varepsilon$ easily follows.
\begin{proof}[Proof of Theorem \ref{thm_discrete_to_continuous}]
 It is easy to check that by Proposition \ref{prop:equicontinuity} the sequence $\mathcal{U}_\varepsilon$ satisfies the assumptions of the Arzela-Ascoli Theorem for discontinuous functions (cf.~Proposition \ref{prop:arzela_ascoli}) on each compact subset of $\R_+ \times \overline{\R}_+$ in the case $0 \leq \lambda < 1$, respectively $ \R_+^2$ in the case $1 \leq \lambda < 2$. Thus by exhaustion with compact sets and a diagonal argument each sequence $\varepsilon \to 0$ has a subsequence (not relabeled) such that $\mathcal{U}_\varepsilon \to \mathcal{U}$ locally uniformly for some function $\mathcal{U} \in C^0(\R_+ \times \overline{\R}_+)$, respectively $C^0(\R_+^2)$. To identify the limit, let $T > 0$ and $f \in C_c^\infty((0,T)\times \R_+)$. Then by Proposition \ref{prop:approximate_weak_formulation}, applied with $\varphi(t,\cdot) = \mathcal{T}_\lambda(T-t)f$, we have that 
\begin{align*}
\int_0^T \int_{\R_+} f(T-t,x)\, \mathcal{U}_\varepsilon(t,x) \dx{t} \dx{x} = \int_{\R_+}\mathcal{T}_\lambda(T)f(x) \,\mathcal{U}_\varepsilon(0,x) \dx{x} + \mathrm{o}(1), \qquad\text{as } \varepsilon \to 0.
\end{align*}
Letting $\varepsilon \to 0$ and using that $\mathcal{U}_\varepsilon(0,\cdot) \rightharpoonup \mu_0$ we arrive at
\begin{align*}
\int_0^T \int_{\R_+} f(T-t,x) \, \mathcal{U}(t,x) \dx{t} \dx{x} = \int_{\R_+}\mathcal{T}_\lambda(T)f(x) \dx\mu_0(x).
\end{align*} 
Thus $\mathcal{U}(t,x)$ is a weak solution of equation \eqref{NP'} with initial data $\mu_0$ and because of continuity it is unique on $\R_+ \times \overline{\R}_+$, respectively $\R_+^2$, which in turn implies that the convergence holds for every sequence $\varepsilon \to 0$. Thus in the case $0 \leq \lambda < 1$, the limit is completely characterized, whereas for $1 \leq \lambda < 2$ we cannot identify the limit at $x=0$ but only have boundedness of $\mathcal{U}_\varepsilon(t,0)$ for $t > 0$ by Proposition \ref{prop:equicontinuity}.
\end{proof}
It remains to prove Lemma~\ref{lem:replacement} and Proposition~\ref{prop:approximate_weak_formulation}, which is done in the next two subsections.

\subsection{Replacement lemma} \label{Ss.replacement}

We split the proof of Lemma \ref{lem:replacement} into several steps. 
\begin{lem} \label{lem:repl_derivative_bounds}
Let $\varphi$ be as in Lemma \ref{lem:replacement}. Then, it holds
\begin{align}
|\del_x \varphi|(x) &\leq \norm{\mathcal{L}_\lambda \varphi}_\infty x^{1-\lambda}, \label{bound_1} \\
|\del_x^2 \varphi(x)| &\leq 2\norm{\mathcal{L}_\lambda \varphi}_\infty x^{-\lambda}. \label{bound_2}
\end{align}
\end{lem}
\begin{proof}
Because $a_\lambda \del_x \varphi$ is Lipschitz and equal to zero at the boundary, we have
\begin{align*}
|a_\lambda(x)\del_x \varphi(x)| \leq \norm{\del_x(a_\lambda \del_x \varphi)}_\infty x,
\end{align*}
which gives the first statement after dividing by $a_\lambda$. This estimate then directly implies that $a_\lambda' \del_x \varphi$ is bounded by $\norm{\del_x(a_\lambda \del_x \varphi)}_\infty$, and by Leibniz rule
\begin{equation*}
|a_\lambda(x)\del_x^2 \varphi(x)| \leq |\del_x(a_\lambda \del_x \varphi)(x)| + |a_\lambda'(x) \del_x \varphi(x)| \leq 2\norm{\del_x(a_\lambda \del_x \varphi)}_\infty. \qedhere
\end{equation*}
\end{proof}
\begin{lem} \label{repl_lem_1}
Let $\varphi$ be as in Lemma \ref{lem:replacement}. Then, it holds
\begin{align*}
\norm{\varepsilon^{-1} L_\lambda \pi_\varepsilon \varphi}_\infty \lesssim \norm{\mathcal{L}_\lambda \varphi}_\infty \varepsilon^\alpha.
\end{align*}
\end{lem}
\begin{proof}
First we consider the case $k \geq 3$. Writing out the term we get
\begin{align*}
\varepsilon^{-1} L_\lambda \pi_\varepsilon \varphi(k) &= \varepsilon^{-1}\bra*{a_\lambda(k)\del^+\pi_\varepsilon \varphi(k)- a_\lambda(k-1)\del^+\pi_\varepsilon \varphi(k-1)} \\
&= \varepsilon^{-2\alpha} \bra*{a_\lambda(k\varepsilon^\alpha)\del^+\pi_\varepsilon \varphi(k)- a_\lambda((k-1)\varepsilon^\alpha)\del^+\pi_\varepsilon \varphi(k-1)} \\
&= \varepsilon^{-2\alpha}\bra[\Big]{a_\lambda(k\varepsilon^\alpha)-a_\lambda((k-1)\varepsilon^\alpha)}\del^+\pi_\varepsilon \varphi(k-1) \\
&\quad + \varepsilon^{-2\alpha}a_\lambda(k\varepsilon^\alpha)\bra*{\del^+\pi_\varepsilon \varphi(k)-\del^+\pi_\varepsilon \varphi(k-1) } \\
&= \mathrm{I} + \mathrm{II},
\end{align*}
and one further calculates
\begin{align*}
\mathrm{I} &= \varepsilon^{-2\alpha}\bra[\big]{a_\lambda(k\varepsilon^\alpha)-a_\lambda((k-1)\varepsilon^\alpha)}\int_{(k-1)\varepsilon^\alpha}^{k\varepsilon^\alpha} \bra*{ \varphi(x) - \varphi(x-\varepsilon^\alpha)} \dx{x}, \\
\mathrm{II} &= \varepsilon^{-2\alpha}a_\lambda(k\varepsilon^\alpha)\int_{(k-1) \varepsilon^\alpha}^{k\varepsilon^\alpha} \bra[\big]{\varphi(x+\varepsilon^\alpha) -2\varphi(x) + \varphi(x-\varepsilon^\alpha) } \dx{x}.
\end{align*}
To estimate $\mathrm{I}$, we note that in the case $0 < \lambda < 1$ the mean value theorem and estimate \eqref{bound_1} imply the bound
\begin{align*}
|a_\lambda(k\varepsilon^\alpha)-a_\lambda((k-1)\varepsilon^\alpha)| &\lesssim \varepsilon^{\alpha} ((k-1)\varepsilon^\alpha)^{\lambda-1}, \\
|\varphi(x)-\varphi(x-\varepsilon^\alpha)| &\lesssim \norm*{\cL_\lambda \varphi}_\infty \varepsilon^\alpha (k\varepsilon^\alpha)^{1-\lambda},
\end{align*}
for any $k\geq 2$ and $x\in [(k-1)\eps^\alpha,k\eps^\alpha)$, while for $1 \leq \lambda < 2$ we have the estimate
\begin{align*}
\abs*{a_\lambda(k\varepsilon^\alpha)-a_\lambda((k-1)\varepsilon^\alpha)} &\lesssim \varepsilon^{\alpha} (k\varepsilon^\alpha)^{\lambda-1}, \\
|\varphi(x)-\varphi(x-\varepsilon^\alpha)| &\lesssim \norm*{\cL_\lambda \varphi}_\infty \varepsilon^\alpha ((k-2)\varepsilon^\alpha)^{1-\lambda},
\end{align*}
for $k \geq 3$. In both cases we get the desired estimate for $\mathrm{I}$. For the second term we apply a similar argument. Here, the estimate \eqref{bound_2} and Taylor expansion imply for $k \geq 3$ that 
\begin{align*}
|\varphi(x+\varepsilon^\alpha) -2\varphi(x) + \varphi(x-\varepsilon^\alpha)| \lesssim \norm*{\cL_\lambda \varphi}_\infty \varepsilon^{2\alpha} ((k-2)\varepsilon^\alpha)^{-\lambda},
\end{align*}
which then gives the correct estimate for $\mathrm{II}$. In the remaining cases $k \in \{1,2\}$ we have
\begin{equation*}
|a_\lambda(k\varepsilon^\alpha)| \lesssim \varepsilon^{\alpha\lambda}, \qquad\text{and}\qquad
|\varphi(x)-\varphi(x \pm \varepsilon^\alpha)| \lesssim \norm*{\cL_\lambda \varphi}_\infty \varepsilon^{\alpha(2-\lambda)},
\end{equation*}
where in the case $1 \leq \lambda < 2$ we use that \eqref{bound_1} implies Hölder continuity with exponent $2-\lambda$. Hence we have 
\begin{align*}
|\varepsilon^{-1}a_\lambda(k)\del^+\pi_\varepsilon \varphi(k)| = |\varepsilon^{-2\alpha}a_\lambda(\varepsilon
^\alpha k)\del^+\pi_\varepsilon \varphi(k)| \lesssim \norm*{\cL_\lambda \varphi}_\infty \varepsilon^\alpha,
\end{align*}
which finishes the proof.
\end{proof}
\begin{lem}[Taylor expansion]\label{lem:Taylor}
Let $\varphi$ be as in Lemma \ref{lem:replacement}. Then for $\varepsilon > 0$ and every $m \geq 0$ it holds
\begin{align*}
\pi_\varepsilon \varphi(\cdot \pm \varepsilon^\alpha) = \sum_{l=0}^m \frac{(\pm\varepsilon)^{l\alpha}}{l!} \pi_\varepsilon \del_x^{l} \varphi + R_m(\varphi, \pm \varepsilon),
\end{align*}
with
\begin{align*}
|R_m(\varphi,\pm \varepsilon)(k)| \leq \frac{\varepsilon^{(m+1)\alpha}}{(m+1)!}\int_{I_{\pm}^\varepsilon(k)} |\del^{m+1}_x \varphi(x)| \dx{x},
\end{align*}
and
\begin{align*}
I_\sigma^\eps(k) = \begin{cases}
[(k-1)\varepsilon^\alpha, (k+1)\varepsilon^\alpha), &\text{if } \sigma = +\varepsilon; \\
[(k-2)\varepsilon^\alpha, k\varepsilon^\alpha), &\text{if } \sigma = -\varepsilon.
\end{cases}
\end{align*}
\end{lem}
\begin{proof}
The statement follows directly by standard Taylor expansion, where we use the integral representation for the residual term
\begin{align*}
\MoveEqLeft{\pi_\varepsilon \varphi(\cdot + \varepsilon^\alpha)(k) = \int_{(k-1)\varepsilon^\alpha}^{k\varepsilon^\alpha} \varphi(x+\varepsilon^\alpha) \dx{x}} \\
&= \int_{(k-1)\varepsilon^\alpha}^{k\varepsilon^\alpha} \sum_{l=0}^m \frac{\varepsilon^{l\alpha}}{l!} \del_x^{l} \varphi(x) \dx{x} +\frac{1}{(m+1)!} \int_{(k-1)\varepsilon^\alpha}^{k\varepsilon^\alpha} \int_x^{x+\varepsilon^\alpha}\!\!\!(x+\varepsilon^\alpha - s)^m \del^{m+1}_{x}\varphi(s) \dx{s} \dx{x} \\
&= \sum_{l=0}^m \frac{\varepsilon^{l\alpha}}{l!} \pi_\varepsilon \del_x^{l} \varphi + R_m(\varphi,\varepsilon).
\end{align*}
We then calculate 
\begin{align*}
|R_m(\varphi,\varepsilon)(k)| &\leq \frac{\varepsilon^{m\alpha}}{(m+1)!} \int_{(k-1)\varepsilon^\alpha}^{k\varepsilon^\alpha} \int_x^{x+\varepsilon^\alpha} |\del^{m+1}_{x}\varphi(s)| \dx{s} \dx{x} \\
&\leq \frac{\varepsilon^{m\alpha}}{(m+1)!} \int_{(k-1)\varepsilon^\alpha}^{k\varepsilon^\alpha} \int_{(k-1)\varepsilon^\alpha}^{(k+1)\varepsilon^\alpha} |\del^{m+1}_{x}\varphi(s)| \dx{s} \dx{x} \\
&= \frac{\varepsilon^{(m+1)\alpha}}{(m+1)!}\int_{(k-1)\varepsilon^\alpha}^{(k+1)\varepsilon^\alpha} |\del^{m+1}_x \varphi(s)| \dx{s}.
\end{align*}
The calculation for $\varphi(\cdot - \varepsilon^\alpha)$ works similarly.
\end{proof}
With this preparation we can now prove Lemma \ref{lem:replacement}.
\begin{proof}[Proof of Lemma \ref{lem:replacement}]
Lemma \ref{repl_lem_1} proves the statement $\norm{\varepsilon^{-1} L_\lambda \pi_\varepsilon \varphi}_\infty \lesssim \norm{\mathcal{L}_\lambda \varphi}_\infty \varepsilon^\alpha$. 
Thus it remains to bound the difference $\mathcal{R}(k) = \varepsilon^{-1} L_\lambda \pi_\varepsilon \varphi(k) - \pi_\varepsilon \mathcal{L}_\lambda \varphi(k)$ for $k \geq 3$. By the fundamental theorem of calculus we have
\begin{align*}
\pi_\varepsilon \mathcal{L}_\lambda \varphi(k) &= \int_{(k-1)\varepsilon^\alpha}^{k\varepsilon^\alpha} \del_x(a_\lambda \del_x \varphi)(x) \dx{x} \\
 &= a_\lambda(k\varepsilon^\alpha)\del_x \varphi(k\varepsilon^\alpha) - a_\lambda((k-1)\varepsilon^\alpha)\del_x \varphi((k-1)\varepsilon^\alpha),
\end{align*}
By using that
\begin{align*}
\del_x \varphi(k\varepsilon^\alpha)-\del_x \varphi((k-1)\varepsilon^\alpha) = \pi_\varepsilon \del_x^2 \varphi(k),
\end{align*}
we can split the error terms into
\begin{align*}
\mathcal{R}(k) &= \mathcal{R}_1(k) + \mathcal{R}_2(k),
\end{align*}
where
\begin{align*}
\mathcal{R}_1(k) &=  (a_\lambda(k\varepsilon^\alpha)-a_\lambda((k-1)\varepsilon^\alpha))\left(\varepsilon^{-2\alpha}\del^+\pi_\varepsilon \varphi(k-1) - \del_x \varphi((k-1)\varepsilon^\alpha) \right), \\
\mathcal{R}_2(k) &=  a_\lambda(k\varepsilon^\alpha)\left( \varepsilon^{-2\alpha}(\del^+\pi_\varepsilon \varphi(k)-\del^+\pi_\varepsilon \varphi(k-1))-\pi_\varepsilon \del_x^2 \varphi(k) \right).
\end{align*}
For $\mathcal{R}_1(k)$ we use the Taylor expansion from Lemma~\ref{lem:Taylor} to first order ($m=1$)
\begin{align*}
\del^+\pi_\varepsilon \varphi(k-1) &=   \pi_\varepsilon \varphi(k) - \pi_\varepsilon \varphi(\cdot - \varepsilon^\alpha)(k) \\
 &= \varepsilon^\alpha \pi_\varepsilon \del_x \varphi(k) + R_1(\varphi,-\varepsilon)(k) \\
 &= \varepsilon^\alpha (\varphi(k\varepsilon^\alpha)-\varphi((k-1)\varepsilon^\alpha)) + R_1(\varphi,-\varepsilon)(k) .
\end{align*}
Hence, we can estimate the first order commutator by writing
\begin{align*}
\MoveEqLeft{\varepsilon^{-2\alpha}\del^+\pi_\varepsilon \varphi(k-1) - \del_x \varphi((k-1)\varepsilon^\alpha)} \\
&= \varepsilon^{-\alpha}\bra[\big]{\varphi(k\varepsilon^\alpha)-\varphi((k-1)\varepsilon^\alpha)} - \del_x \varphi((k-1)\varepsilon^\alpha) + \varepsilon^{-2\alpha}R_1(\varphi,-\varepsilon)(k) \\
&= -\frac{\varepsilon^{-\alpha}}{2}\int_{(k-1)\varepsilon^\alpha}^{k\varepsilon^\alpha}(k\varepsilon^\alpha - s) \del^2_x \varphi(s) \dx{s} + \varepsilon^{-2\alpha}R_1(\varphi,-\varepsilon)(k),
\end{align*}
which yields with the bound from Lemma~\ref{lem:Taylor}
\begin{align*}
|\mathcal{R}_1(k)| \lesssim  |a_\lambda(k\varepsilon^\alpha)-a_\lambda((k-1)\varepsilon^\alpha)|\int_{(k-2)\varepsilon^\alpha}^{k\varepsilon^\alpha} |\del^2_x \varphi(x)| \dx{x}.
\end{align*}
If $ \lambda = 0$, the error term $\mathcal{R}_1$ vanishes. Otherwise we have
\begin{align*}
|a_\lambda(k\varepsilon^\alpha)-a_\lambda((k-1)\varepsilon^\alpha)| \leq 
\begin{cases} \varepsilon^{\alpha} a_\lambda'((k-1)\varepsilon^\alpha) , &\ 0 < \lambda < 1, \\
\varepsilon^{\alpha} a_\lambda'(k\varepsilon^\alpha), &\ \lambda \geq 1,
\end{cases}
\end{align*}
which implies 
\begin{align*}
|\mathcal{R}_1(k)| &\lesssim \varepsilon^\alpha \int_{(k-2)\varepsilon^\alpha}^{k\varepsilon^\alpha} x^{\lambda-1} |\del^{2}_x \varphi(x)| \dx{x}.
\end{align*}
For the above estimate we used the fact that $k \geq 3$. For the second error term $\mathcal{R}_2$ we have to expand to second order
\begin{align*}
\del^+\pi_\varepsilon \varphi(k) - \del^+\pi_\varepsilon \varphi(k-1) &= \pi_\varepsilon \bra[\big]{\varphi(\cdot + \varepsilon^\alpha)}(k) - 2\pi_\varepsilon \varphi(k) + \pi_\varepsilon\bra[\big]{\varphi(\cdot - \varepsilon^\alpha)}(k) \\
&= \varepsilon^{2\alpha}\pi_\varepsilon \del^2_x \varphi(k) + R_2(\varphi,\varepsilon)(k)-R_2(\varphi,-\varepsilon)(k) .
\end{align*}
Hence, by the same argument as before, we obtain
\begin{equation*}
|\mathcal{R}_2(k)| \lesssim \varepsilon^\alpha a_\lambda(k\varepsilon^\alpha)\int_{(k-2)\varepsilon^\alpha}^{(k+1)\varepsilon^\alpha} |\del^3_x \varphi(x)| \dx{x} \lesssim \varepsilon^\alpha \int_{(k-2)\varepsilon^\alpha}^{(k+1)\varepsilon^\alpha} x^\lambda |\del^3_x \varphi(x)| \dx{x} . \qedhere
\end{equation*}
\end{proof}
\subsection{Approximate weak solutions} \label{Ss.approximate_weak_sol}

The estimates from Corollary~\ref{cor:semigroup_estimates} allow us to control powers of $\mathcal{L}_\lambda$ of solutions to equation \eqref{NP'}. The error term in the replacement Lemma however is not of this form. Hence we first need an interpolation inequality for the operator $\mathcal{L}_\lambda$.
\begin{lem} \label{lem:interpolation}
Let $\varphi \in C^2(\R_+)$. Then it holds
\begin{align*}
\norm{a_\lambda \del_x \varphi}_{\infty} \lesssim \left(\norm{a_\lambda  \varphi}_{\infty}+\norm{\del_x a_\lambda \varphi}_1 \right)^{\frac{1}{2}} \norm{\mathcal{L}_\lambda \varphi}_{\infty}^{\frac{1}{2}}.
\end{align*}
\end{lem}
\begin{proof}
Define the function 
\begin{align*}
\psi(x) = \int_0^x a_\lambda(y) \del_y \varphi(y) \dx{y}.
\end{align*}
Then we have $\del_x \psi = a_\lambda \del_x \varphi$, $\del_x^2 \psi = \mathcal{L}_\lambda \varphi$ and by standard interpolation, the inequality
\begin{align*}
\norm{a_\lambda \del_x \varphi}_{\infty} = \norm{\del_x \psi}_{\infty} \lesssim \norm{\psi}_{\infty}^{\frac{1}{2}} \norm{\del_x^2 \psi}_{\infty}^{\frac{1}{2}} = \norm{\psi}_{\infty}^{\frac{1}{2}} \norm{\cL_\lambda \varphi}_{\infty}^{\frac{1}{2}}
\end{align*}
holds. Integrating by parts we have
\begin{align*}
\psi(x) = a_\lambda(x)\varphi(x)-a_\lambda\varphi(0) - \int_0^x \del_y a_\lambda(y) \varphi(y) \dx{y},
\end{align*}
which implies $\norm{\psi}_{\infty} \lesssim \norm{a_\lambda  \varphi}_{\infty}+\norm{\del_x a_\lambda \varphi}_1 $.
\end{proof}
With this we can prove that $\mathcal{U}_\varepsilon$ is approximately a weak solution of equation \eqref{NP'}.
\begin{proof}[Proof of Proposition \ref{prop:approximate_weak_formulation}]
Let $\varphi(t,\cdot) = \mathcal{T}_\lambda(T-t)f$, $f \in C_c^\infty ((0,T) \times \R_+)$. By \eqref{weak_formulation_calculation} we have 
\begin{align*}
\MoveEqLeft{\int_0^T \int_0^\infty \mathcal{U}_\varepsilon (\del_t \varphi + \mathcal{L}_\lambda \varphi) \dx{x} \dx{t} }\\
&\!\!\!\!\!\!\!=-\int_0^\infty\! \mathcal{U}_\varepsilon(0,x) \varphi(0,x) \dx{x} 
+ \varepsilon^{-\alpha}\!\int_0^T \! \sum_{k=1}^\infty U_\varepsilon(\varepsilon^{-1}t,k) \bra*{\pi_\varepsilon \mathcal{L}_\lambda \varphi(t,k) - \varepsilon^{-1} L_\lambda \pi_\varepsilon \varphi (t,k)} \dx{t},
\end{align*}
hence we have to estimate the term
\begin{align}\label{e:error_approx}
R_\eps = \int_0^T \varepsilon^{-\alpha}\sum_{k=1}^\infty U_\varepsilon(\varepsilon^{-1}t,k) |\mathcal{R}_\varepsilon|(t,k) \dx{t},
\end{align}
where $|\mathcal{R}_\varepsilon| = |\pi_\varepsilon \mathcal{L}_\lambda \varphi - \varepsilon^{-1} L_\lambda \pi_\varepsilon \varphi|$. Let $\sigma(\varepsilon)$ be a non-negative increasing function with $\lim_{\varepsilon \to 0} \sigma(\varepsilon) = 0$ and $\theta(\varepsilon)$ a non-negative decreasing function with $\lim_{\varepsilon \to 0}\varepsilon^\alpha \theta(\varepsilon) = 0$. Then we first split the integration into two regions $[0,\sigma(\varepsilon)]$ and $(\sigma(\varepsilon),\infty)$. In the first region we use the first statement from Lemma \ref{lem:replacement}, which implies $|\mathcal{R}_\varepsilon| \lesssim \varepsilon^\alpha \norm{\mathcal{L}_\lambda \varphi}_\infty $, to estimate
\begin{align*}
\int_0^{\sigma(\varepsilon)} \varepsilon^{-\alpha}\sum_{k=1}^\infty U_\varepsilon(\varepsilon^{-1}t,k) |\mathcal{R}_\varepsilon|(t,k) \dx{t} &\lesssim \norm{\mathcal{L}_\lambda \varphi}_\infty \int_0^{\sigma(\varepsilon)} \sum_{k=1}^\infty U_\varepsilon(\varepsilon^{-1}t,k)  \dx{t} \\
&= \norm{U_{0,\varepsilon}}_1\norm{\mathcal{L}_\lambda \varphi}_\infty \sigma(\varepsilon).
\end{align*}
For the second integral, we split the summation into three regions
\begin{align*}
\sum_{1 \leq k \leq \theta(\varepsilon)} + \sum_{\theta(\varepsilon) < k \lesssim \varepsilon^{-\alpha}} + \sum_{k \gtrsim \varepsilon^{-\alpha}} = \mathrm{I}+\mathrm{II}+\mathrm{III}. 
\end{align*}
In the first region we apply the estimate~\eqref{e:uniform_decay:Linfty} from Theorem~\ref{thm:NP} for $U$ that yields $U_\varepsilon(\varepsilon^{-1}t,k) \lesssim \norm{U_{0,\varepsilon}}_1(\sigma(\varepsilon)\varepsilon^{-1})^{-\alpha}$, since $t \geq \sigma(\varepsilon)$, and the estimate for $\mathcal{R}_\varepsilon$ from above to obtain 
\begin{align*}
\mathrm{I} \lesssim \norm{U_{0,\varepsilon}}_1 \norm{\mathcal{L}_\lambda \varphi}_\infty \varepsilon^\alpha \sigma(\varepsilon)^{-\alpha} \theta(\varepsilon).
\end{align*}
For the other two sums we use the estimate from Lemma \ref{lem:replacement} that yields
\begin{align*}
|\mathcal{R}_\varepsilon| \lesssim \varepsilon^\alpha \int_{(k-2)\varepsilon^{\alpha}}^{(k+1)\varepsilon^{\alpha}} x^{\lambda - 1} |\del_x^2 \varphi| + x^\lambda |\del_x^3 \varphi| \dx{x}.
\end{align*}
For the second order term we can apply Lemma \ref{lem:repl_derivative_bounds} to conclude $x^{\lambda - 1} |\del_x^2 \varphi| \lesssim \norm{\mathcal{L}_\lambda\varphi}_\infty x^{-1}$. Next we apply Lemma \ref{lem:interpolation} to the function $\mathcal{L}_\lambda \varphi$ to obtain
\begin{align*}
\norm{a_\lambda \del_x \mathcal{L}_\lambda \varphi}_\infty \lesssim \left(\norm{a_\lambda  \mathcal{L}_\lambda \varphi}_{\infty}+\norm{\del_x a_\lambda \mathcal{L}_\lambda \varphi}_1 \right)^{\frac{1}{2}} \norm{\mathcal{L}_\lambda^2 \varphi}_{\infty}^{\frac{1}{2}}.
\end{align*}
We calculate the term on the left
\begin{align*}
a_\lambda \del_x \mathcal{L}_\lambda \varphi = a_\lambda \del_x^2 (a_\lambda \del_x \varphi) = a_\lambda (a_\lambda \del_x^3 \varphi + 2 \del_x a_\lambda \del_x^2 \varphi + \del_x^2 a_\lambda \del_x \varphi).
\end{align*}
By using Lemma \ref{lem:repl_derivative_bounds} it holds
\begin{align*}
\abs[\big]{2 \del_x a_\lambda \del_x^2 \varphi + \del_x^2 a_\lambda \del_x \varphi}(x) \lesssim \norm{\mathcal{L}_\lambda \varphi}_\infty x^{-1},
\end{align*}
which then implies the bound
\begin{align*}
|a_\lambda \del_x^3 \varphi|(x) \lesssim \left(\norm{a_\lambda  \mathcal{L}_\lambda \varphi}_{\infty}+\norm{\del_x a_\lambda \mathcal{L}_\lambda \varphi}_1 \right)^{\frac{1}{2}} \norm{\mathcal{L}_\lambda^2 \varphi}_{\infty}^{\frac{1}{2}}x^{-\lambda} + \norm{\mathcal{L}_\lambda \varphi}_\infty x^{-1-\lambda}.
\end{align*}
In particular the negative powers in the last expression are bounded for $x \gtrsim 1$, thus we can conclude 
\begin{align*}
\mathrm{III} &\lesssim \norm{U_{0,\varepsilon}}_1(\left(\norm{a_\lambda  \mathcal{L}_\lambda \varphi}_{\infty}+\norm{\del_x a_\lambda \mathcal{L}_\lambda \varphi}_1 \right)^{\frac{1}{2}} \norm{\mathcal{L}_\lambda^2 \varphi}_{\infty}^{\frac{1}{2}} + \norm{\mathcal{L}_\lambda \varphi}_\infty ) \varepsilon^\alpha  \\
&\leq \norm{U_{0,\varepsilon}}_1 C_1(\varphi,T) \varepsilon^\alpha ,
\end{align*}
where
\begin{align*}
C_1(\varphi,T) = \sup_{t \in [0,T]} (\left(\norm{a_\lambda  \mathcal{L}_\lambda \varphi}_{\infty}+\norm{\del_x a_\lambda \mathcal{L}_\lambda \varphi}_1 \right)^{\frac{1}{2}} \norm{\mathcal{L}_\lambda^2 \varphi}_{\infty}^{\frac{1}{2}} + \norm{\mathcal{L}_\lambda \varphi}_\infty ).
\end{align*}
To estimate the term II we use again the $L^\infty$ estimate for $U_\varepsilon$ and get
\begin{align*}
\mathrm{II} &\lesssim \norm{U_{0,\varepsilon}}_1C_1(\varphi,T) \varepsilon^\alpha \sigma(\varepsilon)^{-\alpha} \int_{\varepsilon^\alpha \theta(\varepsilon)}^1 x^{-1-\lambda} \dx{x} 
\lesssim \norm{U_{0,\varepsilon}}_1C_1(\varphi) \varepsilon^\alpha \sigma(\varepsilon)^{-\alpha} (\varepsilon^\alpha \theta(\varepsilon))^{-\lambda} \\
&= \norm{U_{0,\varepsilon}}_1C_1(\varphi,T) \varepsilon^{\alpha(1-\lambda)}\sigma(\varepsilon)^{-\alpha} \theta(\varepsilon)^{-\lambda}.
\end{align*}
In summary we obtain the following estimate for the full error term in~\eqref{e:error_approx}
\begin{align*}
R_\eps \lesssim \norm{U_{0,\varepsilon}}_1C_1(\varphi,T) \bra*{ \sigma(\varepsilon) + T \bra*{\varepsilon^\alpha \sigma(\varepsilon)^{-\alpha} \theta(\varepsilon) + \varepsilon^{\alpha(1-\lambda)}\sigma(\varepsilon)^{-\alpha} \theta(\varepsilon)^{-\lambda} + \varepsilon^\alpha }}.
\end{align*}
To make the right-hand side converge to zero we need to choose appropriate functions $\sigma$ and $\theta$. We make the ansatz $\sigma(\varepsilon) = \varepsilon^{a}$, $\theta(\varepsilon) = \varepsilon^{-\alpha + b}$ for some $a,b > 0$. This leads to the requirements
\begin{align*}
-\alpha a + b > 0 \qquad\text{and}\qquad  \alpha(1-a) - \lambda b > 0,
\end{align*}
which are satisfied for some $a,b$ small enough with $b > \alpha a$. Finally, we have to check that we have good control of the norms of $\mathcal{L}_\lambda \varphi$ involved in the quantity~$C_1(\varphi,T)$. This follows easily from Corollary \ref{cor:semigroup_estimates} and the explicit formula for $\varphi$, since we have
\begin{align*}
\norm{\mathcal{L}_\lambda \varphi(t,\cdot)}_\infty &= \norm{\mathcal{L}_\lambda\mathcal{T}_\lambda(T-t)f}_\infty \leq \int_0^{T-t}\norm{\mathcal{L}_\lambda \cS_\lambda(T-t-s)f(s,\cdot)}_\infty \dx{s} \\
&\leq T \norm{\mathcal{L}_\lambda f}_{L^\infty(\R_+^2)},
\end{align*}
and similarly for the other norms. Hence we conclude that
\begin{align*}
\varepsilon^{-\alpha} \int_0^T \sum_{k=1}^\infty U_\varepsilon(\varepsilon^{-1}t,k) |\mathcal{R}_\varepsilon|(t,k) \dx{t} \leq \norm{U_{0,\varepsilon}}_1 C(T,f)\varepsilon^{r},
\end{align*}
for some exponent $r > 0$, finishing the proof.
\end{proof}
\section{Convergence to self-similarity} \label{S.convergence_empirical_meas}
\subsection{Convergence of the empirical measure and moments}
Let $u$ be a solution to equation \eqref{DP}. In this subsection we apply Corollary \ref{cor:U_longtime} to the tail distribution of $u$ to extract statements regarding weak convergence and convergence of moments. To that end, let $\sigma \colon \overline{\R}_+ \to \overline{\R}_+$ and define the empirical measure associated to $u$ and $\sigma$ by
\begin{align} \label{empirical_measure_u}
\mu(t) = \sigma(t) \sum_{k = 1}^\infty u(t,k) \delta_{\sigma(t)^{-1} k}.
\end{align}
Then with the notion of weak convergence in Definition \ref{defn:weak_convergence} we have the following result.
\begin{prop}[Weak convergence of the empirical measure] \label{prop:weak_convergence_u}
Let $u$ be a solution to equation \eqref{DP} with $M_1[u] = \rho$, $\sigma \colon \overline{\R}_+ \to \overline{\R}_+$ with the property $\lim_{t \to \infty} t^{-\alpha} \sigma(t) = 1$, $\mu$ the associated empirical measure as above and $g_\lambda$ as in \eqref{profile_u}. Then for $0 \leq \lambda < 1$, $\mu(t) \rightharpoonup \rho\,g_\lambda$ with respect to $\mathcal{C}$ as $t \to \infty$, whereas for $1 \leq \lambda < 2$, $\mu(t) \rightharpoonup \rho\,g_\lambda$ with respect to $\mathcal{C}_0$.
\end{prop}
\begin{proof}
We first consider the case $0 \leq \lambda < 1$. Then for $x \geq 0$ we have
\begin{align*}
\mu\bra[\big]{t,(x,\infty)} = \sigma(t) \sum_{k = \lfloor \sigma(t) x \rfloor + 1}^\infty u(t,k) = t^{-\alpha}\sigma(t) \hat{U}(t,x(t)),
\end{align*}
with $x(t) = t^{-\alpha}\sigma(t)x$ and $\hat{U}$ as in Corollary \ref{cor:U_longtime}. In particular $\mu(t)$ is bounded in total variation. By the assumption on $\sigma$ we have $x(t) \to x$ as $t \to \infty$. Because the convergence in Corollary \ref{cor:U_longtime} is uniform and the limit is a continuous function, this implies that $\hat{U}(t,x(t)) \to \rho\,\mathcal{G}_\lambda(x)$ as $t\to \infty$. Thus we conclude
\begin{align*}
\mu\bra[\big]{t,(x,\infty)} \to \rho\,\mathcal{G}_\lambda(x) = \rho\,\int_x^\infty g_\lambda(y) \dx{y}, \qquad\text{as } t\to \infty, 
\end{align*}
and, more generally,
\begin{align*}
\mu\bra[\big]{t,(a,b]} = \hat{U}(t,a) - \hat{U}(t,b) \to \rho\bigl(\mathcal{G}_\lambda(a) - \mathcal{G}_\lambda(b)\bigr) = \rho \int_a^b g_\lambda(x) \dx{x}, \qquad\text{as } t\to \infty, 
\end{align*}
for $0 \leq a < b < \infty$. Note that in the case $a = 0$ the integral over $[a,b]$ coincides with the integral over $(a,b]$ since $\mu(t)(\{0 \}) = 0$. By linearity and tightness of the measure $\mu$ (the first moment is constant in time) we conclude that
\begin{align*}
\int_0^\infty \chi(x) \mu(t,x) \to \int_0^\infty \chi(x) \rho(x) \dx{x}, \qquad\text{as } t\to \infty , 
\end{align*}
for all functions $\chi = \sum_{k=0}^\infty \theta_k \dsOne_{I_k}$, $I_0 = [0,a_1]$, $I_k = (a_k, a_{k+1}]$, $\theta_k \leq C$, $a_k < a_{k+1}$, $a_k \to \infty$. Then by approximation (Corollary \ref{cor:U_longtime} implies that $\mu$ is uniformly bounded) the above convergence holds for bounded continuous functions, and using the bound on the first moment of $\mu$ the convergence is also extended to the class of functions $\mathcal{C}$. The argument in the case $1 \leq \lambda < 2$ works in the same way, except that the convergence on the level of characteristic functions only holds for functions with support outside of $0$, and thus we can only approximate continuous functions vanishing at $0$.
\end{proof}
\begin{cor} \label{cor:small_moment_bound}
Let $u = u(t,k)$ be a solution to equation \eqref{DP} with $M_0[u](0) = 1, M_1[u] = \rho$. Then for every $\nu \in (0,1]$ we have
\begin{align*}
\lim_{t \to \infty} t^{\alpha(1-\nu)} M_\nu[u](t) &= \rho \int_0^\infty x^\nu g_\lambda(x) \dx{x},
\end{align*}
and in the case $0 \leq \lambda < 1$ the above identity also holds for $\nu = 0$.
\end{cor}
\begin{proof}
For $\nu < 1$ this follows from Proposition \ref{prop:weak_convergence_u} with $\sigma(t) = t^{\alpha}$, applying the weak convergence to the test function $f(x) = x^\nu $, because then
\begin{align*}
 t^{\alpha(1-\nu)}M_\nu[u] = \int_0^\infty x^\nu \dx\mu(t,x) \to \rho \int_0^\infty x^\nu g_\lambda(x) \dx{x}, \qquad\text{as } t\to \infty.
\end{align*}
The statement for $\nu = 1$ follows directly from conservation of the first moment.
\end{proof}
The next goal is to show that a result similar to Corollary \ref{cor:small_moment_bound} holds for higher moments in the case $\lambda \geq 1$. The main idea is that differentiating a high moment in time gives a lower moment so we can bootstrap estimates from lower to higher moments.
\begin{lem}
Let $\lambda \geq 1$, $\nu > 1$ and $u$ be a solution to equation \eqref{DP} with $M_1[u] = \rho$ and $M_\nu[u_0] < \infty$. Then there exists an explicit positive constant $C = C(\nu,\lambda,\rho)$ such that
\begin{align*}
\lim_{t \to \infty} t^{\alpha(1-\nu)}M_\lambda[u] = C.
\end{align*}
\end{lem}
\begin{proof}
We show that if the statement holds for $\nu+\lambda - 2$ with constant $C$, then it holds for $\nu$ with constant $\frac{\nu}{\alpha} C$. To that end we use that 
\begin{align*}
\pderiv{}{t}M_\nu[u] = \sum_{k=1}^\infty k^\lambda \Delta_\N(k^\nu) u(t,k).
\end{align*}
While clear on a formal level, the integration by parts here poses a potential problem, since the boundary term contains large powers of $k$. However, since the equation is linear this can be resolved by proving the above identity for the fundamental solution to equation \eqref{DP} and using a representation similar to~\eqref{e:NP:representation}. Next, by Lemma~\ref{lem:L_asymptotics} we have
\begin{align*}
\frac{k^\lambda \Delta_\N(k^\nu)}{k^{\nu + \lambda - 2}} \to \nu(\nu - 1) .
\end{align*}
Hence, for $\varepsilon > 0$ there exists $k_0$ such that for all $k \geq k_0$ it holds
\begin{align*}
\bra[\big]{\nu(\nu - 1)-\varepsilon}k^{\nu + \lambda - 2}\leq k^\lambda \Delta_\N(k^\nu) \leq \bra[\big]{\nu(\nu - 1)+\varepsilon}k^{\nu + \lambda - 2},
\end{align*}
Therewith, we can estimate
\begin{align*}
\sum_{k=1}^\infty k^\lambda \Delta_\N(k^\nu) u(t,k) &\geq  \sum_{k=1}^{k_0} k^\lambda \Delta_\N(k^\nu) u(t,k) + (\nu(\nu - 1)-\varepsilon) \sum_{k=k_0 +1}^\infty k^{\nu + \lambda - 2}u(t,k) \\
 &= \sum_{k=1}^{k_0} (k^\lambda \Delta_\N(k^\nu)-(\nu(\nu - 1)-\varepsilon)k^{\nu + \lambda - 2}) u(t,k) \\
 &\phantom{=} + (\nu(\nu - 1)-\varepsilon)M_{\nu + \lambda -2}[u].
\end{align*}
In the finite sum we estimate $u(t,k) \leq M_0[u] \leq Ct^{-\alpha}$, while using the assumption for $M_{\nu + \lambda -2}[u]$ to conclude that 
\begin{align*}
\liminf_{t \to \infty} t^{\alpha(1-\nu)+1} \sum_{k=1}^\infty k^\lambda \Delta_\N(k^\nu) u(t,k) &\geq \liminf_{t \to \infty} t^{\alpha(1-\nu)+1} (\nu(\nu - 1)-\varepsilon)M_{\nu + \lambda -2}[u] \\
&= (\nu(\nu - 1)-\varepsilon)C.
\end{align*}
Note that the finite sum vanishes in the limit since $t^{-\alpha} \cdot t^{\alpha(1-\nu)+1} = t^{1-\alpha \nu} \to 0$, because $\alpha \geq 1, \nu > 1$. By an analogous computation we also have the upper bound
\begin{align*}
\limsup_{t \to \infty} t^{\alpha(1-\nu)+1} \sum_{k=1}^\infty k^\lambda \Delta_\N(k^\nu) u(t,k) \leq (\nu(\nu - 1)+\varepsilon)C.
\end{align*}
To compute the limit of $t^{\alpha(1-\nu)}M_\lambda[u]$, the lower and upper bound from above imply that there exists $t_0 > 0$ such that for all $t \geq t_0$ it holds
\begin{align*}
(\nu(\nu-1)C - \varepsilon)t^{\alpha(\nu-1)-1} \leq \pderiv{}{t}M_\nu[u] \leq (\nu(\nu-1)C + \varepsilon)t^{\alpha(\nu-1)-1} .
\end{align*}
Integrating these inequalities from $t_0$ to $t$ and letting $t \to \infty$ we obtain that 
\begin{align*}
\frac{(\nu(\nu-1)C - \varepsilon)}{\alpha(\nu-1)} \leq \liminf_{t \to \infty} t^{\alpha(1-\nu)}M_\lambda[u] \leq \liminf_{t \to \infty} t^{\alpha(1-\nu)}M_\lambda[u] \leq \frac{(\nu(\nu-1)C + \varepsilon)}{\alpha(\nu-1)},
\end{align*}
which shows that the limit exists and is equal to $\frac{\nu}{\alpha}C$. The full statement of the Lemma is then easily obtained by induction. Let $I_n$ for $n \geq 0$ be defined as 
\begin{align*}
I_n = \Bigl(n(2-\lambda),(n+1)(2-\lambda)\Bigr].
\end{align*}
Then let $n_0$ be the smallest $n$ such that $I_n \cap (1,\infty) \neq \emptyset$. Then for $\nu \in I_{n_0}$ with $\nu \leq 1$ there is nothing to show because Corollary \ref{cor:small_moment_bound} applies while for $\nu \in I_{n_0} \cap (1,\infty)$ we have $0 < \nu + \lambda - 2 \leq 1$ by construction. Hence by Corollary \ref{cor:small_moment_bound} the desired limit holds for $M_{\nu+\lambda-2}$, and hence by the above considerations also for $M_\nu[u]$. Thus the statement holds for all $\nu \in I_{n_0}$. Then for all $n > n_0$ we have $\nu > 1$ and $\nu+\lambda-2 \in I_{n-1}$ by construction, enabling the inductive argument. This finishes the proof.
\end{proof}
Since our main interest for the rest of this section is in the quantity $M_\lambda[u]$, we summarize our findings in the following Corollary.
\begin{cor}\label{cor:lambda_moment_estimate}
Let $u = u(t,k)$ be a solution to equation \eqref{DP} with $M_0[u](0) = 1, M_1[u] = \rho$. Then there exists a constant $C = C(\lambda,\rho)$ such that 
\begin{align*}
\lim_{t \to \infty} t^{\alpha(1-\lambda)}M_\lambda[u] = C.
\end{align*}
\end{cor}

\subsection{Self-similar behavior: Proof of Theorem \ref{thm_2}}

Recall that equation \eqref{DP} and equation \eqref{e:EDG:lambda} are linked by the time change $\tau$ defined in~\eqref{e:def:tau}.
Therewith, the function $u(\tau,k)$ defined by $u(\tau(t),k) = c_k(t)$ for $k \geq 1$ is a solution to equation \eqref{DP} if $c_k(t)$ is a solution to the system \eqref{e:EDG:lambda}. Then the asymptotic behavior of the moments of $u$ implies the following result.
\begin{prop} \label{prop:tau_asymptotics}
Let $0 \leq \lambda < 2$ and set $\beta = (3-2\lambda)^{-1}$. Then for every $\lambda \in [0,2)$ and $\rho \in (0,\infty)$ there exists $C = C(\lambda,\rho) > 0$ such that the following statements hold:
\begin{enumerate}
\item If $0 \leq \lambda < \sfrac{3}{2}$ and $c$ is a global solution to equation \eqref{e:EDG:lambda} with $M_1[c] = \rho$, then 
\begin{align*}
\lim_{t \to \infty} t^{-\frac{\beta}{\alpha}} \tau(t) = C.
\end{align*}
\item If $\lambda = \sfrac{3}{2}$ and $c$ is a global solution to equation \eqref{e:EDG:lambda} with $M_1[c] = \rho$, then for every $0 < \varepsilon < C$ it holds
\begin{align*}
\lim_{t \to \infty} \exp(-(C+\varepsilon)t) \tau(t) = 0, \ \lim_{t \to \infty} \exp(-(C-\varepsilon)t) \tau(t) = \infty.
\end{align*}
\item If $\sfrac{3}{2} < \lambda \leq 2$ and $c$ is a solution to equation \eqref{e:EDG:lambda} with blow-up time $t^*$, then it holds
\begin{align*}
\lim_{t \to t^*}(t^*-t)^{-\frac{\beta}{\alpha}}\tau(t) = C.
\end{align*}
\end{enumerate}
\end{prop}
\begin{proof}
Because $\tau(t) \to \infty$ and Corollary \ref{cor:lambda_moment_estimate} we have that for every small $\varepsilon > 0$ there exists $t_0 > 0$ such that 
\begin{align*}
(C-\varepsilon)\tau(t)^{\alpha(\lambda-1)} \leq M_\lambda[u(\tau(t),\cdot)] \leq (C+\varepsilon)\tau(t)^{\alpha(\lambda-1)},
\end{align*}
for $t \geq t_0$, where $C$ is as in Corollary \ref{cor:lambda_moment_estimate}. Using these refined bounds in the differential equation for $\tau$, we obtain 
\begin{align} \label{e:tau_differential_ineq}
(C-\varepsilon)\tau^{\alpha(\lambda-1)} \leq \dot{\tau} \leq (C+\varepsilon)\tau^{\alpha(\lambda-1)},
\end{align}
for $t \geq t_0$. Dividing by $\tau^{\alpha(\lambda-1)}$ and integrating from $t_0$ to $t$ then yields
\begin{align*}
\left(\tau(t_0)^{\frac{\alpha}{\beta}}+\frac{\alpha}{\beta}(C-\varepsilon)(t-t_0) \right)^{\frac{\beta}{\alpha}} \leq \tau(t) \leq \left(\tau(t_0)^{\frac{\alpha}{\beta}}+\frac{\alpha}{\beta}(C+\varepsilon)(t-t_0) \right)^{\frac{\beta}{\alpha}},
\end{align*}
and passing to the limit $t \to \infty$ we get
\begin{align*}
\left(\frac{\alpha}{\beta}(C-\varepsilon) \right)^{\frac{\beta}{\alpha}}\leq \liminf_{t \to \infty} t^{-\frac{\beta}{\alpha}} \tau(t) \leq \limsup_{t \to \infty} t^{-\frac{\beta}{\alpha}} \tau(t) \leq \left(\frac{\alpha}{\beta}(C+\varepsilon) \right)^{\frac{\beta}{\alpha}},
\end{align*}
which gives the desired statement with constant $\left(\frac{\alpha}{\beta}C \right)^{\frac{\beta}{\alpha}}$ after letting $\varepsilon \to 0$. In the case $\lambda = \sfrac{3}{2}$ we have $\alpha(\lambda-1) = 1$ and the inequality \eqref{e:tau_differential_ineq} gives
\begin{align*}
\tau(t_0)\exp((C-\varepsilon)t) \leq \tau(t) \leq \tau(t_0)\exp((C+\varepsilon)t),
\end{align*}
which yields the second statement. For the third statement, let $t^*$ denote the blow-up time of $\tau$. Then dividing the inequalities \eqref{e:tau_differential_ineq} by $\tau^{\alpha(\lambda-1)}$ and integrating from $t$ to $t^*$ for $t_0 < t < t^*$ we get
\begin{align*}
\left(-\frac{\alpha}{\beta}(C+\varepsilon)(t^* - t) \right)^{\frac{\beta}{\alpha}} \leq \tau(t) \leq \left(-\frac{\alpha}{\beta}(C-\varepsilon)(t^* - t) \right)^{\frac{\beta}{\alpha}},
\end{align*}
which implies the third statement.
\end{proof}
With these preparations we can prove Theorem \ref{thm_2}. Recall that for a given solution $c$ of equation \eqref{e:EDG:lambda} and a scaling function $s\colon\overline{\R}_+ \to \overline{\R}_+$ the corresponding empirical measure is given by
\begin{align*}
\mu_c(t) &= s(t) \sum_{k = 1}^\infty c_k(t) \delta_{s(t)^{-1} k}.
\end{align*}
\begin{proof}[Proof of Theorem \ref{thm_2}]
We start with the case $0 \leq \lambda < \frac{3}{2}$. Let $u(\tau(t),k) = c_k(t)$, define the function $\sigma$ by $\sigma(\tau(t)) = s(t)$ and rewrite the empirical measure in terms of $u$ and $\sigma$ as
\begin{align*}
\mu(t) = \sigma(\tau(t)) \sum_{k = 1}^\infty u(\tau(t),k) \delta_{\sigma(\tau(t))^{-1} k}.
\end{align*}
Note that $\sigma$ satisfies
\begin{align*}
\tau(t)^{-\alpha}\sigma(\tau(t)) = \tau(t)^{-\alpha}s(t) = C^{-1}\left( t^{-\frac{\beta}{\alpha}}\tau(t) \right) \to C^{-1}C = 1, 
\end{align*}
as $t \to \infty$ by Proposition \ref{prop:tau_asymptotics}. Hence Proposition \ref{prop:weak_convergence_u} applies and the desired convergence result follows. In the case $\lambda = \frac{3}{2}$ we simply take the scaling function $s(t) = \tau(t)^\alpha$, then the statement follows immediately from Proposition \ref{prop:tau_asymptotics} and Proposition \ref{prop:weak_convergence_u}. The case $\frac{3}{2} < \lambda < 2$ follows along the same lines as in the case $0 \leq \lambda < \frac{3}{2}$.
\end{proof}

\appendix 
\section{A weighted logarithmic Sobolev inequality}\label{S.lsi}

We introduce the relative entropy and Fisher information for any test function $f:\R_+\to \R_+$ with respect to the self-similar profile from~\eqref{scaling_solution} by
\begin{equation}\label{e:def:ent}
   \Ent_{\gamma_\lambda}(f) = \int f \log f \dx{\gamma_\lambda} \qquad\text{and}\qquad \cE_{\gamma_\lambda}(f,f) = \int \abs{x}^\lambda \abs{f'}^2 \dx{\gamma_\lambda} .
\end{equation}
 Then, we have the following result. 
\begin{lem}[Weighted log-Sobolev inequality]\label{lem:LSI}
 For any $\lambda\in [0,2]$ exists $C_{\LSI}(\lambda)$ such that the measure $\gamma_\lambda(\cdot) = \gamma_\lambda(1,\cdot)$ from~\eqref{scaling_solution} satisfies for all $f: \R_+ \to \R_+$ with $\cE_{\gamma_\lambda}(f,f)<\infty$ the logarithmic Sobolev inequality
 \begin{equation}
   \Ent_{\gamma_\lambda}(f^2) \leq C_{\LSI} \cE_{\gamma_\lambda}(f,f)  .
 \end{equation}
\end{lem}
\begin{proof}
We are going to apply~\cite[Theorem 3]{BartheRoberto2003}, which generalizes the result from~\cite{BobkovGoetze1999} to an applicable form for the present situation. By setting $\mu= Z_\lambda\gamma_\lambda$ and $\dx\nu(x) = Z_\lambda\abs{x}^\lambda \gamma_\lambda$, we have to show that 
\begin{align*}
B_- &= \sup_{x < 1} \ \mu([0,x])\log \left(1 + \frac{e^2}{\mu([0,x])} \right)\int_x^1 \frac{1}{\nu(x)} \dx{x}  < \infty ;\\
B_+ &= \sup_{x > 1} \ \mu([x,\infty))\log \left(1 + \frac{e^2}{\mu([x,\infty))} \right)\int_1^x \frac{1}{\nu(x)} \dx{x}  < \infty .
\end{align*}
Then, we have that $C_{\LSI}(\lambda)\leq 4 \max\set{B_-, B_+}$, where we use that the particular choice of the median in the proof of~\cite[Theorem 3]{BartheRoberto2003} does not enter the upper bound.

Let us first consider $B_+$, for which we show that asymptotically for $x \to \infty$ it is equivalent to
\begin{align*}
\mu([x,\infty)) \simeq (2-\lambda) x^{\lambda - 1} \exp(-\alpha^2 x^{2-\lambda}) \quad\text{and}\quad \int_1^x \frac{1}{\nu(x)} \dx{x} \simeq (2-\lambda)  x^{-1} \exp(\alpha^2 x^{2-\lambda}) .
\end{align*}
Therewith, the claim follows directly by plugging the above identities into the definition of $B_+$. Because both sides are strictly positive and go to $0$, respectively $\infty$, as $x \to \infty$, it suffices to show that the derivatives are asymptotically comparable by L'Hospital
\begin{align*}
\pderiv{}{x} \mu([x,\infty)) &= - \exp(-\alpha^2x^{2-\lambda}), \\
\pderiv{}{x} \int_0^x \frac{1}{\nu(x)} \dx{x} &= x^{-\lambda} \exp(\alpha^2x^{2-\lambda}),
\end{align*}
while 
\begin{align*}
\pderiv{}{x}\bra*{(2-\lambda) x^{\lambda - 1} \exp(-\alpha^2x^{2-\lambda})} &= (2-\lambda)(\lambda - 1) x^{\lambda - 2} \exp(-\alpha^2x^{2-\lambda}) -  \exp(-\alpha^2x^{2-\lambda}), \\
\pderiv{}{x} \bra*{ (2-\lambda) x^{-1} \exp(\alpha^2x^{2-\lambda})} &= - (2-\lambda) x^{-2} \exp(\alpha^2 x^{2-\lambda}) + x^{-\lambda}\exp(\alpha^2 x^{2-\lambda}),
\end{align*}
which gives the correct asymptotic, since $\lambda < 2$. Similar arguments show that for $x\ll 1$ it holds
\begin{equation}
\mu([0,x]) \simeq x \quad\text{and}\quad \int_x^1 \frac{1}{\nu(x)} \dx{x} \simeq 
\begin{cases}
  C_\lambda - \frac{x^{1-\lambda}}{1-\lambda}, &\text{for } \lambda \in [0,1) ;\\
  -\log x ,   &\text{for } \lambda = 1 ; \\
  \frac{x^{-(\lambda-1)}}{\lambda-1},  &\text{for } \lambda \in (1,2] .
\end{cases}
\end{equation}
From here, we also deduce that $B_-<\infty$ proving the claim by an application of~\cite[Theorem 3]{BartheRoberto2003}.
\end{proof}
\begin{proof}[Proof of~\eqref{e:LSI}]
By rescaling it is enough to prove~\eqref{e:LSI} for $t=1$ and we drop the subscript $1$ for now.
By setting $f(x)^2 = \frac{\dx\mu}{\rho\dx\gamma}$, we see that~\eqref{e:LSI} is equivalent to
 \[
   \Ent_{\rho\gamma_\lambda}(f^2)= \rho \Ent_{\gamma_\lambda}(f^2) \leq 4 \, C_{\LSI}\, \rho \cE_{\gamma_\lambda}(f,f) = 4 \, C_{\LSI} \cE_{\rho\gamma_\lambda}(f,f) ,
 \]
 with $\Ent_{\gamma}$ and $\cE_{\gamma}$ as in~\eqref{e:def:ent}. Hence, the logarithmic Sobolev inequality~\eqref{e:LSI} follows from Lemma~\ref{lem:LSI}.
\end{proof}

\section{Arzela-Ascoli Theorem for discontinuous functions}
\begin{prop} \label{prop:arzela_ascoli}
Let $(X,d)$ be a compact separable metric space and $f_n\colon X \to \R$ a sequence of functions with the following properties:
\begin{enumerate}
\item For all $x \in X$, $f_n(x)$ is a bounded sequence.
\item For each $\varepsilon > 0$ there exists $n_0 \in \N$ and $\delta > 0$ such that
\begin{align*}
|f_n(x) - f_n(y)| \leq \varepsilon,
\end{align*}
for all $x,y \in X$ with $d(x,y) \leq \delta$ and $n \geq n_0$.
\end{enumerate}
Then there exists a continuous function $f\colon X \to \R$ and a subsequence $n_k \to \infty$ such that $f_{n_k} \to f$ uniformly. 
\end{prop}
\begin{proof}
Let $Z \subset X$ be a countable dense subset. Then by the first property, $f_n(z)$ is a bounded sequence for all $z \in Z$, and by Bolzano-Weierstrass and a diagonal argument there exists a subsequence (not relabelled) such that $f_n(z)$ is a Cauchy sequence for all $z \in Z$. Next we show that this implies that $f_n(x)$ is a Cauchy sequence for all $x \in X$. Indeed, let $x \in X$ and $\varepsilon > 0$. By the second property there exists a $n_0$ and $\delta > 0$ with $|f_n(x) - f_n(y)| \leq \varepsilon$ for all $x,y \in X$ with $d(x,y) \leq \delta$ and $n \geq n_0$. Then by density there exists $z \in Z$ with $d(x,z) \leq \delta$ and because $f_n(z)$ is a Cauchy sequence there exists $n_1 \geq n_0$ with $|f_n(z) - f_m(z)| \leq \varepsilon$ for $n,m \geq n_1$. For such $n,m$ we have by triangle inequality
\begin{align*}
|f_n(x) - f_m(x)| &\leq |f_n(z) - f_m(z)| + |f_n(z) - f_n(x)| + |f_m(z) - f_m(x)| \leq 3\varepsilon.
\end{align*}
This shows that $f_n(x)$ is a Cauchy-sequence and thus has a limit $f(x) = \lim_{n \to \infty} f_n(x)$. The continuity of $f$ easily follows from the second property of $f_n$ by letting $n \to \infty$. To show uniform convergence let $\varepsilon > 0$ and choose $\delta, n_0$ according to the second property. Then by compactness we can cover $X$ with finitely many $\delta$-balls around some points $x_k$, $k = 1,..,N$ for some $N \in \N$. Then by convergence there exists $n_1 \geq n_0$ such that $|f_n(x_k) - f(x_k)| \leq \varepsilon$ for all $k=1,..,N$. Because by construction for every $x \in X$ there exists $k$ with $d(x,x_k) \leq \delta$ we have by triangle inequality for all $n \geq n_1$
\begin{align*}
|f_n(x) - f(x)| &\leq |f_n(x_k) - f(x_k)| + |f_n(x_k) - f_n(x)| + |f(x_k) - f(x)| \leq 3 \varepsilon,
\end{align*}
where the third term is also smaller than $\varepsilon$ by letting $n \to \infty$ in the second property. This shows uniform convergence.
\end{proof}

\section{The case \texorpdfstring{$\lambda = 0$}{λ=0}}\label{S.lambda0}

For the case $\lambda=0$ we analyse equation~\eqref{DP} directly, since the fundamental solution is explicit. The equation becomes the discrete heat equation with absorbing boundary
\begin{equation}\label{e:DP:lambda0}
\begin{cases} 
\partial_t u = \Delta_{\N} u,  &\text{on} \ \N, \\
u(t,0) = 0,  &\text{for} \ t \geq 0,\\
u(0,k) = c_k, &\text{for} \ k \geq 1.
\end{cases}
\end{equation} 
The fundamental solution $\psi$ of~\eqref{e:DP:lambda0} is obtained from the one~$\varphi$ of the whole lattice $\Z$ by reflection and  satisfies $\psi(t,k,l)=\varphi(t,k-l)-\varphi(t,k+l)$.
The fundamental solution to the discrete heat equation on $\Z$ is obtained as $\varphi(t,k) = \exp\bra{2t}I_k(2t)$ by using Fourier series and an integral representation formula for the modified Bessel's $I_k$ of the first kind (cf.~\cite[p. 376]{AbramowitzStegun}). Hence, the solution $u$ to~\eqref{e:DP:lambda0} is given by
\begin{equation}\label{e:DP:lambda0:u}
u(t,k) = \sum_{l \geq 1}\psi(t,k,l)c_l = \sum_{l \geq 1} \bra[\big]{\varphi(t,k+l)-\varphi(t,k-l)}c_l.
\end{equation}
Since $I_\nu$ is defined for real values $\nu$, we regard $x \mapsto \psi(t,x,l)$, and consequentially also $x \mapsto u(t,x)$ as a function of the continuous variable $x\in \overline\R_+$. In this form, we can prove the following result.
\begin{prop}\label{prop:lambda_0_asymptotics}
Every solution $u$ to~\eqref{e:DP:lambda0} with $M_1[c] = \rho$ satisfies
\begin{align*}
u(t,x\sqrt{t}) \simeq \frac{\rho \, x}{t\sqrt{4\pi}}  \exp\bra*{-\frac{x^2}{4}} \qquad \text{as }  t \to \infty.
\end{align*}
Here, $a(t) \simeq b(t)$ as $t \to \infty$ denotes asymptotic equivalence $\lim_{t \to \infty} \sfrac{a(t)}{b(t)} = 1$. 
\end{prop}
To prove the above result, we first derive another formula for $\psi$.
\begin{lem}
For $(t,x,l)\in \overline\R_+\times \overline\R_+ \times \N$, it holds
\begin{align}\label{e:psi_expansion}
\psi(t,x,l) = \frac{1}{t}\sum_{m=1}^l (x-l+2m-1)\varphi(t,x-l+2m-1).
\end{align}
\end{lem}
\begin{proof}
The modified Bessel functions of the first kind satisfy the relation
\begin{align*}
I_{x-1}(t)-I_{x+1}(t) = \frac{2x}{t}I_k(t),
\end{align*}
which by inserting in~\eqref{e:DP:lambda0:u} yields
\begin{align*}
\varphi(t,x-1) - \varphi(t,x+1) = \frac{x}{t}\varphi(t,k).
\end{align*}
Then, we expand $\psi(t,x,l)$ as
\begin{align*}
\varphi(t,x-l) - \varphi(t,x+l) 
 &= \sum_{m=1}^l \varphi(t,x-l+2(m-1)) - \varphi(t,x-l+2m) \\
 &= \frac{1}{t}\sum_{m=1}^l (x-l+2m-1)\varphi(t,x-l+2m-1). \qedhere
\end{align*}
\end{proof}
Next we find the scaling behavior by proving the following asymptotics for the involved Bessel functions.
\begin{lem} \label{Bessel_asymptotics}
The modified Bessel functions satisfy for every $x > 0$ the asymptotic
\begin{align*}
I_{x \sqrt{t}}(t) \simeq \frac{\exp(t)}{\sqrt{2\pi t}}\exp\bra*{-\frac{x^2}{2}} \qquad\text{as } t \to \infty .
\end{align*}
Consequently, it holds
\begin{align}
\varphi(t,x\sqrt{t}) \simeq \frac{1}{\sqrt{4\pi t}} \exp\bra*{- \frac{x^2}{4}} \qquad\text{as } t \to \infty .
\end{align}
\end{lem} 
\begin{proof}
The proof is similar to the result in~\cite[Theorem 2.13]{Kreh2012}. We have
\begin{align*}
I_{x \sqrt{t}}(t) &\simeq \frac{1}{\pi} \int_{0}^\pi \exp(t\cos(\theta))\cos(x \sqrt{t}\theta) \dx\theta \qquad\text{as } t\to\infty .
\end{align*}
By substituting $u = 2\sqrt{t}\sin\bra*{\frac{\theta}{2} }$ in the above integral we obtain
\begin{align*}
I_{x \sqrt{t}}(t) &\simeq \frac{\exp(t)}{\pi\sqrt{t}} \int_0^{2\sqrt{t}} \exp\bra*{-\frac{u^2}{2} }\frac{\cos \bra*{x 2\sqrt{t}\sin^{-1}\bra*{\frac{u}{2\sqrt{t}} } }}{\sqrt{1-\frac{u^2}{4t^2}}} \dx{u} \qquad\text{as } t\to\infty .
\end{align*}
Clearly the point-wise limit of the integrand is
\begin{align*}
\exp\bra*{-\frac{u^2}{2} }\frac{\cos \bra*{x 2\sqrt{t}\sin^{-1}\bra*{\frac{u}{2\sqrt{t}} } }}{\sqrt{1-\frac{u^2}{4t^2}}} \to \exp\bra*{-\frac{u^2}{2} }\cos(x u) \qquad\text{as } t\to\infty ,
\end{align*} 
which leads by dominated convergence to
\begin{align*}
I_{x \sqrt{t}}(t) &\simeq \frac{\exp(t)}{\pi\sqrt{t}} \int_0^{\infty} \exp\bra*{-\frac{u^2}{2} }\cos(x u) \dx{u} .
\end{align*}
By standard Fourier analysis, we obtain in the last step 
\begin{align*}
\int_0^{\infty} \exp\bra*{-\frac{u^2}{2} }\cos(x u) \dx{u} &= \frac{1}{2}\int_\R \exp\bra*{-\frac{u^2}{2} }\exp(-ix u) \dx{u} = \sqrt{\frac{\pi}{2}}\exp\bra*{-\frac{x^2}{2} } . \qedhere
\end{align*}
\end{proof}
\begin{cor}\label{cor:psi_bound}
For $(t,x,l)\in \overline\R_+\times \overline\R_+ \times \N$, it holds
\begin{align}\label{e:psi_bound}
|\psi(t,x,l)|\lesssim \frac{l}{t} \bra*{\frac{x}{\sqrt{t}} + 1 }.
\end{align}
\end{cor}
\begin{proof}
We use the identity~\eqref{e:psi_expansion} to estimate 
\begin{align*}
t|\psi(t,x,l)| &\leq l x \sup_{y \in [0,\infty)} |\varphi(t,y)| + \sum_{m=1}^l |-l+2m-1|\varphi(t,x-l+2m-1) \\
&\lesssim l \frac{x}{\sqrt{t}} + l\sum_{m \in \Z} \varphi(t,m) = l \bra*{\frac{x}{\sqrt{t}} + 1 }. \qedhere
\end{align*}
\end{proof}
With the above bound we can pass to the limit in the representation formula for $u$.
\begin{proof}[Proof of Proposition~\ref{prop:lambda_0_asymptotics}]
By the representation formula~\eqref{e:DP:lambda0:u}, we have
\begin{align*}
t\,u\bra[\big]{t,x\sqrt{t}} = \sum_{l=1}^\infty t\,\psi\bra[\big]{t,x\sqrt{t},l} c_l.
\end{align*}
Next, we note that for every fixed $k \in \Z$ we have
\begin{align*}
\sqrt{t}\,\varphi\bra*{t,x\sqrt{t}+k} \to \frac{1}{\sqrt{4\pi}}\exp\bra*{-\frac{x^2}{4}}.
\end{align*}
Indeed, the proof for $k=0$ was done in Lemma~\ref{Bessel_asymptotics} and works with minor modifications in the same way for general $k$. Thus the formula~\eqref{e:psi_expansion} for $\psi$ implies that as $t\to\infty$
\begin{align*}
 t\, \psi\bra[\big]{t,x\sqrt{t},l} &=  \sum_{m=1}^l x \sqrt{t}\,\varphi\bra[\big]{t,x\sqrt{t}-l+2m-1} + \mathcal{O}\bra[\big]{t^{-\frac{1}{2}}} \to \frac{l\,x}{\sqrt{4\pi}}\exp\bra*{-\frac{x^2}{4}}.
\end{align*}
Furthermore, by Corollary~\ref{cor:psi_bound} we have that $|t\psi(t,x\sqrt{t},l)| \lesssim l\,(x+1)$, so with dominated convergence we conclude
\begin{align*}
\lim_{t \to \infty} t\, u\bra[\big]{t,x\sqrt{t}} &= \sum_{l=1}^\infty \bra*{\lim_{t \to \infty} t\, \psi\bra[\big]{t,x\sqrt{t},l}}c_l
= \frac{\rho \, x}{\sqrt{4\pi}}\exp\bra*{-\frac{x^2}{4}}. \qedhere
\end{align*}
\end{proof}

\addtocontents{toc}{\SkipTocEntry}
\subsection*{Acknowledgement}

The authors are grateful to Emre Esenturk, Stefan Grosskinsky, Barbara Niethammer, and Juan Velazquez for a number of insightful discussions and remarks on this and related topics. The authors thank the Hausdorff Research Institute for Mathematics (Bonn) for the hospitality during the Junior Trimester Program on \emph{Kinetic Theory}. 
The authors are supported by the Deutsche Forschungsgemeinschaft (DFG, German Research Foundation) 
under Germany's Excellence Strategy EXC 2044 -- 390685587,
 \emph{Mathematics M\"unster: Dynamics--Geometry--Structure}, 
and EXC 2047 -- 390685813, the \emph{Hausdorff Center for Mathematics}, 
as well as the Collaborative Research Center 1060 -- 211504053, 
\emph{The Mathematics of Emergent Effects} at the Universität Bonn.

\bibliographystyle{abbrv}
\bibliography{bib}

\end{document}